\documentclass[a4paper, 11pt]{amsproc}

\usepackage[
    a4paper,
    left=4cm,
    right=4cm,
    top=2.5cm,
    bottom=2.5cm
]{geometry}

\usepackage{defs}
\usepackage{orcidlink}
\usepackage{calc}
\usepackage{graphicx,wrapfig,lipsum}
\usepackage[dvipsnames]{xcolor}
\usepackage{tabularray}
\usepackage{empheq} % for boxing equations
\newtcolorbox{empheqboxed}{colback=white, 
    colframe=black,
    boxrule=0.25mm,
    width=\columnwidth,
    sharpish corners,
    top=-2mm, % default value 2mm
    left=2pt,
    bottom=5pt
}
\definecolor{metablue}{HTML}{0064E0}
\definecolor{metafg}{HTML}{1C2B33}
\definecolor{metabg}{HTML}{F1F4F7}
\definecolor{metabgdeep}{HTML}{D9EFFF}
\definecolor{metagreen}{HTML}{EAFFE8}
\definecolor{metagreen}{HTML}{FCFFEE}
\definecolor{metared}{HTML}{FFEAE8}
\newtheorem{theorem}{Theorem}
\newtheorem{assumption}{Assumption}
\newtheorem{remark}{Remark}
\newtheorem{definition}{Definition}

\newtheorem{proposition}{Proposition}
\newtheorem{lemma}{Lemma}

\newcommand{\edit}[1]{\textcolor{black}{#1}}
\newcommand{\newedit}[1]{\textcolor{black}{#1}}

\renewcommand{\emptyset}{\varnothing}

\DeclareSymbolFont{extraup}{U}{zavm}{m}{n}
\DeclareMathSymbol{\varheart}{\mathalpha}{extraup}{86}
\DeclareMathSymbol{\vardiamond}{\mathalpha}{extraup}{87}
\DeclareMathSymbol{\varclub}{\mathalpha}{extraup}{84}
\DeclareMathSymbol{\vardspade}{\mathalpha}{extraup}{85}

\usepackage[framemethod=tikz,xcolor=true]{mdframed}

\newmdenv[backgroundcolor=metabg, roundcorner=0pt, skipabove=4pt, linewidth=0pt, innertopmargin=8pt]{myframe}

\newmdenv[backgroundcolor=metabgdeep, roundcorner=1pt, skipabove=4pt, linewidth=0pt, innertopmargin=6pt]{myOCP}

\newmdenv[backgroundcolor=metared, roundcorner=10pt, skipabove=7pt, linewidth=0pt, innertopmargin=7pt]{myalgo}
\newmdenv[%
    leftmargin=0.5cm,
    backgroundcolor=yellow!10,%
    roundcorner=5pt,%
    tikzsetting={draw=red, line width=2.0pt}%
    ]{SpecialText}%

    \newmdenv[
  backgroundcolor=metabg,
  roundcorner=5pt,
  skipabove=4pt,
  linewidth=0.5pt,
  innertopmargin=1pt,
  tikzsetting={draw=black, line width=0.5pt},
  frametitlefont=\bfseries,
  frametitlebackgroundcolor=blue!10,
  frametitleaboveskip=0.5ex,
  frametitlebelowskip=0.5ex
]{myOCPboxTitle}

\usepackage{cleveref}

\newcounter{ocpproblem}

\crefname{ocpproblem}{Problem}{Problems}
\Crefname{ocpproblem}{Problem}{Problems}

\title[MP-MFG]{Beyond Nominal Equilibria: Risk-Averse Multi-Population Mean-Field Games}

\author[B. Jeloka]{Bhavini Jeloka\,\orcidlink{0000-0001-7474-1731}}
\author[S. Ganguly]{Siddhartha Ganguly\,\orcidlink{0000-0003-2046-2061}}
\author[P. Tsiotras]{Panagiotis Tsiotras\,\orcidlink{0000-0001-7563-4129} \\
\\
\small (BJ and SG contributed equally)}

\thanks{%
B. Jeloka, S. Ganguly, and P. Tsiotras are with \faGroup\ Daniel Guggenheim School of Aerospace, \faUniversity\ Georgia Institute of Technology, \faMapMarker\  Atlanta, USA.
	Contact Information: (SG) \faHome\ \url{https://sites.google.com/view/siddhartha-ganguly}, \faEnvelope\ \texttt{sganguly41@gatech.edu}, (BJ) \faEnvelope\ \texttt{bjeloka3@gatech.edu}, (PT) \faHome\ \url{https://dcsl.gatech.edu/tsiotras.html}, \faEnvelope\ \texttt{tsiotras@gatech.edu}.
}

\begin{document}
\subjclass[2020]{93E20, 93E03, 49K20, 49L99, 58E25, 65K10} 
\keywords{Mean Field Games, Risk Aversion, Fixed Point Theorems, Variational Analysis.}

\maketitle

\begin{abstract}
Recent advances in mean-field games and its multi-population variants enable large-scale heterogeneous multi-agent systems to be modeled through \emph{representative agents} and their associated \emph{mean-field distributions}. However, existing approaches do not explicitly account for uncertainty in the behavior of \emph{other populations}. 
To this end, we introduce a new paradigm: \emph{risk-averse multi-population mean-field games}, where each population optimizes a worst-case expected reward over dynamically feasible ambiguity sets of mean-field flows of a subset of the other populations. 
%\panos{Mean field flows of all populations? need to be specific to avoid ambiguity}
Employing an occupation-measure formulation along with tools from set-valued analysis, we establish, under mild assumptions, several theoretical properties of the multi-population game, including the geometric properties of the ambiguity sets and the existence of a novel risk-averse multi-population mean-field equilibrium.
Further, we derive contractivity results of the fixed-point operator under entropy regularization and show that it can be utilized to learn the equilibrium.
%
%\panos{not clear what contraction refers to here}
%
Finally, we propose a risk-averse fictitious-play scheme and show that exploitability decays to zero, despite the additional nonlinearity introduced by the worst-case objective. 
We report several numerical experiments to illustrate convergence and risk-averse behavior.
\end{abstract}

\tableofcontents

\section{Introduction}\label{sec:introduction}

Mean-field theory provides a principled framework to address the curse of dimensionality in multi-agent decision making by approximating interactions in the infinite-population limit~\cite{huang2006large,lasry2007mean}.
This perspective has led to mean-field games (MFGs) for non-cooperative settings~\cite{huang2006large, cui2021approximately}, mean-field control (MFC) for cooperative settings~\cite{saldi2023partially, sen2019mean},  multi-population mean-field games (MP-MFGs)~\cite{bensoussan2018meanfieldcontrolmean, perolat2021scalingmeanfieldgames} and mean-field team games (MFTGs)~\cite{guan2024zero, jeloka2025learninglargescalecompetitiveteam}, with the latter two capturing heterogeneous and mixed cooperative–competitive interactions, respectively.

The presence of multiple sub-populations with potentially conflicting objectives imposes additional complexity in MP-MFG and MFTG problems.
Despite \newedit{the} recent progress in computing equilibria for MP-MFGs and MFTGs~\cite{perolat2021scalingmeanfieldgames, guan2024zero}, robustness in these settings remains relatively underexplored. 
In practice, populations face uncertainty arising from environmental noise, unforeseen events, and ambiguity about the behavior and distributions of other populations. 
Such uncertainty motivates risk-averse decision-making, in which agents trade off expected performance against exposure to unfavorable outcomes.
To address this gap, we introduce a framework for \embf{risk-averse multi-population mean-field games}  (RA-MP-MFGs) and study the associated solution concepts along with their theoretical guarantees.

\textbf{A motivating example.} The problem of stock trading is inherently subject to the curse of multi-agency, as aggregate market trends and price discovery emerge from the collective decisions of a vast population of interacting agents. 
To motivate the discussion,
consider a stylized market with algorithmic (AI) traders and human investors trading a single asset.
While both populations seek to maximize expected returns, human behavior is affected by cognitive biases such as trend-following, herding, and panic selling, as well as risk aversion toward stochastic price fluctuations.
Since AI traders cannot directly observe or predict the realized human population distribution, which is generally infeasible to observe or compute, they must account for deviations from its nominal trajectory, motivating strategic risk aversion to uncertainty in human behavior.

Our central idea is to replace each risk-averse population’s nominal mean-field trajectory with a dynamically feasible ambiguity set of plausible trajectories and optimize the worst-case expected reward over this set. 
This yields a coupled robust optimization problem in which each population accounts for uncertainty in the behavior of the others, leading to a new \emph{risk-averse equilibrium} concept.

\textbf{Contributions.} Our primary advancements are the following:

\begin{enumerate}[label=\textup{(Cont-\alph*)}, leftmargin=*, widest=b, align=left]

\item We introduce a framework for RA-MP-MFGs and study the associated solution concepts and their theoretical guarantees. Specifically, using tools from set-valued theory and variational analysis, we establish the existence of a risk-averse multi-population mean-field equilibrium through an \emph{occupation-measure} formulation and fixed-point arguments;

\item we show that entropy-regularized RA-MP-MFGs approximate the original equilibrium under suitable conditions via a novel occupation-measure-dependent norm; and 

\item we establish convergence of \emph{fictitious play} under mild assumptions on the reward function, despite the nonlinear worst-case objective.
Numerical experiments, including a real-world, data-driven stock-trading example, demonstrate algorithmic convergence and characterize the resulting risk-averse policies under distributional uncertainty.
\end{enumerate}
%%%%%%%%
We keep our technical results at an \emph{informal level} in the main body of the article; detailed technical statements along with their proofs can be found in Appendix \ref{sec:reach:set:appen}--\ref{thm:fp_FH:appen}.\footnote{A technical `tl;dr' is in Appendix \ref{sec:tech:tldr:appen}; also see Figure \ref{fig:theoretical-roadmap}, which summarizes our theoretical results.} Moreover, additional numerical details are in Appendix \ref{sec:add:num:det:appen}.

\textbf{Related Work.} Multi-population MFGs extend the standard MFG framework to heterogeneous populations with distinct dynamics, rewards, and policies, and have been studied under both decentralized open-loop policies~\cite{perolat2021scalingmeanfieldgames, lauriere2022scalabledeepreinforcementlearning} and feedback information structures~\cite{shao2024reinforcementlearningfinitespace, subramanian2023meanfieldgamesteams, subramanian2022multitypemeanfield}. 
Existing formulations primarily optimize expected rewards and generally do not explicitly account for uncertainty or adversarial deviations in the population trajectories generated by other populations. 
Zero-sum multi-team formulations~\cite{guan2024zero, jeloka2025learninglargescalecompetitiveteam} capture strategic interactions but rely on additional structural assumptions, such as local-state-independent rewards or mean-field-dependent policies. 
In contrast, we consider decentralized policies that remain robust to deviations of other populations from their nominal trajectories.
This perspective is natural in decentralized settings, where each population has no control over the dynamics or the actions of other populations.

Risk-sensitive and robust MFGs have also been studied in continuous- and discrete-time settings~\cite{moon2014linear, tembine2013risk, tembine2015risk, perrin2020fictitiousplaymeanfield, saldi2020approximate, cheng2023risk, bonnans2021discrete, yang2025structuralreinforcementlearningheterogeneous}, but existing approaches typically model environmental or model uncertainty~\cite{liang2026meanfieldgamesmodeluncertainty}, uncertainty over initial distributions~\cite{jeloka2026robustmeanfieldgamesrisk}, entropic uncertainty~\cite{delarue2026robustmeanfieldgamesentropybased} or common noise affecting the population collectively~\cite{carmona2015meanfieldgamescommon, cardaliaguet2025meanfieldgamescommon}. 
Our framework, instead, focuses on the multi-population regime, where we analyze \emph{strategic uncertainty} arising from deviations in the mean-field trajectories of other populations while retaining the underlying population dynamics. 
%
%In contrast to prior discrete-time formulations that often rely on general Polish spaces or additional structural assumptions, we consider finite state--action spaces with explicit dynamics, enabling a tractable MDP and occupation-measure formulation.

% ==================================
\section{Preliminaries: MP-MFGs}\label{sec:prob:form}
% ==================================

\textbf{Notation.} Let \(\mathbb{N}\) be the set of of natural numbers and for any \(d\in \mathbb{N}\), \(\mathbb{R}^d\) is the \(d\)-dimensional Euclidean space which is equipped with the standard norm \(\Real[d] \ni x \mapsto \norm{x}\) and inner product \(\inprod{\cdot}{\cdot}\).  For every fixed \(N\in\Nz\), we write \([N] \Let \{1,\ldots,N\}.\)
For a set \(A\) and an element \(e\), we denote the indicator of \(A\) by \(\mathds{1}_{A}(e)\), where
\begin{align*}
\mathds{1}_{A}(e) \Let \begin{cases} 1, & e\in A,\\ 0, & e\notin A.
\end{cases}
\end{align*}
In particular, \(\mathds{1}_{\{a\}}(b)=1\) if \(a=b\), and \(\mathds{1}_{\{a\}}(b)=0\) otherwise.
Let \(E\) be a finite and nonempty set. We denote the set of probability measures on \(E\) by
\begin{align*}
\P(E) \Let \aset[\Big]{\mu:E\to \lcrc{0}{1} \suchthat \sum_{e\in E}\mu(e)=1}.    
\end{align*}
Since \(E\) is finite, we identify every probability measure \(\mu\in\P(E)\) with its probability-mass vector \((\mu(e))_{e\in E}\). For \(v\in\mathbb{R}^{E}\), we denote its \(\ell^{1}\)-norm by \(\|v\|_{1}:=\sum_{e\in E}|v(e)|.\) For two measures \(\mu\) and \(\nu\) on the same measurable space, we write \(\mu \ll \nu\) if \(\mu\) is absolutely continuous with respect to \(\nu\), i.e., if \(\nu(A)=0\) implies \(\mu(A)=0\) for all Borel subsets \(A\); see \cite{book1,ref:FolReal-99} for more details.

\textbf{Setting and MP-MFGs.} We briefly review MP-MFGs and subsequently establish our risk-averse framework for heterogeneous large population multi-team games.
 We consider a finite-horizon, discrete-time, general-sum stochastic game involving \(M \in \mathbb{N}\) large populations, with each population representing a distinct group of agents, thereby capturing macroscopic heterogeneity. We have the following problem data:

\begin{enumerate}[label=\textup{(D-\alph*)}, leftmargin=*, widest=b, align=left]

\item We use \(j\in[M]\) to index the populations. 
The state and action spaces of the game are finite, nonempty sets \(\X\) and \(\U\), respectively, both equipped with their discrete \(\sigma\)-algebras. 

\item For each population \(j\in[M]\), the state of the \textit{representative agent} at time index $t =0,1,\ldots, T$, will be denoted by $x_t^j$, with initial state distributed according to
\(X_0^j\sim\mu_0^j\in\P(\X)\). Henceforth, we use the term ``agent'' to mean a representative agent.

\item At each time \(t\in\{0,\ldots,T-1\}\), the agent observes its current state \(X_t^j=x_t^j\in\X\) and selects an action \(U_t^j=u_t^j\in\U\) according to a policy \(\pi_t^j:\X \times \U \to [0,1]\).
Here, \(\pi_t^j(\cdot|x)\in\P(\U)\) for every \(x\in\X\). We let \(\Pi_t\) denote set of randomized Markov decision rules at time \(t\); to wit
\begin{align*}
\Pi_t \Let \aset[\Big]{\pi_t:\X\times\U\to[0,1]\suchthat \sum_{u\in\U}\pi_t(u|x)=1\text{ for every }x\in\X}.
\end{align*}
An element of this set assigns an action distribution \(\pi_t(\cdot|x)\in\P(\U)\) to each state \(x\in\X\). Thus \(\Pi_t\) can be identified with \(\prod_{x\in\X}\P(\U)\).

\item Let \(\Pi^j \Let \prod_{t=0}^{T-1}\Pi_t\) be the {set} of policies for agent $j$, and let $\Pi\Let\prod_{j=1}^M\Pi^j$ be the set of policies of all representative agents.
Using the action $u_t^j$ selected according to the policy 
\(\pi_t^j\), agent $j$ then transitions to \(X_{t+1}^j\) according to the controlled transition kernel $f_t^j(\cdot\mid x_t^j,u_t^j)\in\P(\X)$.
As in many classical MFG formulations~\cite{lasry2007mean, perolat2021scalingmeanfieldgames}, 
we assume throughout that the transition kernels are independent of the multi-population distribution.

\end{enumerate}

\textbf{Mean-field.} Let \(\mu_t^j \in \P(\X)\) 
denote the distribution of population $j$ at time $t$ and let $\mu_t \Let (\mu_1^1,\dots,\mu_t^N) \in \P(\X)^M$ denote the collection of the distributions of all populations at time $t$.
We represent the \textit{mean-field flow} of population \(j\) from $t=\{0,\ldots, T\}$ by \(\mu^j\Let(\mu_t^j)_{t=0}^{T}\in \M^j \subseteq \P(\X)^{T+1}\) where $\M^j$ is the set of all mean-field flows of population $j$.
Further, \(\mu\Let(\mu^1,\ldots,\mu^M)\) denotes the collection of the mean-field flows of all populations.
Let $\M\Let\prod_{j=1}^M\M^j$ denote the set of all such flows. 
For a multi-population distribution \(\mu_t\), the collection of distributions of all populations other than $j$ is $\mu_t^{-j} \Let \bigl( \mu_t^1,\ldots,\mu_t^{j-1},\mu_t^{j+1},\ldots,\mu_t^M \bigr)$.
Hence,  $\mu_t = (\mu_t^j,\mu_t^{-j})$.

Since all agents in population $j$ follow the same policy 
\(\pi^j\) {(i.e., the policy of the representative agent)}, given a fixed initial distribution \(\mu_0^j\in\P(\X)\), the evolution of the distribution of the $j$-th population
is deterministic and can be computed recursively via 
\begin{equation} \label{eqn:mf-dynamics-local}
\mu_{t+1}^j(x')=\sum_{x\in\X}\sum_{u\in\U}f_t^j(x'\mid x,u)\pi_t^j(u\mid x)\mu_t^j(x)
\end{equation}
for all \(x\in\X\) and \(t\in\{0,\ldots,T-1\}.\) We sometime use by \(\nxtmf\) instead of \(\mu_{t+1}^j(x')\) to denote the next state.
The mean-field dynamics \eqref{eqn:mf-dynamics-local} can also be expressed in a compact matrix form as $\mu^j_{t+1} = \mu^j_{t} F^j_t(\pi^j_t)$.

For each \(j\in[M]\), let $\B_{\prop}^j:\Pi^j\to\M^j$ denote the propagation operator associated with~\eqref{eqn:mf-dynamics-local} and the fixed initial distribution \(\mu_0^j\). Accordingly, $\B_{\prop}:\Pi\to\M$ is defined by $\B_{\prop}(\pi)
\Let
\bigl(
\B_{\prop}^1(\pi^1),\ldots,
\B_{\prop}^M(\pi^M)
\bigr)$ where $
\pi\Let(\pi^1,\ldots,\pi^M)$ is the joint policy profile, and propagates all distributions. For a fixed mean-field distribution \(\mu_t^j\), varying the local policy \(\pi_t^j\) induces a set of distributions that can be reached at the next time step. 
This motivates the following definition.
\begin{myframe}
\begin{definition}[\highlight{Reachable set}]
\label{def:reachable-set}
Fix \(j\in[M]\), \(t\in\{0,\ldots,T-1\}\), and \(\mu_t^j\in\P(\X)\).
The \textbf{one-step reachable set} of population \(j\) from \(\mu_t^j\)  is defined by $\reachset \Let \{\nxtmf\in\P(\X)~\big\vert~ \exists\,\pi_t^j\in\Pi_t\text{ such that }\mu^+=\mu_t^jF_t^j(\pi_t^j)\}$.
\end{definition}
\end{myframe}
By definition, \(\reachset\) is a polytope contained in \(\P(\X)\); see Appendix \ref{sec:reach:set:appen} for a proof. For a multi-population distribution \(\mu_t\Let(\mu_t^1,\ldots,\mu_t^M)\in\P(\X)^M\), we denote the vector of distributions of all populations other than \(j\) by
\begin{align*}
\mu_t^{-j}\Let \bigl( \mu_t^1,\ldots,\mu_t^{j-1},\mu_t^{j+1},\ldots,\mu_t^M \bigr).
\end{align*}
For each \(j\in[M]\) and \(t\in\{0,\ldots,T-1\}\), we further define
\begin{align}
\R_t \bigl(\mu_t^{-j} \bigr) \Let \prod_{k\neq j}\R_t^k \bigl(\mu_t^k\bigr),\nn
\end{align}
which is the joint one-step reachable set of all populations other than \(j\).

\textbf{Reward.} The reward received by the representative agent of population $j$ at time $t$ is given by $r_t^j:\X\times\P(\X)^M\to\mathbb{R}$. 
The reward depends on the distribution flows of \emph{all} populations, and hence this is the main mechanism by which the decisions (and hence the resulting flows) of the populations are coupled. 
%%%%%%%%%%%%
\begin{assumption}
\label{assmpt:lipschitiz-model}
For each \(j\in[M]\), there exists \(L_r^j>0\) such that
$|r_t^j(x,\nu)-r_t^j(x,\nu')|\leq L_r^j\dtv{\nu,\nu'}$ for all \(t\in\{0,\ldots,T\}\), \(x\in\X\), and \(\nu,\nu'\in\P(\X)^M\). 
\end{assumption}
Lipschitz continuity of the rewards is standard in the MFG literature~\cite{huang2006large, gu2021mean, cui2023learning}.

Occupation measures provide a natural representation of policies and mean-field flows, with a linear structure that facilitates the characterization of dynamically feasible ambiguity sets and the analysis of the game~\cite{yu2026occupation, hu2025mf}.

\textbf{Occupation measure.} For each \(j\in[M]\), let \(\Occ^j\subset \M^j\times\bigl(\mathbb{R}_+^{\X\times\U}\bigr)^T\) 
denote the set of occupation measures \(o^j\Let(\xi^j,\eta^j)\), where \(\xi^j\Let\{\xi_t^j\}_{t=0}^T\) and \(\eta^j\Let\{\eta_t^j\}_{t=0}^{T-1},\) satisfying
\begin{align}
\xi_0^j=\mu_0^j,\quad \sum_{u\in\U}\eta_t^j(x,u)=\xi_t^j(x) \label{occ:constraint:1} \\
\xi_{t+1}^j(y)=\sum_{x\in\X}\sum_{u\in\U}\eta_t^j(x,u)f_t^j(y|x,u) \label{occ:constraint:2}
\end{align}
for all \(x,y\in\X\) and \(t=0,\ldots,T-1.\) 
Here, \(\xi_t^j\) and \(\eta_t^j\) represent, respectively, the state distribution and the joint state--action distribution of the representative agent of population \(j\) at time \(t\). Under a Markov policy \(\pi^j\), we have \(\xi_t^j(x)=\probmeas^{\pi^j}(X_t^j=x)\) for \(t=0,\ldots,T\), and \(\eta_t^j(x,u)=\probmeas^{\pi^j}(X_t^j=x,U_t^j=u)=\xi_t^j(x)\pi_t^j(u\mid x)\) for \(t=0,\ldots,T-1\), where \(\probmeas^{\pi^j}\) denotes the probability law induced by the initial distribution \(\mu_0^j\), the transition kernels \(f_t^j\), and the policy \(\pi^j\). 
We denote the joint occupation-measure space by \(\Occ\Let\prod_{j=1}^M\Occ^j.\) For \(o\Let(o^1,\ldots,o^M)\in\Occ\), we denote the corresponding multi-population state-flow by \(\xi(o)\Let(\xi^1,\ldots,\xi^M).\)

\textbf{Objective.} Given a multi-population flow \(\mu\in\M\), the cumulative performance of the representative agent of population $j$ under policy \(\pi^j\) is
\begin{align}\label{eqn:performance-team-infinite}
    \objective_j \big(\pi^j, \mu\big) =
    \E^{\pi^j}
    \Bigg[\sum_{t=0}^{T} r^j_t(x^j_t, \mu_t) \Bigg].
\end{align}
We use the correspondence \(\B_\opt:\M\rightrightarrows\Pi\) to denote the multi-population best-response correspondence, defined by\footnote{In variational analysis, economics, and game theory, the word correspondence \cite[Chapter 17]{ref:AliBor-06} is commonly used for a set-valued objects, for example \(F:\mathsf{A} \rightrightarrows \mathsf{B}\). We will occasionally use the term `mapping' for such objects in the sequel.}
\begin{align*}
\B_\opt(\mu)\Let\prod_{j=1}^{M}\argmax_{\pi^j\in\Pi^j}J^j(\pi^j,\mu).
\end{align*}

Using the operators $\B_\prop$ and $\B_\opt$, we define the multi-population mean-field equilibrium (MFE).
\begin{myframe}
\begin{definition}[\highlight{MP-MFE~\cite{huang2006large}}]\label{def:mfe}
  The MP-MFE is a consistent pair \((\pi^\star,\mu^\star)\in\Pi\times\M\) such that
\begin{align*}
\pi^\star\in\B_\opt(\mu^\star) \quad\text{and}\quad \mu^\star=\B_\prop(\pi^\star).
\end{align*}
\end{definition}
\end{myframe}
Under Assumption~\ref{assmpt:lipschitiz-model}, the existence of such consistent pair can be established through Kakutani's fixed point argument~\cite{cui2021approximately}. 
Lastly, we define exploitability in order to characterize a policy relative to the optimal, reward maximizing policy. 
%%%
\begin{myframe}
\begin{definition}[\highlight{Exploitability~\cite{perolat2021scalingmeanfieldgames}}]
\label{defn:exploitability:normal}
For a policy profile \(\pi\in\Pi\), let \(\mu^\pi\Let\B_\prop(\pi)\) denote the induced multi-population mean-field flow. The exploitability of \(\pi\) measures the distance to an equilibrium and is defined a
\begin{align*}
\phi(\pi)\Let\sum_{j=1}^{M}\phi^j(\pi),
\quad
\phi^j(\pi)\Let\max_{\pi'^j\in\Pi^j}\objective_j(\pi'^j,\mu^\pi)-\objective_j(\pi^j,\mu^\pi)
\end{align*}
for each \(j\in[M]\).
A policy profile \(\pi\in\Pi\) is a MP-MFE if and only if \(\phi(\pi)=0\).
\end{definition}
\end{myframe}

Although $\objective^j(\pi^j,\mu)$ depends on all population flows, population/agent $j$ controls only its own policy and therefore it must rely on nominal beliefs about the other team population flows. 
Deviations from these beliefs create strategic uncertainty, motivating policies that are robust to unfavorable mean-field trajectories of the other populations.
This motivates the incorporation of risk aversion with respect to such uncertainties in population distributions and adapting the current objective to increase robustness.

\section{Risk-Averse MP-MFGs}\label{subsec:risk:aversion}
In this section, we formulate the Risk-Averse Multi-Population Mean-Field Game (RA-MP-MFG) that adapts the standard MP-MFG objective~\cite{perolat2021scalingmeanfieldgames}
to account for uncertainties in population distributions.
In particular, we analyze the worst-case return that population $j$ can obtain under distributional uncertainty with respect to all other population types.
To this end, we formulate constraints imposed by other populations.
Fix \(\eps \in [0,1]\), and define the \textbf{dynamically feasible ambiguity set} \(\mathcal{W}_{j}(\mu,\eps)\) by 
\begin{align}\label{eq:ambiguity:set}
\hspace{-2mm}\mathcal{W}_j(\mu,\eps)\Let  \left\{
    \begin{gathered}
    \{\hat\mu_t^{k}\}_{t=0}^T, \,
    \forall k\in [M],\ k\neq j
    \end{gathered}
   \;\middle\vert\;
    \begin{array}{@{}l@{}}
    \hat\mu_0^{k}=\mu_0^{k},\,
    \hat\mu_{t+1}^{k}\in \mathcal{R}_t^k(\hat\mu_t^k), \\[2pt]
    \dtv{\mu_t^k,\hat\mu_t^k}\le \eps \hspace{0mm}
    \end{array}
    \right\}.
\end{align}
which contains all reachable flows that are not ``far away'' from their nominal values and \(\hat{\mu}^{k}_{t+1} \in \mathcal{R}_t^k(\hat\mu^k_t), k\neq j\) ensures that the trajectories \(\{\hat{\mu}^k_t\}_{t=0}^{T}\) are {dynamically feasible}.

\begin{remark}
The set $\mathcal{W}_j(\mu,\eps)$ in~\eqref{eq:ambiguity:set} assumes that population $j$ is risk-averse to every population $k\neq j$. 
This can be generalized to the setting where $j$ is risk-averse only to a subset \(\mathcal{K}_j\subseteq[M]\setminus{j}\) of populations, by defining the ambiguity set only over the corresponding mean-field trajectories. 
Under Assumption~\ref{assmpt:lipschitiz-model}, the existence and entropy-approximation results extend directly to this generalized setting, while the contraction and fictitious-play guarantees can be extended by strengthening the contraction condition to account for the Lipschitz dependence on the remaining nominal populations.
The multi-population chasing example in Section~\ref{sec:numexp} ahead empirically demonstrates this extension.
\end{remark}

%%%%%%%%%%%%%%%%%%%%
\begin{myOCP}
\begin{proposition}[\highlight{Informal}]
\label{prop:ambiguity:set}
Fix \(j\in[M]\), \(\mu \in\M\), and let \(\eps \in [0,1]\). 
Then, \(\ambset(\mu,\eps)\) is a nonempty compact polytope. 
\end{proposition}
\end{myOCP}
%%%%%%%%%%%%%%%%%%%%

\begin{remark}[Policy--occupation-measure equivalence]\label{rem:policy:occ:equivalence}
\edit{The occupation-measure representation in \eqref{occ:constraint:1}--\eqref{occ:constraint:2} is equivalent to the Markov-policy representation \cite{ref:YG:AA:LH:KHJ:Dist:Reach:MDP:TAC}.
Indeed, let \(\pi^j\in\Pi^j\), let \(\xi^j=\B_\prop^j(\pi^j)\) be the induced state-flow, and define
\begin{equation}\label{eq:policy:to:occupation:old}
\eta_t^j(x,u)\Let \xi_t^j(x)\pi_t^j(u|x)\quad \text{for }x\in\X,\ u\in\U,\ t=0,\ldots,T-1.
\end{equation}
Then, \(o^j=(\xi^j,\eta^j)\in\Occ^j\). Conversely, every \(o^j=(\xi^j,\eta^j)\in\Occ^j\) induces a Markov policy by
\begin{equation}\label{eq:markov:policy:old}
\pi_t^j(u|x)\Let\frac{\eta_t^j(x,u)}{\xi_t^j(x)}
\quad\text{if }\,\xi_t^j(x)>0,
\end{equation}
and by choosing \(\pi_t^j(\cdot|x)\) arbitrarily in \(\P(\U)\) whenever \(\xi_t^j(x)=0\). Indeed, if \(\xi_t^j(x)=0\), then \eqref{occ:constraint:1} and the nonnegativity of \(\eta_t^j\) imply that \(\eta_t^j(x,u)=0\) for every \(u\in\U\).
The policy constructed in \eqref{eq:markov:policy:old} generates the state-flow \(\xi^j\) and the state-action occupation flow \(\eta^j\).
Consequently, \(\xi^j\in\M^j\) for every \(o^j\in\Occ^j\), and hence \(\xi(o)\in\M\) for every \(o\in\Occ\).}
\end{remark}

From the point of view of population $j$, we are interested in evaluating the performance of a policy $\pi^j$ under the \emph{risk-averse payoff} defined below.
%%%%%%%%%%%%%%%%%%%%
\begin{myframe}
\begin{definition}[\highlight{Risk-averse payoff}]\label{def:RA:payoff}
Fix \(\mu\in\M\) and $\eps \in [0,1]$. 
For each $j\in[M]$ and $\pi^j\in\Pi^j$, the \textbf{risk-averse payoff} of population $j$ is defined by 
\begin{equation}\label{J-bar-inf}
\hspace{-1mm}\bar{\objective}_j(\pi^j;\mu) \Let \hspace{-2mm} \inf_{\widehat\mu^{-j}\in\ambset(\mu,\eps)}\E^{\pi^j}\left[\sum_{t=0}^{T} r_t^j(x_t^j,\mu_t^j,\widehat\mu_t^{-j})\right]\hspace{-1mm},
\end{equation}
where \(x_0^j\sim\mu_0^j\), \(u_t^j\sim\pi_t^j(\cdot|x_t^j)\), and \(x_{t+1}^j\sim f_t^j(\cdot|x_t^j,u_t^j)\). 
\end{definition}
\end{myframe}
%%%%%%%%%%%%%%%%%%%%
We further show in Appendix~\ref{sec:existence:normal:appen} that~\eqref{J-bar-inf} can be reformulated in terms of the occupation measure as
\begin{align}\label{eq:occ:phi}
\Phi_j(\tilde{o}^j;o)\Let \inf_{\widehat\mu^{-j}\in \ambset(\xi(o),\eps)}\sum_{t,x}\tilde \xi_t^j(x)r_t^j\bigl(x,\xi_t^j,\widehat\mu_t^{-j}\bigr)  
\end{align} 
where \(\tilde{o}^j \Let (\tilde{\xi}^j,\tilde{\eta}^j)\in\Occ^j\).

Given a nominal flow $\mu$, population $j$ chooses $\pi^j$ to maximize $\bar{\objective}_j(\pi^j;\mu)$ against dynamically feasible 
adversarial deviations $\hat{\mu}_t=(\mu_t^j,\hat{\mu}_t^{-j})$ for all \(t\in\{0,\ldots,T\}\)
of the other populations' mean-field trajectories.
Define \(\briskopt:\M\to\Pi\) as the \emph{risk-averse best-response correspondence}, i.e., for each \(\mu\in\M\), \((\pi^1,\ldots,\pi^M)\in\briskopt(\mu)\), where, for every \(j\in[M]\), 
\[\pi^j \in \argmax_{\pi^j\in\Pi^j}\bar{\objective}_j(\pi^j;\mu).
\]

\begin{myframe}
\begin{definition}[\highlight{Risk-averse exploitibility}]
\label{def-ra-explt}
    For any \(\eps \in [0,1]\) and $\pi\in\Pi$, let \(\mu^\pi=\B_\prop(\pi)\). 
The risk-averse exploitability is defined as $\phi_{\RA}(\pi)=\sum_{j=1}^M\phi_{\RA}^j(\pi)$, where $\phi_{\RA}^j(\pi)=\sup_{\tilde\pi^j\in\Pi^j}
\bar{\objective}_j(\tilde\pi^j;\mu^\pi) - \bar{\objective}_j(\pi^j;\mu^\pi),~\forall j\in[M]$.
\end{definition}    
\end{myframe}

The quantity \(\phi_{\RA}^j(\pi)\) is the largest
improvement population \(j\) can obtain by unilaterally replacing
its current policy {$\pi$} with a risk-averse best response, while \(\mu^\pi\) is held fixed. In particular,
\(\phi_{\RA}^j(\pi)\geq 0\), and it vanishes precisely when
\(\pi^j\) is a risk-averse best response to \(\mu^\pi\).

\begin{myframe}
\begin{definition}[\highlight{Risk-averse equlibrium}]\label{def:RA:MP:MFE}
Fix \(\eps \in [0,1]\) and let \(\pi^\star \in \Pi\) and \(\mu^\star \in \M\).
A pair \((\pi^\star,\mu^\star) \in \Pi \times \M\) is a \embf{Risk-Averse Multi-Population Mean-Field Equilibrium} (RA-MP-MFE) if the following two conditions hold: i) $\mu^{\star}=\B_\prop(\pi^{\star})\quad \text{for all}\, j\in[M]$; and ii) $\pi^{\star}\in\briskopt(\mu^{\star})$.
\end{definition}
\end{myframe}

Thus, \(\phi_{\RA}(\pi)=0\) iff
\((\pi,\mu^\pi)\) is an RA-MP-MFE.

\begin{myOCP}
\begin{theorem}[\highlight{Existence of equilibria (informal)}]\label{thm:existence:RA:MP:MFE}
Fix \(\mu_0^j\in\P(\X)\) for all \(j\in[M]\). 
Then, under Assumption~\ref{assmpt:lipschitiz-model}, for every fixed \(\eps \in [0,1]\), 
there exists a RA-MP-MFE \((\pi^\star,\mu^\star)\in\Pi\times\M\). 
\end{theorem} 
\end{myOCP}
A proof of Theorem~\ref {thm:existence:RA:MP:MFE} is given in Appendix \ref{sec:existence:normal:appen} and it relies on techniques from variational analysis and fixed-point theory \cite{ref:AliBor-06}.  
However, note that the operator \(\B_\prop\circ\briskopt:\M\to\M\) is generally \emph{not contractive} (this is usual, see \cite{cui2021approximately}); a statement and proof is in  Appendix \ref{sec:existence:normal:appen}. This naturally brings us to the next section.

% ===================================
% ===== Entropy Regularization ======
% ===================================

\section{Entropy Regularization}\label{subsec:Entropy:Regularization}

Non-contractivity implies that the risk-averse best-response dynamics may not support standard fixed-point iteration schemes.
This motivates the introduction of an \emph{entropy-regularized} variant of the risk-averse equilibrium. 
This concept models \emph{boundedly rational} decision-making \cite{mazumdar2025tractable} through an entropy reward {and also provides} a smoother surrogate. 
We first define the entropy-regularized risk-averse payoff.

%%%%%%%%%%%%%%%%%%%%
\begin{myframe}
\begin{definition}[\highlight{Entropy-regularized payoff}]\label{def:entropic:RA:payoff}
Fix  \(\mu\in \M\). 
For each population \(j\in[M]\), policy \(\pi^j\in\Pi^j\), \(\eps\in [0, 1]\), and entropy-regularization parameter \(\entparam>0\), the entropy-regularized risk-averse payoff of population \(j\), \(\bar{\objective}_{j,\entparam}(\pi^j;\mu)\), is defined by
\begin{align}
\bar{\objective}_{j,\entparam}(\pi^j;\mu) \Let
\inf_{\widehat\mu^{-j}\in\ambset(\mu,\eps)}
\E^{\pi^j}\left[\sum_{t=0}^{T} r_t^j\bigl(x_t^j,\mu_t^j,\widehat\mu_t^{-j}\bigr) -\entparam\sum_{t=0}^{T-1} \log \pi_t^j(u_t^j|x_t^j)
\right], \nn
\end{align}
where \(x_0^j\sim\mu_0^j\), \(u_t^j\sim\pi_t^j(\cdot|x_t^j)\), and \(x_{t+1}^j\sim f_t^j(\cdot|x_t^j,u_t^j)\). 
\end{definition}
\end{myframe}
%%%%%%%%%%%%%%%%%%%%

{The rationale behind introducing entropy regularization is that it gives a \emph{tractable} perturbed equilibrium notion, which we call the entropy-regularized RA-MP-MFE and define below.\footnote{{We further show in Appendix \ref{sec:existence:ent-reg:appen} that the entropy-regularized RA-MP-MFE (Definition \ref{def:entropic:RA:MP:MFE}) close to the original RA-MP-MFE (Definition~\ref{def:RA:MP:MFE}) when \(\entparam\) is small.}}}

%%%%%%%%%%%%%%%%%%%%
\begin{myframe}
\begin{definition}[\highlight{Entropy-regularized RA-MP-MFE}]\label{def:entropic:RA:MP:MFE}
A pair \((\pi^\star_\entparam,\mu^\star_\entparam)\in\Pi\times\M\), with \(\pi^\star_\entparam\Let(\pi^{1,\star}_\entparam,\ldots,\pi^{M,\star}_\entparam)\in\Pi\) and \(\mu^\star_\entparam\Let(\mu^{1,\star}_\entparam,\ldots,\mu^{M,\star}_\entparam)\in\M\), is an entropy-regularized RA-MP-MFE with \(\eps\in \lcrc{0}{1}\) and \(\entparam>0\) if the following two conditions hold:
\begin{enumerate}[label=\textup{(\ref{def:entropic:RA:MP:MFE}-\alph*)}, leftmargin=*, widest=b, align=left]
\item The mean-field flows are consistent with the \edit{corresponding population} policies, i.e.,
\begin{align}\label{entropic:risk:av:cond:1}
\mu^{j,\star}_\entparam=\B_\prop^j(\pi^{j,\star}_\entparam)\quad \text{for }\,j\in[M].
\end{align}
\item Each population is entropy-regularized risk-averse optimal, i.e.,
\begin{align}\label{entropic:risk:av:cond:2.1}
\bar{\objective}_{j,\entparam}(\pi^{j,\star}_\entparam;\mu^\star_\entparam)
\ge \bar{\objective}_{j,\entparam}(\pi^j;\mu^\star_\entparam)
\quad \text{for all }\pi^j\in\Pi^j,\ j\in[M],
\end{align}
\newedit{equivalently,}
\begin{align}\label{entropic:risk:av:cond:2.2}
\pi^{j,\star}_\entparam \in \argmax_{\pi^j\in\Pi^j}
\bar{\objective}_{j,\entparam}(\pi^j;\mu^\star_\entparam)\newedit{.}
\end{align}
\end{enumerate}
\end{definition}
\end{myframe}
%%%%%%%%%%%%%%%%%%%%
%
%%%%%%%%%%%%%%%%%%%%
\begin{myOCP}
\begin{theorem}[\highlight{Existence of Ent. Reg. equilibria (informal)}]\label{thm:existence:entropic:RA:MP:MFE}
Let Assumption~\ref{assmpt:lipschitiz-model} be satisfied. 
Then, for every fixed \(\eps\in \lcrc{0}{1}\) and \(\entparam>0\), there exists an entropy-regularized RA-MP-MFE \((\pi^\star_{\entparam},\mu^\star_{\entparam})\in\Pi\times\M\).
\end{theorem}    
\end{myOCP}

A proof of Theorem \ref{thm:existence:entropic:RA:MP:MFE} is given in the Appendix \ref{sec:existence:ent-reg:appen}.

Similar to \(\briskopt\), we define \(\briskoptbeta:\M\to\Pi\)  as the entropy-regularized risk-averse best-response correspondence, i.e., for each \(\mu\in\M\), \((\pi^1,\ldots,\pi^M)\in\briskoptbeta(\mu)\), where, for every \(j\in[M]\), $\pi^j \in \argmax_{\pi^j\in\Pi^j}\bar{\objective}_{j,\entparam}(\pi^j;\mu)$.
Further, we define the composite operator $\Gamma_\beta \Let \B_\prop\circ\briskoptbeta:\M\to\M$.
%

% ====================================
% ======= contractivity ==============
% ====================================

\subsection{Steps towards contractivity}\label{subsec:ent:reg:contractivity}
\black{We are now ready to prove several key results on the contractivity of best response mappings.} 
We provide two separate contractivity results depending on the value of \(\eps\).

\begin{itemize}[leftmargin=*, label = \(\circ\)]
    \item \(\eps=1\); in this case, the ambiguity set is fixed. Under a particular structure on the reward, contractivity follows. This result is reported in Theorem \ref{thm:contractivity:regularized:BR} with the detailed proof given in Appendix \ref{sec:contraction:ent-reg:appen}. 

    \item \(\eps \in (0, 1)\); for this \newedit{case,} and under additional assumptions on the reward, contractivity follows if we further regularize the \edit{inner $\inf$ problem (with respect to the adversarial opponent mean-field flows) in~\eqref{def:entropic:RA:MP:MFE}}.  This result is reported in Theorem \ref{thm:contractivity:inner:regularized:BR} with a detailed proof given in Appendix \ref{sec:contraction:ent-reg:appen}.. 

\end{itemize}

\begin{assumption}\label{assum:eps:one:case}
For every \(j\in[M]\) and \(t\in \aset[]{0,\ldots,T}\), assume that the reward admits the decomposition
\begin{align}\label{eq:reward:decomp}
\bigl(x,\mu^j,\nu^{-j}\bigr) \mapsto r_t^j(x,\mu^j,\nu^{-j}) \Let a_t^j(x,\nu^{-j})+b_t^j(x,\mu^j),
\end{align}
where \(\sup_{t,x,\nu^{-j}}\bigl|a_t^j(x,\nu^{-j})\bigr|<+\infty.\) Moreover, there exists \(\ell_{j,t}\ge 0\) such that \(\|b_t^j(\cdot,\mu^j)-b_t^j(\cdot,\bar\mu^j)\|_{\infty}\le \ell_{j,t}\|\mu^j-\bar\mu^j\|_1\) for all \(\mu^j,\bar\mu^j\in\P(\X)\) with \(L_j \Let \sum_{t=1}^{T}\ell_{j,t}\). 
\end{assumption}

We define the occupation norm under which we show contraction. Let \(o^j=(\xi^j,\eta^j)\), \(\bar o^j=(\bar{\xi}^j,\bar{\eta}^j)\), \(u_t^j \Let \|\xi_t^j-\bar{\xi}_t^j\|_1\), and \(v_t^j \Let \|\eta_t^j-\bar{\eta}_t^j\|_1\), then
\begin{align}\label{eq:occupation:measure:norm}
\|o^j-\bar{o}^j\|_{\onorm} \Let \max \aset[\Big]{\max_{0\le t \le  T}u_t^j,\,\max_{0 \le t<T}v_t^j}
\end{align}
and \(\|o-\bar o\|_{\onorm,\infty}\Let\max_{j\in[M]}\|o^j-\bar o^j\|_{\onorm}\). Further, for \(\entparam>0\), we also define
\begin{align}
\Psi_{j,\entparam}({o}^j)
\Let -\entparam\sum_{t=0}^{T-1}\sum_{x\in\X}\sum_{u\in\U}
{\eta}_t^j(x,u) \log\frac{{\eta}_t^j(x,u)}{{\xi}_t^j(x)}. \nn
\end{align}
and \(\Phi_{j,\entparam}(\tilde{o}^j;o) = \Phi_j(\tilde{o}^j;o)+\Psi_{j,\entparam}(\tilde{o}^j)\), where \(\Phi_j(\tilde{o}^j;o)\) is defined in \eqref{eq:occ:phi}.
%
%%%%%%%%%%%%%%%%%%%%
\begin{myOCP}
 \begin{theorem}[\highlight{Contractivity, \(\eps=1\) (informal)}]
\label{thm:contractivity:regularized:BR}
Let Assumption \ref{assum:eps:one:case} hold. \edit{Fix the parameter \(\entparam >0\).} 
Define \(\bestresp_{\entparam}(o)  \Let \left(\bestresp_{1,\entparam}(o),\ldots,\bestresp_{M,\entparam}(o)\right)\) where \(\bestresp_{j,\entparam}(o)\Let\argmax_{\widetilde{o}^j\in\Occ^j}\Phi_{j,\entparam}(\widetilde o^j;o)\). 
If \(\entparam>\max_{j\in[M]} L_j,\) then \(\bestresp_{\entparam}\) is a contraction on \(\Occ\) in \(\|\cdot\|_{\onorm,\infty}\). Consequently, the recursion $\mu(n+1) = \Gamma_\beta(\mu(n))$ converges to the unique entropy-regularized RA-MP-MFE.
\end{theorem}   
\end{myOCP}

We now furnish our final technical result, a contraction for \(\eps \in (0, 1)\). We \newedit{first} need the following assumption.

\begin{assumption}\label{assum:eps:0:1}
Consider the decomposition~\eqref{eq:reward:decomp}. 
We assume that there exist \(c_t^j:\X\to\Real\) and \(A_t^{j,k}:\X\times\X\to\Real\), \(k\neq j\), such that for all \(x\in\X\) and \(\nu^{-j}\in\P(\X)^{M-1}\), 
\begin{align}\label{eq:affine:opponent:reward}
a_t^j(x,\nu^{-j})=c_t^j(x)+\sum_{k\neq j}\sum_{y\in\X}A_t^{j,k}(x,y)\nu^k(y).
\end{align}
% for all \(x\in\X\) and \(\nu^{-j}\in\P(\X)^{M-1}\). 
\end{assumption}

We also define the quantities which will be used later, in particular, \(\ell_{j,t}^a\Let\sum_{k\neq j}\max_{x\in\X} \|A_t^{j,k}(x,\cdot)\|_\infty\) and 
\(L_j^a\Let\sum_{t=1}^T\ell_{j,t}^a.\)

\begin{myOCP}
 \begin{theorem}[\highlight{Contractivity, \(\eps \in (0,1)\) (informal)}]
\label{thm:contractivity:inner:regularized:BR}
Suppose that Assumptions \ref{assum:eps:one:case} and \ref{assum:eps:0:1} hold.
For \(\widehat\mu^{-j}=\{\widehat\mu_t^k\}_{k\neq j,\;t=0}^T\), let \(\left\|\widehat\mu^{-j}\right\|_2^2\Let\sum_{k\neq j}\sum_{t=0}^T\sum_{x\in\X}\left|\widehat\mu_t^k(x)\right|^2\); define the inner-regularized risk-averse payoff as
\begin{align}\label{eq:inner:reg}
&\Phi_{j,\entparam}^{\sigma}(\widetilde{o}^j;o) \Let
\min_{\widehat\mu^{-j}\in\ambset(\xi(o),\eps)}
\Bigg{[} \sum_{t=0}^T\sum_{x\in\X}\widetilde \xi_t^j(x) \cdot \nn \\& r_t^j\bigl(x,\xi_t^j,\widehat\mu_t^{-j}\bigr) +\frac{\sigma}{2}\left\|\widehat\mu^{-j}\right\|_2^2 \Bigg{]} +\Psi_{j,\entparam}(\widetilde{o}^j). 
\end{align}
Define \(\bestresp_{j,\entparam}^{\sigma}(o) \Let \argmax_{\widetilde{o}^j\in\occset^j}\Phi_{j,\entparam}^{\sigma}(\widetilde{o}^j;o),\) and the concatenated qunatity \(\bestresp_{\entparam}^{\sigma}(o)
\Let \bigl(\bestresp_{1,\entparam}^{\sigma}(o),\ldots,\bestresp_{M,\entparam}^{\sigma}(o)\bigr).\)
\edit{There exist constants \(\Lambda_{j,\eps}\ge0\) such that, for \(\entparam>\max_{j\in[M]}(L_j+2L_j^a\Lambda_{j,\eps})\), \(\bestresp_{\entparam}^{\sigma}\) is a contraction on \(\Occ\) in \(\|\cdot\|_{\onorm,\infty}\), uniformly in \(\sigma>0\).
As \(\sigma\to0^+\), its fixed point converges to the unique fixed point of \(\bestresp_{\entparam}\). Moreover, the recursion  $\mu(n+1) = \Gamma_\beta(\mu(n))$ converges  to the unique entropy-regularized RA-MP-MFE.} %for every \(\mu_0\in\P(\X)^M\).}
\end{theorem}   
\end{myOCP}

\section{Fictitious Play}\label{sec:fic:play}
Fictitious play (FP) has been a long-standing algorithm used to compute equilibria for mean-field games with and without entropy regularization~\cite{ref:Fip:Play:brown1951iterative, perrin2020fictitiousplaymeanfield, eich2025boundedrationalityequilibriumlearning,ref:BERGER2007572}.
In this work, we present an extended version of the algorithm to multiple populations under risk aversion, thereby utiziling it as an effective algorithm for computing robust policies.
In particular, for each population \(j\in[M]\), each time step \(t=0,\ldots,T\), and iteration index \(\ell\), the algorithm iteratively computes a best response $\bar \pi^{j, \ell}_t$ to the population distribution induced by the average of the best responses from previous iterations.
Specifically, for each population \(j\in[M]\), each time step \(t=0,\ldots,T\), and iteration index \(\ell\), we have~\cite{perrin2020fictitiousplaymeanfield}:
\begin{equation}
\label{eq:discrete-bar-mu}
\bar \mu^{j, \ell}_t= \frac{\ell-1}{\ell}\bar \mu^{j, \ell-1}_t + \frac{1}{\ell}\mu^{\pi^{j, \ell}}_t\newedit{.}
\end{equation}
The policy generating this average distribution \newedit{is~\cite{perrin2020fictitiousplaymeanfield}}:
\begin{equation}
\label{eq:discrete-bar-pi}
   \bar \pi^{j, \ell}_t(u|x) = \frac{\sum _{i=0}^\ell \mu^{\pi^{j, i}}_t(x) \pi^{j, i}_t(u|x)}{\sum_{i=0}^\ell \mu^{\pi^{j,i}}_t(x)},
\end{equation}
\edit{where $\bar \pi^{j, \ell}_t(u|x)$ is the average of the best responses of population $j$ from iteration $0$ upto iteration $\ell$.} More details are given in Appendix \ref{sec:fic-play:conv:appen}. 

Similar to~\cite{eich2025boundedrationalityequilibriumlearning} we present the proof of the algorithm for continuous time iterations and the implemented discrete time version can be seen as a numerical approximation.
The continuous time analysis enables us to deploy standard tools from calculus and control theory, in particular, continuous-time Lyapunov functions.
We denote the fictitious play time by $s$, so as to explicitly differentiate it from the MP-MFG time horizon.
We make the following standard assumptions on the reward structure \edit{defined in~\eqref{eq:reward:decomp}~\cite{perrin2020fictitiousplaymeanfield, perolat2021scalingmeanfieldgames, eich2025boundedrationalityequilibriumlearning}}:

\begin{assumption}\label{assum:eps:one:monotone:reward}
Let \(\mathcal{V} \subset \Real[|\X|]\) be an open set containing \(\P(\X)\).
Under Assumption~\ref{assum:eps:one:case}, for every \(j\in[M]\), \(t\in\{0,\ldots,T\}\), and \(x\in\X\), $b_t^j(x,\cdot)$ is continuously differentiable on \(\mathcal{V}\) 
and \(\nu \mapsto b_t^j(\cdot,\nu)\) is monotone, i.e., $\sum_{x\in\X}(\nu(x) -\nu'(x)) \bigl( b_t^j(x,\nu)-b_t^j(x,\nu' \bigr) \le 0$ for all \(\nu, \nu' \in \P(\X)\).
\end{assumption}

%%%%%%% rem %%%%%%%%%%%

\begin{remark}
Observe that Assumption~\ref{assum:eps:one:monotone:reward} imposes a monotonicity condition on the \emph{own-population} component \(b_t^{j}\) of the reward function and \emph{does not} restrict the reward component depending on the opponent-dependent populations \(a^j_t\).
It is therefore weaker than requiring monotonicity of the full multi-population reward with respect to the entire mean-field profile, \edit{which is the case in existing literature~\cite{perrin2020fictitiousplaymeanfield, perolat2021scalingmeanfieldgames, eich2025boundedrationalityequilibriumlearning}}.
\end{remark}

Recall the definition of risk-averse exploitibility in Definition \ref{def-ra-explt}. Then, the following result holds.
%%%%%%% fic play convergence %%%%%%%%%
\begin{myOCP}
\begin{theorem}[\highlight{Convergence of FP (informal)}]
\label{thm:fp_FH}
Let the RA-MP-MFG satisfy Assumptions~\ref{assmpt:lipschitiz-model} and \ref{assum:eps:one:monotone:reward}, and let $\eps=1$ in~\eqref{eq:ambiguity:set}. 
Then, the exploitability of the continuous-time fictitious play process is a \emph{strong Lyapunov function} of the system, i.e., $\frac{\dd}{\dd s} \phi^j_{\RA}(\bar\pi^s) \le - \frac{1}{s} \phi^j_{\RA}(\bar\pi^s) ~\forall~s \ge 1 \text{ and }j\in[M]. $
Consequently, $\phi_{\RA}(\bar\pi^s) = \mathcal{O}(1/s)$.
\end{theorem}    
\end{myOCP}

\section{Numerical Experiments}\label{sec:numexp}

In this section, we investigate the effects of risk aversion in MP-MFGs and demonstrate the efficacy of our proposed entropy-regularized risk-averse fixed-point iteration (RA-FPI) and risk-averse fictitious play (RA-FP) for computing equilibrium solutions across a range of settings.
For all problems, we report the exploitability of the RA-MP-MFE (Definition~\ref{def-ra-explt}) and its entropy-regularized counterpart, where \(\bar{\objective}_j(\pi^j;\mu^\pi)\) in Definition~\ref{def-ra-explt} is replaced by \(\bar{\objective}_{j,\beta}(\pi^j;\mu^\pi)\).

To the best of our knowledge, there are no existing baselines that can be directly applied to the class of RA-MP-MFG problems considered here.
In the absence of standard baselines, we instead investigate the sensitivity of different MP-MFG equilibrium notions to deviations in the mean-field trajectories.
Specifically, we consider the standard multi-population MFE (MP-MFE)~\cite{perolat2021scalingmeanfieldgames}, its entropy-regularized counterpart (ER-MP-MFE), our proposed RA-MP-MFE (Definition~\ref{def:RA:MP:MFE}) computed using RA-FP, and its entropy-regularized counterpart (Definition~\ref{def:entropic:RA:MP:MFE}) computed using RA-FPI, hereafter referred to a ER-RA-MP-MFE, respectively.
For each equilibrium policy, we generate \(N\) perturbations of the nominal equilibrium mean-field distributions. 
For each perturbed mean-field flow, we evaluate the expected cumulative reward of the representative agent under the corresponding equilibrium policy.
We then rank the \(N\) perturbed trajectories according to their expected cumulative rewards and report the mean difference between the nominal performance and the average performance over the worst \(10\%\) of perturbations
\begin{align}\label{regret-metric}
&\Delta_{\textrm{EQ}} \Let
\mathbb{E}^{\pi^{j,\textrm{EQ}}}\left[
\sum_{t=0}^{T} r_t^j \left(x_t^j,\mu_t^{j,\textrm{EQ}},
\mu_t^{-j,\textrm{EQ}}\right)\right]- \nn \\
&  \frac{1}{|\mathcal{I}_{10\%}|}
\sum_{i\in\mathcal{I}_{10\%}}
\mathbb{E}^{\pi^{j,\textrm{EQ}}}
\left[\sum_{t=0}^{T}r_t^j
\left(x_t^j,\mu_t^{j,\textrm{EQ}},\mu_t^{-j,(i)}\right)
\right],
\end{align}
where \(\mathcal{I}_{10\%}\) denotes the set of indices corresponding to the \(10\%\) of perturbations with the lowest expected cumulative rewards, and $\textrm{EQ}\in
\left\{
\substack{
\textrm{MP-MFE},\textrm{ER-MP-MFE},\\
\textrm{RA-MP-MFE},\textrm{ER-RA-MP-MFE}
}
\right\}$.
Thus, \(\Delta_{\textrm{EQ}}\) provides an empirical measure of the sensitivity of each equilibrium policy to unfavorable mean-field trajectory deviations.
Additionally, we perform extensive ablation studies to examine the effects of the key parameters introduced in our formulation, namely \(\varepsilon\), \(\beta\), and \(\sigma\).
Further details are given in Appendix~\ref{sec:add:num:det:appen}.

In the case where \(\eps=1\), the ambiguity set~\eqref{eq:ambiguity:set}  becomes independent of the nominal trajectory, allowing the inner minimization over the uncertain population distributions in~\eqref{J-bar-inf} to be reformulated using dual variables $V_t \in\mathbb{R}^{|\X|}$ for all $t = 0, \ldots, T$. 
A detailed derivation of this dual reformulation is provided in Appendix~\ref{sec:add:num:det:appen}.

\begin{table}[t]
    \centering
    \scriptsize
    \setlength{\tabcolsep}{3pt}
    \renewcommand{\arraystretch}{0.9}
    \begin{tabular}{llcccc}
        \toprule
        & & MP- & ER-MP- & RA-MP- & ER-RA-MP- \\
        & & MFE & MFE & MFE & MFE \\
        \midrule
         Grid & Nominal & \textbf{7.67} & 7.60 & 0.35 & 0.37 \\
             Navigation      & $\Delta_{\textrm{EQ}}$ & 7.65 & 7.58 & \textbf{0.29} & 0.35 \\
        \midrule
        3 Population   & Nominal & \textbf{12.13} & 11.73 & 12.05 & 11.77 \\
           Chasing          & $\Delta_{\textrm{EQ}}$ & 1.16 & 1.91 & \textbf{1.10} & 2.0 \\
            \midrule
        4 Population   & Nominal & \textbf{11.91} & 11.59 & 11.78 & 11.58 \\
           Chasing          & $\Delta_{\textrm{EQ}}$ & 1.26 & 2.0 & \textbf{1.13} & 2.06 \\
            \midrule
        Stock   & Nominal & \textbf{136.65} & 136.77 & 130.56 & 130.49 \\
           Trading          & $\Delta_{\textrm{EQ}}$ & 25.21 & 25.51 & \textbf{22.09} & 22.16 \\
        \bottomrule
    \end{tabular}
    \caption{Expected rewards under nominal mean-field flow $(\mu_t^{j,\textrm{EQ}},
\mu_t^{-j,\textrm{EQ}})$ and sensitivity of corresponding equilibrium policies as defined in~\eqref{regret-metric}.}
    \label{tab:sensitivity}
     \vspace{-0.1in}
\end{table}

\textbf{Grid Navigation ($\eps = 1$).} We consider an {RA-MP-MFG} with two populations traversing a \(5\times5\) grid. 
Population 2 aims to reach the bottom-right corner (Figure~\ref{fig:navigation}(c)) and receives a per-time-step reward proportional to the fraction of the population at the target. 
Population 1, at a given state $x^1_t\in\X$, receives $r_t^1(x_t^1,\mu_t^1,\mu_t^2)=\mu_t^2(x_t^1)$, thereby incentivizing it to align its distribution with that of Population 2.
We assume that Population 1 is risk-averse to deviations in Population 2's trajectory.
Table~\ref{tab:sensitivity} shows that MP-MFE and ER-MP-MFE achieve higher nominal cumulative rewards but exhibit a substantial degradation in performance under deviations from the equilibrium flows, indicating greater sensitivity to trajectory perturbations.
In contrast, RA-MP-MFE and ER-RA-MP-MFE achieve lower nominal rewards while exhibiting greater robustness to deviations, highlighting the tradeoff between reward maximization and robustness.

As shown in Figure~\ref{fig:navigation}(d)-(f), risk aversion leads Population 1 from a nearly deterministic trajectory to a more spatially diversified strategy. 
When \(\eps=1\), the ambiguity set contains all dynamically feasible trajectories of Population 2, whose behavior is independent of Population 1.
The resulting worst-case objective therefore favors policies that distribute probability mass across multiple states, providing greater robustness to distributional deviations.
Increasing \(\beta\) strengthens entropy regularization and, consequently, improves the reliability of RA-FPI, as illustrated in Figure~\ref{fig:navigation}(a).
However, increasing \(\beta\) also encourages greater stochasticity in the decision-making policies, leading to more uniform population trajectories.
Figure~\ref{fig:navigation}(b) further shows that the FP exploitability decays at a rate consistent with Theorem~\ref{thm:fp_FH}. 
\begin{figure}[t!]
    \centering
    \includegraphics[scale=0.3]{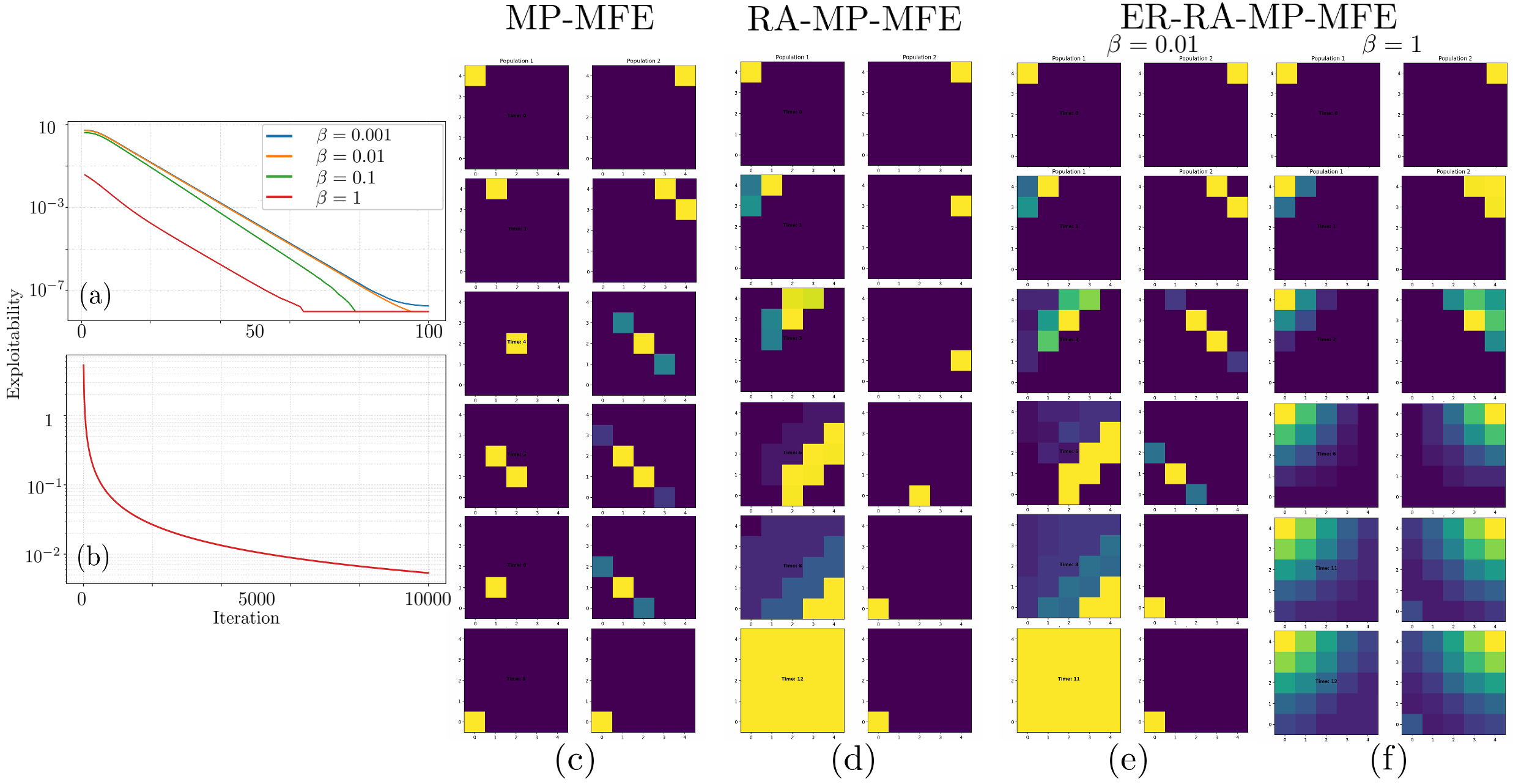}
    \caption{(a) RA-FPI; (b) RA-FP; (c)-(f) population distributions of Navigation game.}
    \label{fig:navigation}
    \vspace{-0.25in}
\end{figure}

\textbf{Multi-Population Chasing ($\eps = 1$).} The multi-population chasing game introduced in~\cite{perolat2021scalingmeanfieldgames} is a standard MP-MFG setting in which \(M\) populations move over a 2D grid, with rewards of the form: $r^j(x, \mu^1, \ldots, \mu^M) \Let -\log(\mu^j(x)) + \sum_{i\neq j}\mu^i(x)\bar r^{j, i}(x)$, for all $x\in \X$.
The dependence of each population's reward on the mean-field distributions of all populations results in a coupled problem.
We consider two variants of this problem on a \(5\times5\) grid: (i) a 3 population chasing problem in which Population 1 is risk-averse to deviations in Population 2's trajectory, and (ii) a 4 population chasing problem in which Population 1 is risk-averse to deviations in the trajectories of Populations 2 and 3, while all other interactions remain risk-neutral.
Once again, Table~\ref{tab:sensitivity} shows that RA-MP-MFE may achieve lower nominal rewards but its average
worst-case reward degrades less under deviations.
We attribute the higher sensitivity of ER-RA-MP-MFE to the larger value of $\beta$ ($=1$), which, although necessary to ensure convergence, can induce increasingly uniform behavior and lead to greater deviations from the true objective function.

Figure~\ref{fig:chasing} similarly shows that entropy regularization promotes convergence toward the ER-RA-MP-MFE as \(\beta\) increases, while FP exhibits slower but steady convergence toward the RA-MP-MFE in both variants.
In the 4 population case, the risk-neutral equilibrium is characterized by symmetric, spatially dispersed trajectories (Figure~\ref{fig:chasing}(e)).
Risk aversion qualitatively changes this behavior, with the key differences between the risk-neutral and risk-averse settings highlighted by the red box.
Under risk aversion, Population 1 breaks this symmetry and adopts a \emph{safer trajectory} with respect to Populations 2 and 3 by reducing the mass pursuing Population 2, thereby reducing its vulnerability to feasible deviations in Population 2's trajectory. 

\begin{figure}[t!]
    \centering
    \includegraphics[scale=0.5]{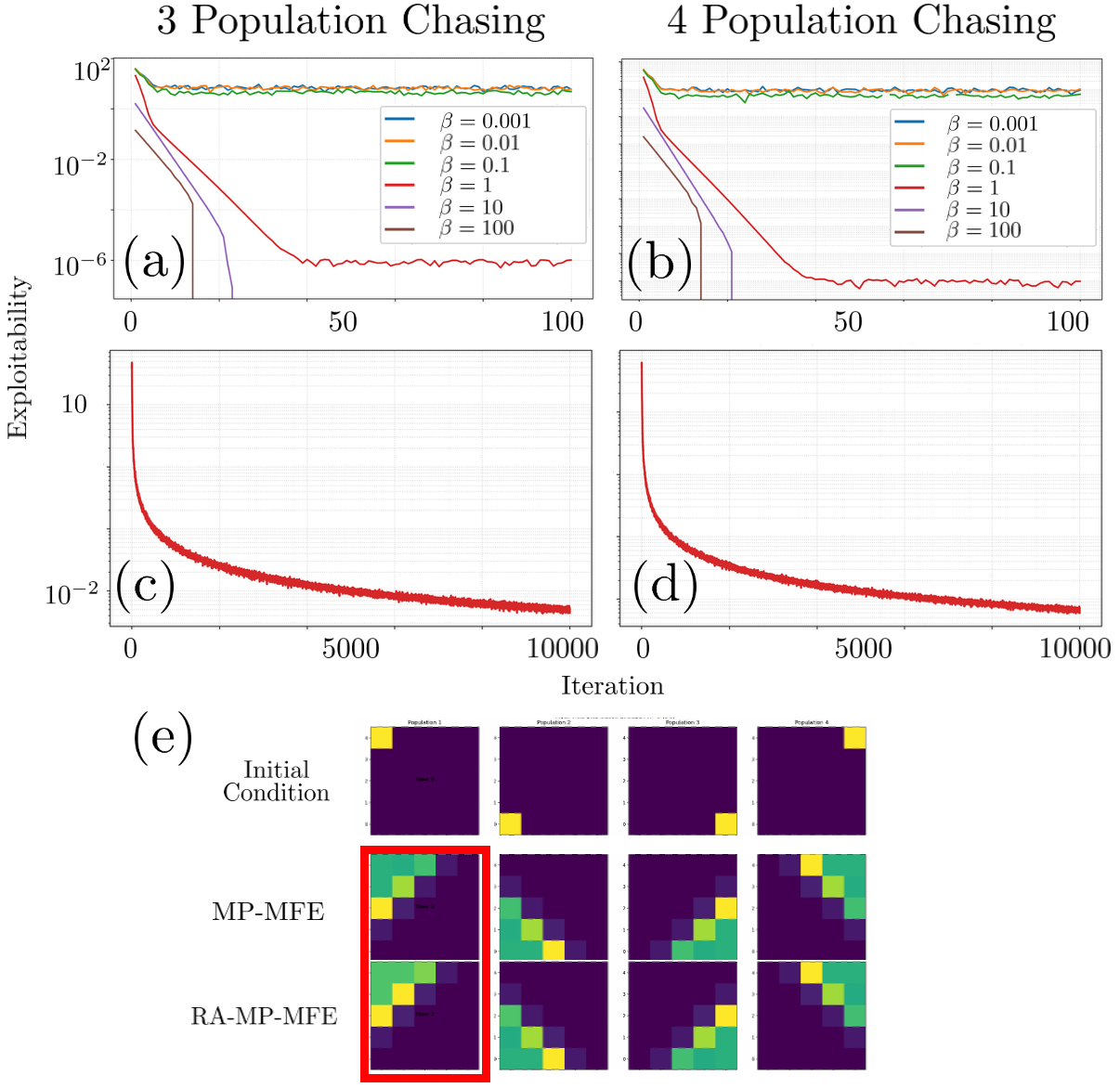}
    \caption{(a)-(b) RA-FPI; (c)-(d) RA-FP; (e) population distributions of 4 Population Chasing.}
    \label{fig:chasing}
    \vspace{-0.25in}
\end{figure}

\textbf{Stock Market Trading ($\eps = 1$).} Consider the stock market example presented in Section~\ref{sec:introduction}.
We construct a simplified, multi-population model of the the trading scenario introduced as the \emph{price impact problem} in~\cite{carmona2020applicationsmeanfieldgames, cardaliaguet2018mean} between AI (Population 1) and human (Population 2) populations by discretizing the inventory~\cite{wibault2026recurrentstructuralpolicygradient} and assuming endogenous price changes that are independent of the mean-field.

The discrete inventory positions are $\X \Let \{-K, \ldots, 0, \ldots, K\}$ for some integer $K>0$ representing the holdings of a population.
The actions available to each population are \texttt{SELL}, \texttt{HOLD}, and \texttt{BUY}, with inventory transitions evolving deterministically.
The change in the stock price at time $t$ is denoted by $\Delta p_t$ and follows  $\Delta p_t \Let \gamma_t + \theta_t z_t$, where $\gamma_t$ represents the price drift, $\theta_t$ represents the volatility and $z\sim \textrm{Discrete}(-1, 0, 1)$, where $\E[z_t]= 0$ and $\textrm{var}(z_t) = 1$.
The rewards for the AI population follows $r^1_t(x^1_t, \mu^1_t, \mu^2_t) \Let \gamma_tx^1_t + c_0x^1_t\cdot \left\langle q, \mu^2_t\right\rangle - c_1(x^1_t)^2$, where $q\t \Let [-K, \ldots, 0, \ldots K]$ and $c_0, c_1>0$.
The term $c_0x^1_t\cdot \left\langle q, \mu^2_t\right\rangle$ represents the price pressure created by aggregate human inventory shifts and $c_1(x_t)^2$ penalizes large AI inventory positions due to volatility and holding costs.
The rewards for the human population is given by $r^2_t(x^2_t,  \mu^1_t, \mu^2_t) \Let \gamma_tx^2_t -w_0(\theta_t x_t)^2+ w_1x^2_t\cdot \left\langle q, \mu^2_t\right\rangle - w_2\mathds{1}_{\gamma_t < -\chi}(x^2_t + K)$, where  $-w_0(\theta_t x_t)^2$, $w_1x^2_t\cdot \left\langle q, \mu^2_t\right\rangle$ and   $w_2\mathds{1}_{\gamma_t < -\chi}(x^2_t + K)$ represent the penalty due to volatility in the price, herding behavior and penalty due to panic selling when the stock price drift falls below $-\chi$.

In order to bridge the gap between mean-field theory and real-world applications, we calibrate the price model to mimic the dynamics of Apple Inc. (AAPL)  over September 2025, using data retrieved through \texttt{yfinance}~\cite{aroussi2026yfinance} and analyze the effect of risk-aversion of the AI population to the deviations in the human population's mean-field trajectory.
One can verify that the AI reward does not satisfy the monotonocity condition in Assumption~\ref{assum:eps:one:monotone:reward}.
We therefore use a small entropy-regularization parameter, $\beta=10^{-6}$, and compute the ER-RA-MP-MFE as an approximation to the RA-MP-MFE, as supported by Proposition~\ref{thm:approx:vanishing:entropic:RA:MP:MFE} in Appendix~\ref{sec:existence:ent-reg:appen}.
Table~\ref{tab:sensitivity} shows that the risk-averse policy exhibits the lowest sensitivity to trajectory deviations, while Figure~\ref{fig:trading} illustrates the resulting exploitability for different values of $\beta$.

\begin{figure}[t!]
    \centering
    \includegraphics[scale=0.65]{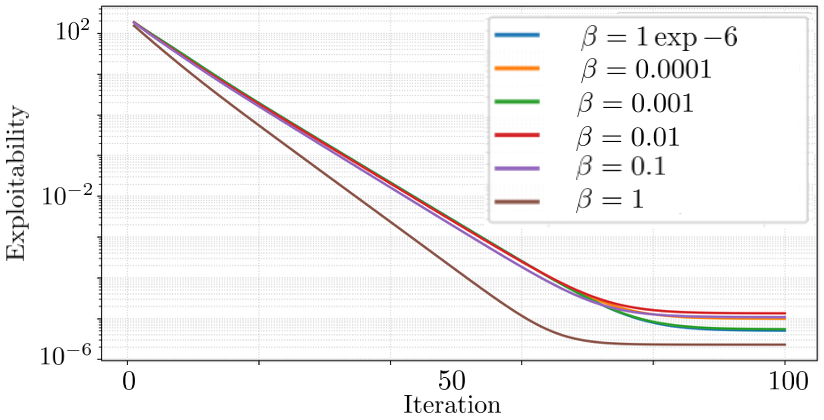}
    \caption{Exploitability for Stock Market Trading.}
    \label{fig:trading}
    \vspace{-0.2in}
\end{figure}

\textbf{Linear Navigation ($\eps \in (0, 1)$).} 
We consider an MP-MFG with two populations traversing a three-state environment, where agents can move \texttt{LEFT}, \texttt{RIGHT}, or remain \texttt{STATIONARY}.
Similar to the Grid Navigation problem, Population 1, at a given state $x^1_t\in\X$, receives $r_t^1(x_t^1,\mu_t^1,\mu_t^2)  \Let \mu_t^2(x_t^1)$, which satisfies Assumption~\ref{assum:eps:0:1}. 
We assume that Population 1 is risk-averse to deviations in Population 2's trajectory.
For \(\eps\in(0,1)\), \(\ambset(\mu,\eps)\) depends explicitly on the nominal mean-field trajectory. 
Consequently, the inner minimization is coupled with the nominal trajectory and the dual formulation no longer \edit{yields a tractable problem or a computational speedup.}
We therefore solve the resulting max-min problem directly using the iterative \texttt{MINIMAX-APPA} algorithm proposed in~\cite{pmlr-v125-lin20a}, which is designed to handle such coupled minimax optimization problems.
Figure~\ref{fig:linear_navigation}(a) shows a reduction in exploitability for fixed values of \(\sigma\) and \(\beta=1\), consistent with the guarantee of Theorem~\ref{thm:contractivity:inner:regularized:BR}.
We also observe that the risk-averse payoff decreases as \(\epsilon\) increases (Figure~\ref{fig:linear_navigation}(c)), as expected since the set of dynamically feasible perturbations expands with \(\epsilon\).
\edit{Finally, an ablation over different values of \(\sigma\), with fixed \(\eps=0.1\) and \(\beta=1\) (Figure~\ref{fig:linear_navigation}(b)), shows that reducing the value of \(\sigma\) makes the inner minimization problem more difficult to solve due to weaker convexity.}
%
%In contrast, larger values of \(\sigma\) require larger values of the entropy parameter \(\beta\) to ensure convergence, as predicted by Theorem~\ref{thm:contractivity:inner:regularized:BR}. 

\edit{While \texttt{MINIMAX-APPA} performs well in low-dimensional settings, its iterative nature can lead to numerical errors as the state-action space grows.
Indeed, directly solving the resulting max-min problem can become challenging due to the complexity of alternating optimization and the scalability of iterative minimax methods.
Developing efficient and numerically viable approaches for larger state-action spaces remains an important direction for future work.}

\begin{figure}[t!]
    \centering
    \includegraphics[scale=0.7]{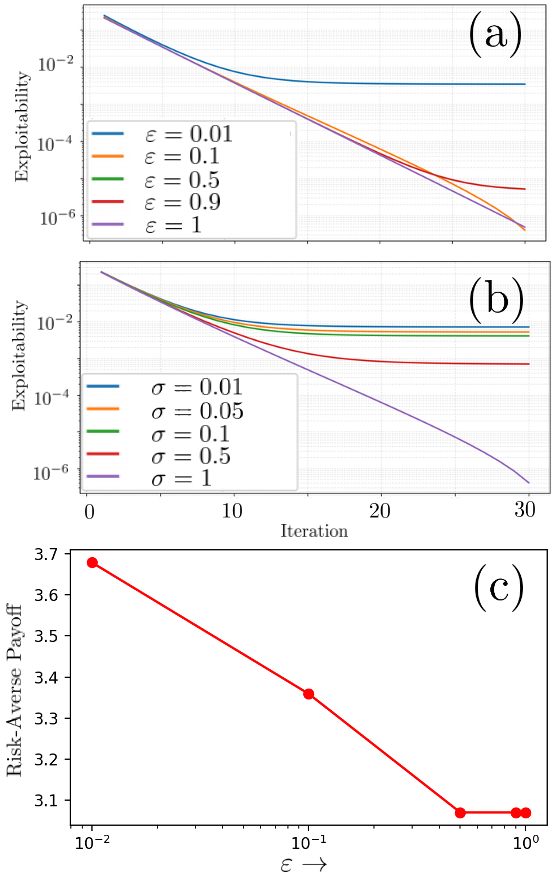}
    \caption{(a) Varying $\eps$ under $\sigma = 1$; (b) Varying $\sigma$ under $\eps=0.1$; (c) Cumulative rewards.}
    \label{fig:linear_navigation}
     \vspace{-0.2in}
\end{figure}

\section{Conclusion}

We introduced a framework for risk-averse multi-population mean-field games, in which each population solves a coupled robust optimization problem against dynamically feasible deviations of the other populations.
This coupling gave rise to a novel risk-averse multi-population mean-field equilibrium.
Using tools from variational analysis, we established the existence of an RA-MP-MFE and leveraged entropy regularization to obtain contractive approximate RA-MP-MFEs.
Through Lyapunov analysis, we further showed that an extended fictitious-play dynamics converges to an RA-MP-MFE under appropriate monotonicity conditions.
Numerical experiments demonstrated that RA-MP-MFEs are less sensitive to deviations in the nominal mean-field trajectories compared to traditional MP-MFEs, highlighting the robustness induced by our formulation.
An important direction for future work is to extend the framework to model-free settings with function approximation.

\bibliographystyle{amsalpha}
\bibliography{refs}

\appendix

\clearpage

\appendix

\setcounter{theorem}{0}
\setcounter{proposition}{0}
\setcounter{definition}{0}

\section*{Appendix}\label{appendix:a}

\section{Technical TL;DR and a roadmap}\label{sec:tech:tldr:appen}
We summarize our primary theoretical contributions in brief; see also Figure~\ref{fig:theoretical-roadmap} for a bird's-eye view.
%%%%%%%%%%%%%%%%%
\begin{enumerate}[label={\textup{(Res-\alph*)}}, leftmargin=*,align=left]
\item \label{results:1} \highlight{Main theorems:} Theorem \ref{thm:existence:RA:MP:MFE} establishes existence of the RA-MP-MFG;
Theorems \ref{thm:existence:entropic:RA:MP:MFE}--\ref{thm:contractivity:inner:regularized:BR:appen} provide existence and contractivity results (for \(\eps =1\) and \(\eps \in (0,1)\)) for the entropy-regularized RA-MP-MFG. Without any entropy regularization, for the RA-MP-MFG, Theorem \ref{thm:fp_FH} analyzes continuous-time fictitious play and provides a certificate of convergence. Proofs are in Appendices~\ref{sec:reach:set:appen}--\ref{sec:add:num:det:appen}.

\item \label{results:2} \highlight{Auxiliary and standalone results:} Lemma \ref{prop:reach:set:properties:appen}, \ref{lem:strong:concavity:entreg}, and Proposition \ref{prop:ambiguity:set} are technical results needed to prove the Theorems mentioned in \ref{results:1}. Proposition \ref{thm:non-contractivity:RA:MP:MFG} and \ref{thm:approx:vanishing:entropic:RA:MP:MFE} concerns non-contractivity of the RA-MP-MFE and approximation of MFEs. These are standalone results.
\end{enumerate}
%%%%%%%%%%%%%%%%%%%
% Figure~\ref{fig:theoretical-roadmap} provides an overview of these results and the proofs are provided in Appendices~\ref{sec:reach:set:appen}--\ref{sec:add:num:det:appen}. 
%%%%%%%%%%%%%%%%%%%

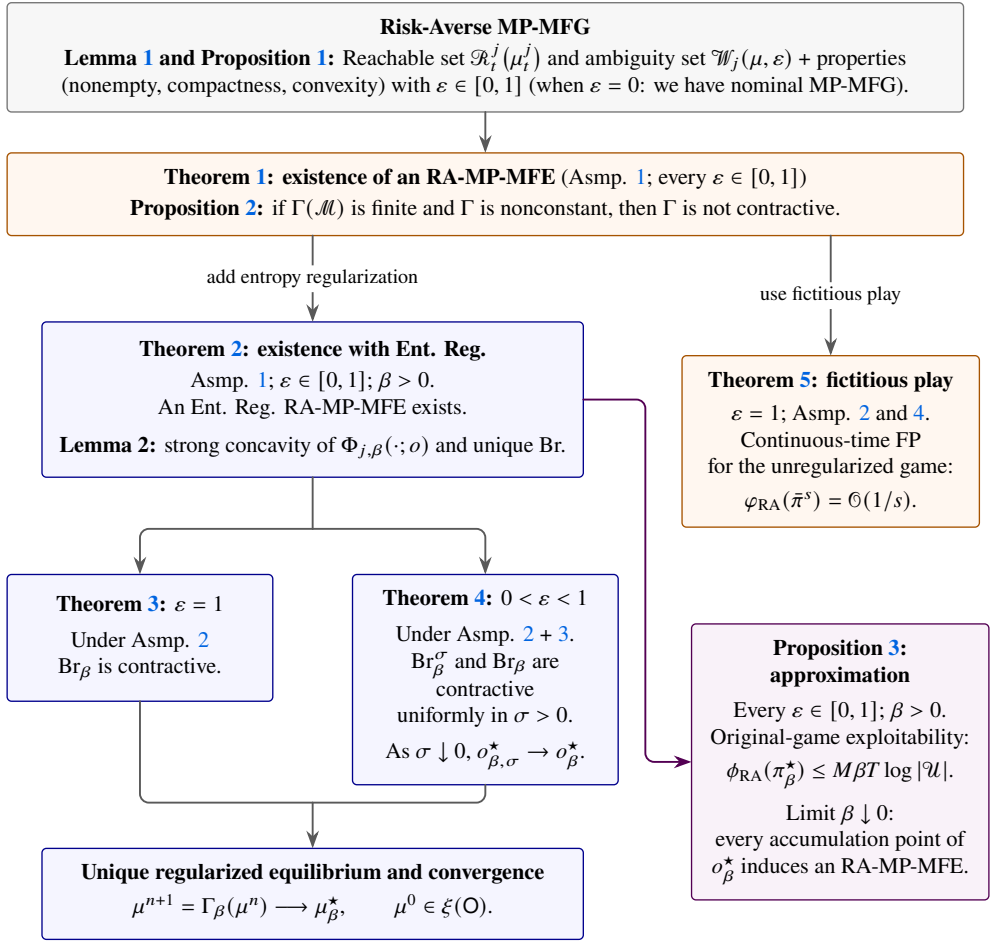
\begin{figure}[!tp]
\centering
\resizebox{\textwidth}{!}{%
\begin{tikzpicture}[
    y=1.2cm,
    font=\small,
    box/.style={draw=black!55,rounded corners=2pt,line width=0.6pt,
        align=flush center,inner sep=7pt,outer sep=0pt},
    foundation/.style={box,fill=black!3},
    original/.style={box,draw=orange!65!black,fill=orange!6},
    regularized/.style={box,draw=blue!55!black,fill=blue!4},
    approximation/.style={box,draw=violet!65!black,fill=violet!5},
    connector/.style={
        draw=black!65,
        line width=0.85pt,
        rounded corners=4pt,
        line cap=round,
        line join=round
    },
    arrow/.style={
        connector,
        -{Stealth[length=2.8mm,width=2mm]}
    },
    edge label/.style={font=\footnotesize,fill=white,inner sep=2pt}
]

% Structural foundation and existence.
\node[foundation,text width=14.6cm] (structure) at (0,0) {
    \textbf{Risk-Averse MP-MFG}\\[2pt]
    \textbf{Lemma \ref{prop:reach:set:properties:appen} and Proposition \ref{prop:ambiguity:set}:} Reachable set \(\reachset\) and ambiguity set
    $\mathcal{W}_j(\mu,\eps)$ + properties (nonempty, compactness, convexity) with $\eps\in[0,1]$ (when $\eps=0$: we have nominal MP-MFG).
};

\node[original,text width=14.6cm] (existence) at (0,-1.8) {
    \textbf{Theorem \ref{thm:existence:RA:MP:MFE}: existence of an RA-MP-MFE} (Asmp. \ref{assmpt:lipschitiz-model}; every $\eps\in[0,1]$)\\[2pt]
    \textbf{Proposition \ref{thm:non-contractivity:RA:MP:MFG}:} if $\Gamma(\mathcal{M})$ is finite and
    $\Gamma$ is nonconstant, then $\Gamma$ is not contractive.
};

\draw[arrow] (structure) -- (existence);

% Two computational routes.
\node[regularized,text width=8.05cm,minimum height=2.45cm]
    (entropy) at (-2.725,-4.5) {
    \textbf{Theorem \ref{thm:existence:entropic:RA:MP:MFE}: existence with Ent. Reg.}\\[2pt]
    Asmp. \ref{assmpt:lipschitiz-model}; $\eps\in[0,1]$; $\beta>0$.\\
    An Ent. Reg. RA-MP-MFE exists.\\[5pt]
    \textbf{Lemma 2:} strong concavity of
    $\Phi_{j,\beta}(\cdot;o)$ and unique Br.
};

\node[original,text width=4.2cm,minimum height=2.45cm]
    (fp) at (5.45,-5.05) {
    \textbf{Theorem \ref{thm:fp_FH}: fictitious play}\\[3pt]
    $\eps=1$; Asmp. \ref{assum:eps:one:case} and \ref{assum:eps:one:monotone:reward}.\\
    Continuous-time FP\\
    for the unregularized game:\\[3pt]
    $\varphi_{\mathrm{RA}}(\bar\pi^s)=\mathcal{O}(1/s)$.
};

\draw[arrow] (existence.south -| entropy.north)
    -- node[edge label]{add entropy regularization} (entropy.north);

\draw[arrow] (existence.south -| fp.north)
    -- node[edge label]{use fictitious play} (fp.north);

% Contractivity: two ambiguity-radius regimes.
\node[regularized,text width=3.7cm,minimum height=2cm,anchor=north]
    (full) at (-5.45,-6.8) {
    \textbf{Theorem \ref{thm:contractivity:regularized:BR}: $\eps=1$}\\[4pt]
    Under Asmp. \ref{assum:eps:one:case}\\
    $\operatorname{Br}_\beta$ is contractive.
};

\node[regularized,text width=3.7cm,minimum height=2cm,anchor=north]
    (partial) at (0,-6.8) {
    \textbf{Theorem \ref{thm:contractivity:inner:regularized:BR}: $0<\eps<1$}\\[4pt]
    Under Asmp. \ref{assum:eps:one:case} + \ref{assum:eps:0:1}.\\
    $\operatorname{Br}_\beta^\sigma$ and $\operatorname{Br}_\beta$ are contractive\\
    uniformly in $\sigma>0$.\\[4pt]
    As $\sigma\downarrow0$, $o_{\beta,\sigma}^{\star}\to o_\beta^{\star}$.
};

% Draw the common stem separately to keep the fork connected.
\coordinate (split) at (-2.725,-6.2);
\draw[connector] (entropy.south) -- (split);
\draw[arrow] (split) -| (full.north);
\draw[arrow] (split) -| (partial.north);

% Approximation follows from Theorem 2, without contractivity.
\node[approximation,text width=4.2cm,minimum height=3.8cm,anchor=north]
    (limit) at (5.6,-7.45) {
    \textbf{Proposition \ref{thm:approx:vanishing:entropic:RA:MP:MFE}: approximation}\\[4pt]
    Every $\eps\in[0,1]$; $\beta>0$.\\
    Original-game exploitability:\\[3pt]
    $\phi_{\mathrm{RA}}(\pi_\beta^\star)\le M\beta T\log|\mathcal{U}|$.\\[6pt]
    Limit $\beta\downarrow0$:\\
    every accumulation point of\\
    $o_\beta^\star$ induces an RA-MP-MFE.
};

\draw[arrow,draw=violet!65!black]
    (entropy.east) -- (2.525,-4.5) |- (limit.west);

% Common consequence of Theorems 3 and 4.
\node[regularized,text width=8.05cm] (convergence) at (-2.725,-11) {
    \textbf{Unique regularized equilibrium and convergence}\\[3pt]
    $\mu^{n+1}=\Gamma_\beta(\mu^n)\longrightarrow\mu_\beta^\star,\qquad \mu^0\in\xi(\mathsf{O})$.
};

\coordinate (merge) at (-2.725,-9.8);
\draw[connector] (full.south) |- (merge);
\draw[connector] (partial.south) |- (merge);
\draw[arrow] (merge) -- (convergence.north);

\end{tikzpicture}%
}
\caption{Our primary theoretical developments.}
\label{fig:theoretical-roadmap}
\end{figure}

\section{Reachable set}\label{sec:reach:set:appen}

Let us recall the definition of the reachable set (this is same as that of Definition \ref{def:reachable-set} in Section \ref{sec:prob:form}).
%%%%%%%% def reach set %%%%%%
\begin{myframe}
\begin{definition}[Reachable set]
\label{def:reachable-set:appen}
Fix \(j\in[M]\), \(t\in\{0,\ldots,T-1\}\), and \(\mu_t^j\in\P(\X)\).
The one-step reachable set of population \(j\) from \(\mu_t^j\) is defined by
\begin{equation}
\label{eqn:reachable-set-def}
\reachset \Let\aset[\Big]{\nxtmf\in\P(\X)\suchthat \exists\,\pi_t^j\in\Pi_t\text{ such that }\mu^+=\mu_t^jF_t^j(\pi_t^j)}.
\end{equation}
\end{definition}
\end{myframe}

Lemma \ref{prop:reach:set:properties:appen} records some structural properties of \(\reachset\), which will be utilized later in our analysis. 

\begin{myOCP}
\begin{lemma}[Structure of \(\reachset\)]\label{prop:reach:set:properties:appen}
Fix \(j\in[M]\), \(t\in\{0,\ldots,T-1\}\), and \(\mu^j_t\in\P(\X)\). Then, the following assertions hold:
\begin{enumerate}[label=\textup{(\ref{prop:reach:set:properties:appen}-\alph*)}, leftmargin=*, widest=b, align=left]

\item \label{prop:reach:set:properties:I} The reachable set \(\reachset\) defined in \eqref{eqn:reachable-set-def} is nonempty, compact, convex, and polyhedral. In particular, \(\reachset\) is a polytope contained in \(\P(\X)\). 

\item \label{prop:reach:set:properties:II} \(\reachset\) admits the representation
\begin{align}\label{eq:adm:reach:set}
   \reachset = \left\{\nxtmf \;\middle\vert\;  
    \begin{array}{@{}l@{}}
       \exists\, \eta:\X \times \U \to \Real_{+} \text{ such that } \sum_{u\in \U}\eta(x,u)=\mu^j_t(x) \\ \text{ for all }x\in \X,\,
       \nxtmf(y)=\sum_{x\in \X}\sum_{u\in \U}\eta(x,u)f^j_t(y|x,u) \\ \text{ for all } y\in \X
    \end{array}
    \right\},
\end{align}  
where \(\nxtmf \in\P(\X)\). 
\end{enumerate}

\end{lemma}
\end{myOCP}

% ========================================
% ========= Prop: Reachable set ==========
% ========================================

\begin{proof}
Fix \(j \in [M]\), \(t \in \aset[]{0,\ldots,T-1}\), and \(\mu^j_t \in \P(\X)\).
\edit{
Recall that \(\Pi_t\) is the set of randomized Markov decision rules at time \(t\)
\begin{align*}
\Pi_t \Let \aset[\Big]{\pi_t:\X\times\U\to[0,1]\suchthat \sum_{u\in\U}\pi_t(u|x)=1\text{ for every }x\in\X}.
\end{align*}
An element of this set assigns an action distribution \(\pi_t(\cdot|x)\in\P(\U)\) to each state \(x\in\X\). Thus \(\Pi_t\) can be identified with \(\prod_{x\in\X}\P(\U)\).
}
Since \(\X\) and \(\U\) are finite and nonempty, \(\Pi_t\) is nonempty. 
For \(\pi^j_t\in\Pi_t\), denote the next mean-field state by \(\nxtmf \Let \mu^j_tF^j_t(\pi^j_t)\), so that
\begin{align}
\X \ni y \mapsto \nxtmf(y)=\sum_{x\in \X}\sum_{u\in \U}\mu^j_t(x)\pi^j_t(u|x)f^j_t(y|x,u). \nn
\end{align}
It follows that \(\nxtmf(y)\ge 0\) for all \(y\in\X\), and 
\begin{align*}\sum_{y\in \X}\nxtmf(y)=\sum_{x\in \X}\sum_{u\in \U}\mu^j_t(x)\pi^j_t(u|x)\sum_{y\in \X}f^j_t(y|x,u)=\sum_{x\in\X}\mu^j_t(x)\sum_{u\in\U}\pi^j_t(u|x)=1.\end{align*}
Hence, \(\nxtmf\in\P(\X)\), and thus \(\reachset\neq\varnothing\).

Finiteness of \(\X\) and \(\U\) implies that \(\Pi_t\) is a compact and convex polytope. 
The map \(\Pi_t \ni \pi^j_t\mapsto\mu^j_tF^j_t(\pi^j_t) \in \P(\X)\) is continuous and affine.
Therefore, its image \(\reachset\) is compact \cite[Ch. 3]{ref:Rud-Analysis} and convex \cite[Sec. 2]{ref:SB:LV-ConvOpt}.

We now show that \(\reachset\) has the representation \eqref{eq:adm:reach:set} and it admits a polyhedral structure. 
Let \(\nxtmf\in\reachset\).
The following holds:
\begin{enumerate}[label=\(\circ\),leftmargin=*,widest=ii,align=left]

\item By definition, there exists \(\pi^j_t\in\Pi_t\) such that \(\nxtmf=\mu^j_tF^j_t(\pi^j_t)\).

\item Define \(\eta:\X\times\U\to\mathbb R_+\) by \((x,u) \mapsto \eta(x,u)\Let \mu^j_t(x)\pi^j_t(u|x)\). 
Then, \(\sum_{u\in\U}\eta(x,u)=\mu^j_t(x)\sum_{u\in\U}\pi^j_t(u|x)=\mu^j_t(x)\) for all \(x\in\X\), and \(\nxtmf(y)=\sum_{x\in\X}\sum_{u\in\U}\eta(x,u)f^j_t(y|x,u)\) for all \(y\in\X\). 

\item Conversely, suppose that \(\eta(\cdot)\) satisfies \(\sum_{u\in\U}\eta(x,u)=\mu^j_t(x)\) for all \(x\in\X\). 
If \(\mu^j_t(x)>0\), define \(\pi^j_t(u|x)\Let \eta(x,u)/\mu^j_t(x)\) for \(u\in\U\).
If \(\mu^j_t(x)=0\), then \(\sum_{u\in\U}\eta(x,u)=0\), and since \(\eta(x,u)\ge0\), we have \(\eta(x,u)=0\) for all \(u\in\U\); in this case choose \(\pi^j_t(\cdot|x)\) arbitrarily in \(\P(\U)\).
Thus, \(\pi^j_t\in\Pi_t\). 
Moreover, for every \(y\in\X\), 
\begin{align*}
\nxtmf(y)=\sum_{x\in\X}\sum_{u\in\U}\eta(x,u)f^j_t(y|x,u)=\sum_{x\in\X}\sum_{u\in\U}\mu^j_t(x)\pi^j_t(u|x)f^j_t(y|x,u),
\end{align*}
and hence \(\nxtmf=\mu^j_tF^j_t(\pi^j_t)\). 
\end{enumerate}
Therefore, \(\nxtmf\in\reachset\), proving \eqref{eq:adm:reach:set}.
Finally, note that, the feasible set for \(\eta\) in \eqref{eq:adm:reach:set} is described by finitely many linear equalities and inequalities.
It is also bounded, since 
\begin{align*}
0\le\eta(x,u)\le\sum_{u'\in\U}\eta(x,u')=\mu^j_t(x)\le 1. 
\end{align*}
Hence, it is a polytope. The map \(\eta\mapsto \sum_{x\in\X}\sum_{u\in\U}\eta(x,u)f^j_t(y|x,u)\) is linear.
Therefore, \(\reachset\) is a linear image of a polytope, and hence is itself a polytope.
Our proof is complete.
\end{proof}

\section{Ambiguity set}\label{sec:amb:set:appen}

Recall the definition of the ambiguity set from Section \ref{subsec:risk:aversion}
\begin{align}\label{eq:ambiguity:set:appen}
\mathcal{W}_j(\mu,\eps)\Let \left\{
    \begin{gathered}
    \{\hat\mu_t^{k}\}_{t=0}^T, \\
    \forall k\in [M],\ k\neq j
    \end{gathered}
    \;\middle\vert\;
    \begin{array}{@{}l@{}}
    \hat\mu_0^{k}=\mu_0^{k}, \,
    \hat\mu_{t+1}^{k}\in \mathcal{R}_t^k(\hat\mu_t^k), \\
    \dtv{\mu_t^k,\hat\mu_t^k}\le \eps 
    \end{array}
    \right\}.
\end{align}

% ========================================
% ========= Prop: ambiguity set ==========
% ========================================

We now restate and provide a detailed proof of Proposition \ref{prop:ambiguity:set} from Section \ref{subsec:risk:aversion}.

\begin{myOCP}
\begin{proposition}[Structure of the ambiguity set]
\label{prop:ambiguity:set:appen}
Fix \(j\in[M]\), a nominal flow \(\mu \Let \{\mu_t\}_{t=0}^T\) with  \(\mu_t \Let (\mu_t^1,\ldots,\mu_t^M)\in \P(\X)^M,\) and let \(\eps \in [0,1]\). 
Then: 
\begin{enumerate}[label=\textup{(\ref{prop:ambiguity:set}-\alph*)}, leftmargin=*, widest=b, align=left]

\item \(\ambset(\mu,\eps)\) is a compact polyhedral subset of
\(\prod_{k\neq j}\prod_{t=0}^T \P(\X)\).

\end{enumerate}

\noindent
If, in addition, for every \(k\neq j\), the nominal flow
\(\mu^k=\{\mu_t^k\}_{t=0}^T\) is \emph{dynamically feasible}, i.e., there exists
a policy sequence \(\{\pi_t^k\}_{t=0}^{T-1}\) such that
\(\mu_{t+1}^k=\mu_t^kF_t^k(\pi_t^k)\) for \(t=0,\ldots,T-1\), then,

\begin{enumerate}[label=\textup{(\ref{prop:ambiguity:set}-\alph*)}, leftmargin=*, widest=b, align=left]
\setcounter{enumi}{1}

\item \(\ambset(\mu,\eps)\neq\varnothing\). In that case,
\(\ambset(\mu,\eps)\) is a nonempty compact polytope.

\end{enumerate}
\end{proposition}
\end{myOCP}

\begin{proof}
We first show that \(\ambset(\mu,\eps)\) is compact and admits a polyhedral nature. 
By Lemma \ref{prop:reach:set:properties:appen}, for each \(k \neq j\) and \(t \in \aset[]{0,\ldots,T-1}\), the reachability condition \(\muhat{t+1} \in \R_t^k\bigl(\muhat{t}\bigr)\) is equivalent to the existence of the mapping \(\eta_t^k : \X \times \U \to \mathbb{R}_{+}\)
such that \(\sum_{u \in \U}\eta_t^k(x,u) = \muhat{t}(x)\) for all \(x \in \X\) and \(\muhat{t+1}(y) = \sum_{x \in X} \sum_{u \in \U} \eta_t^k(x,u)f_t^k(y|x,u)\) for all \(y \in \X\).
To show compactness of \(\ambset(\mu,\eps)\), we first construct a larger set which we will show is compact. Then, we obtain \(\ambset(\mu,\eps)\) as a continuous projection. 

To this end, let \(\ambsetnew(\mu,\eps)\) be the set of all variables \((\hat\mu^{-j},\eta^{-j},s^{-j})\) satisfying the following finite family of constraints
\begin{align}
&\hat\mu_t^k(x)\ge 0,\quad \sum_{x\in\X}\hat\mu_t^k(x)=1,\quad k\neq j,\ t=0,\ldots,T \text{ for }x\in\X,\,\,  \hat\mu_0^k=\mu_0^k \text{ for }k\neq j,\nn\\
% comes from the reachability constraints
&\eta_t^k(x,u)\ge 0,\quad \sum_{u\in\U}\eta_t^k(x,u)=\hat\mu_t^k(x) \text{ for }k\neq j,\ t=0,\ldots,T-1,\ x\in\X, \text{ and }u\in\U, \nn\\
% comes from the reachability constraints
&\hat\mu_{t+1}^k(y)=\sum_{x\in\X}\sum_{u\in\U}\eta_t^k(x,u)f_t^k(y|x,u) \text{ for } k\neq j,\ t=0,\ldots,T-1, \text{ and }y\in\X, \nn\\
% comes from the reachability constraints
&s_t^k(x)\ge 0,\,\, s_t^k(x)\ge \hat\mu_t^k(x)-\mu_t^k(x),\,\,s_t^k(x)\ge \mu_t^k(x)-\hat\mu_t^k(x), 
\sum_{x\in\X} \tfrac{s_t^k(x)}{2}\le\eps, \nn
% comes from the TV constraints 
\end{align}
for \(k\neq j\), \(t=0,\ldots,T\), and \(x\in\X\).
The auxiliary variables \(x \mapsto s_t^k(x)\) encode the total-variation constraints.
Indeed, if \(\dtv{\mu_t^k,\hat\mu_t^k}\le\eps\), one may choose \(x \mapsto s_t^k(x) \Let \abs{\mu_t^k(x)-\hat\mu_t^k(x)}\). 
Conversely, if such \(s_t^k(x)\) exist, then \(\abs{\mu_t^k(x)-\hat\mu_t^k(x)}\le s_t^k(x)\) for every \(x\in\X\), and therefore
\begin{align*}
\dtv{\mu_t^k,\hat\mu_t^k}=\frac{1}{2}\sum_{x\in\X}\abs{\mu_t^k(x) -\hat\mu_t^k(x)}\le\frac{1}{2}\sum_{x\in\X}s_t^k(x)\le\eps.
\end{align*}
All constraints defining \(\ambsetnew(\mu,\eps)\) are linear equalities or linear inequalities in the finite-dimensional variables \((\hat\mu^{-j},\eta^{-j},s^{-j})\).
Therefore \(\ambsetnew(\mu,\eps)\) is a polyhedron.
Moreover, \(\ambsetnew(\mu,\eps)\) is bounded. Indeed, \(0\le \hat\mu_t^k(x)\le1\), \(0\le \eta_t^k(x,u)\le\sum_{u'\in\U}\eta_t^k(x,u')=\hat\mu_t^k(x)\le1\), and \(0\le s_t^k(x)\le2\eps\).
\newedit{Hence,} \(\ambsetnew(\mu,\eps)\) is a bounded polyhedron, and therefore compact. If it is nonempty, it is a polytope.

%%% showing W is compact %%%%%
By Lemma~\ref{prop:reach:set:properties:appen} and the auxiliary-variable representation of the total-variation constraints, \(\ambset(\mu,\eps)\) is the projection of \(\ambsetnew(\mu,\eps)\) onto the variables \((\hat\mu^{-j},\eta^{-j},s^{-j}) \mapsto \hat\mu^{-j}\).
Since projection is a linear (and thus continuous) map, the projection of a compact set is compact. 
Moreover, the projection of a polyhedron is polyhedral.
\newedit{Therefore,} \(\ambset(\mu,\eps)\) is a compact polyhedral subset of \(\prod_{k\neq j}\prod_{t=0}^T \P(\X)\).

%%%% nonemptyness %%%%%%
We now prove the non-emptiness of \(\ambset(\mu,\eps)\).
Suppose that, for every \(k\neq j\), there exists a policy sequence \(\{\pi_t^k\}_{t=0}^{T-1}\) such that \(\mu_{t+1}^k=\mu_t^kF_t^k(\pi_t^k)\) for \(t=0,\ldots,T-1\). Define \(\hat\mu_t^k\Let \mu_t^k\) for \(k\neq j\) and \(t=0,\ldots,T\).
Then \(\hat\mu_0^k=\mu_0^k\), and \(\hat\mu_{t+1}^k=\mu_{t+1}^k=\mu_t^kF_t^k(\pi_t^k)=\hat\mu_t^kF_t^k(\pi_t^k)\), so \(\hat\mu_{t+1}^k\in\R_t^k(\hat\mu_t^k)\).
Finally,
\begin{align}
  \dtv{\mu_t^k,\hat\mu_t^k}=\dtv{\mu_t^k,\mu_t^k}=0\le\eps. \nn  
\end{align}
Therefore \(\hat\mu^{-j}=\mu^{-j}\) belongs to \(\ambset(\mu,\eps)\), and hence \(\ambset(\mu,\eps)\neq\emptyset\).
Since \(\ambset(\mu,\eps)\) has already been shown to be compact and polyhedral, it follows in the nonempty case that \(\ambset(\mu,\eps)\) is a nonempty compact polytope.
\end{proof}

\section{Proof of Theorem \ref{thm:existence:RA:MP:MFE}}\label{sec:existence:normal:appen}

In this section we provide detailed proof of Theorem \ref{thm:existence:RA:MP:MFE}; we restate it for convinience.
% ========================================
% ========= Thrm: Existence Eq. ==========
% ========================================

\begin{myOCP}
\begin{theorem}[Existence of RA-MP-MFE]\label{thm:existence:RA:MP:MFE:appen}
Recall the problem data from Section \ref{sec:prob:form} and let \(\mu_0^j\in\P(\X)\) be fixed for all \(j\in[M]\). 
Let Assumption \ref{assmpt:lipschitiz-model} hold.  Then, for every fixed \(\eps \in \lcrc{0}{1}\), there exists a RA-MP-MFE \((\pi^\star,\mu^\star)\in\Pi\times\M\). 
More precisely, there exists \(\pi^\star\) such that, \(\mu^\star \Let \B_\prop(\pi^\star)\), for every \(j\in[M]\), and
\begin{equation}\label{eq:existence:RA:MP:MFE}
\pi^{j,\star}\in \argmax_{\pi^j\in\Pi^j}\inf_{\widehat\mu^{-j}\in\ambset(\mu^\star,\eps)}\E^{\pi^j}\left[\sum_{t=0}^{T}r_t^j\bigl(x_t^j,\mu_t^{j,\star},\widehat\mu_t^{-j}\bigr)\right].  
\end{equation}
\end{theorem}    
\end{myOCP}

\begin{remark}[Proof roadmap for Theorem 1]\label{rem:proof:roadmap:RA:MP:MFE}
The main ingredient of our proof consists of techniques from set-valued and variational analysis \cite{ref:AliBor-06,ref:ZeidlerFixed-86,ref:DonRoc-14}. Our proof for Theorem \ref{thm:existence:RA:MP:MFE} roughly proceeds via the following steps (see also Fig.~\ref{fig:proof:raodmap}): \textbf{(a)} First, for each population, we introduce the finite-horizon occupation-measure set and use the equivalence between Markov policies and occupation measures to reformulate the risk-averse best-response problem on a nonempty compact convex set.
%%%%%%%% a TikZ picture %%%%%%%%%%
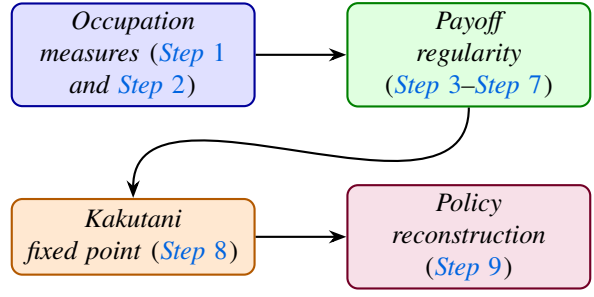
\begin{figure}[h]
\begin{center}
\begin{tikzpicture}[
    >={Stealth},
    box/.style={
        draw,
        thick,
        rounded corners=4pt,
        align=center,
        font=\small,
        text width=3cm,
        minimum height=1cm,
        inner sep=3pt
    }
]

\node[box, fill=blue!12,   draw=blue!60!black]   (a) {Occupation\\measures (\ref{step:one} and  \ref{step:two})};
\node[box, fill=green!12,  draw=green!50!black,  right=1.2cm of a] (b) {Payoff\\regularity (\ref{step:three}--\ref{step:seven})};

\node[box, fill=orange!18, draw=orange!70!black, below=1.2cm of a] (c) {Kakutani\\fixed point (\ref{step:eight})};
\node[box, fill=purple!12, draw=purple!60!black, right=1.2cm of c] (d) {Policy\\reconstruction (\ref{step:nine})};

\draw[->, thick] (a) -- (b);
\draw[->, thick] (b.south) to[out=-90,in=90] (c.north);
\draw[->, thick] (c) -- (d);

\end{tikzpicture}
\end{center}
\caption{A roadmap of the proof steps for Theorem \ref{thm:existence:RA:MP:MFE}.}
\label{fig:proof:raodmap}
\end{figure}
%%%%%%%% a TikZ picture %%%%%%%%%%
\textbf{(b)} The main technical step is to establish the required regularity of the risk-averse payoff. To this end, we show that the ambiguity-set correspondence has closed graph and is both upper and lower hemicontinuous. Its continuity, together with the continuity of the reward functions, allows us to apply Berge's maximum principle~\cite[Chapter 17, Theorem 17.31]{ref:AliBor-06} and conclude that the occupation-measure payoff is jointly continuous. We also show that the payoff is concave in the candidate occupation measure. These properties imply that the resulting occupation-measure best-response correspondence has nonempty, compact, and convex values and is upper hemicontinuous. \textbf{(c)} Kakutani's fixed-point theorem~\cite[Chapter 17, Corollary 17.55]{ref:AliBor-06} therefore yields a fixed occupation-measure profile. \textbf{(d)} Finally, we reconstruct Markov policies from this fixed point and verify that the resulting policy and mean-field pair satisfies both the risk-averse optimality and mean-field consistency conditions.
\end{remark}

\begin{proof}[\highlight{Proof of Theorem \ref{thm:existence:RA:MP:MFE:appen}}]
We write the proof in several steps as outlined in Remark~\ref{rem:proof:roadmap:RA:MP:MFE}.

%%%%%%%%%%%%%%%%%%%%%%%%%%%%%%%%
%%%%%%%%%% step one %%%%%%%%%%%%
%%%%%%%%%%%%%%%%%%%%%%%%%%%%%%%%

\noindent\proofstep{step:one}{Occupation measures} For each population \(j\in[M]\), let \(\Occ^j\) denote the set of occupation measures \(o^j \Let (\xi^j,\eta^j)\), where \(\xi^j\Let\{\xi_t^j\}_{t=0}^T\) and \(\eta^j\Let\{\eta_t^j\}_{t=0}^{T-1}\), satisfying the properties \eqref{occ:constraint:1}--\eqref{occ:constraint:2}.\footnote{As usual, \(\xi_t^j = \probmeas(X_t^j =x)\) and \(\eta_t^j(x,u) = \probmeas(X_t^j =x,U_t^j =u)\).}
%
% \begin{align}
% \xi_0^j=\mu_0^j,\qquad \sum_{u\in\U}\eta_t^j(x,u)=\xi_t^j(x)\quad \text{for all }\,x\in\X,\ t=0,\ldots,T-1,\label{occ:constraint:1}\\
% m_{t+1}^j(y)=\sum_{x\in\X}\sum_{u\in\U}\eta_t^j(x,u)f_t^j(y|x,u)\quad \text{for all }\,y\in\X,\ t=0,\ldots,T-1.\label{occ:constraint:2}
% \end{align}
%
By the finite-dimensional linear representation used in Lemma~\ref{prop:reach:set:properties:appen}, each \(\Occ^j\) is a nonempty compact convex polytope. Hence \(\Occ\Let\prod_{j=1}^M\Occ^j\) is a nonempty compact convex subset of a finite-dimensional Euclidean space \cite[Section 2.3]{ref:SB:LV-ConvOpt}.

%%%%%%%%%%%%%%%%%%%%%%%%%%%%%%%%
%%%%%%%%%% step two %%%%%%%%%%%%
%%%%%%%%%%%%%%%%%%%%%%%%%%%%%%%%

\noindent\proofstep{step:two}{Structural properties of sets}
We recall the standard equivalence between Markov policies and occupation measures \cite{ref:YG:AA:LH:KHJ:Dist:Reach:MDP:TAC}. 
Every policy \(\pi^j\in\Pi^j\) induces \(o^j=(\xi^j,\eta^j)\in\Occ^j\) by \(\eta_t^j(x,u)=\xi_t^j(x)\pi_t^j(u|x)\), where \(\xi^j\) is the state-flow generated by \(\pi^j\).
Conversely, every \(o^j=(\xi^j,\eta^j)\in\Occ^j\) induces a Markov policy by
\begin{equation}\label{eq:markov:policy}
\X  \ni x \mapsto \pi_t^j(u|x)\Let \frac{\eta_t^j(x,u)}{\xi_t^j(x)}\quad \text{if }\,\xi_t^j(x)>0,
\end{equation}
and by choosing \(\pi_t^j(\cdot|x)\) arbitrarily in \(\P(\U)\) if \(\xi_t^j(x)=0\). This policy generates the same occupation measure and the same state-flow \cite{ref:YG:AA:LH:KHJ:Dist:Reach:MDP:TAC}.

For \(o \Let (o^1,\ldots,o^M)\in\Occ\), denote the corresponding multi-population flow by \(\xi(o)\Let(\xi^1,\ldots,\xi^M)\).
For each \(j\in[M]\), define the payoff of a candidate occupation measure \(\tilde{o}^j \Let (\tilde{\xi}^j,\tilde{\eta}^j)\in\Occ^j\) against the nominal occupation profile \(o\) by 
\begin{equation}\label{robust:occ:payoff}
\Occ^j \times \Occ \ni \bigl(\tilde{o}^j;o\bigr) \mapsto \Phi_j(\tilde{o}^j;o)\Let \inf_{\widehat\mu^{-j}\in \ambset(\xi(o),\eps)}\sum_{t=0}^T\sum_{x\in\X}\tilde \xi_t^j(x)r_t^j(x,\xi_t^j,\widehat\mu_t^{-j}).
\end{equation}
Since \(\xi(o)\) is dynamically feasible by construction, Proposition~\ref{prop:ambiguity:set} implies that \(\mathcal{W}_j(\xi(o),\eps)\) is nonempty and compact.

%%%%%%%%%%%%%%%%%%%%%%%%%%%%%%%%
%%%%%%%%%% step three %%%%%%%%%%
%%%%%%%%%%%%%%%%%%%%%%%%%%%%%%%%

\noindent\proofstep{step:three}{Regularity of $\Phi_j : \Occ^j \times \Occ \to \mathbb{R}$} We need two properties of the function \(\Phi_j(\cdot)\):
\begin{enumerate}[leftmargin=*]
    \item the joint continuity of \(\Occ^j \times \Occ \ni \bigl(\tilde{o}^j;o\bigr) \mapsto \Phi_j(\tilde{o}^j;o) \in \Real\); and 
    \item concavity of \(\Occ^j \ni \tilde{o}^j \mapsto \Phi_j(\tilde{o}^j;o) \in \Real\) for each fixed \(o \in \Occ\). 
\end{enumerate}
We need the following ingredients to prove them.

%%%%%%%%%%%%%%%%%%%%%%%%%%%%%%%%
%%%%%%%%%% step four %%%%%%%%%%%
%%%%%%%%%%%%%%%%%%%%%%%%%%%%%%%%

\noindent\proofstep{step:four}{\(\mu \mapsto \ambset(\mu,\eps)\) has closed graph \cite[Chapter 5, Problem 5.3]{ref:EberZeidler:AFA:MathPhysc}} Let \(\feasflow\) denote the set of dynamically feasible opponent flows with the prescribed initial distributions, i.e., 
\begin{align}\label{eq:feas:flow:set}
\feasflow \Let \aset[\bigg]{\nu^{-j}=\{\nu_t^k\}_{t=0}^T \suchthat \nu_0^k=\mu_0^k,\ \nu_{t+1}^k\in R_t^k(\nu_t^k)\ \text{for all }\,k\neq j,\ t=0,\ldots,T-1}.
\end{align}
By Proposition~\ref{prop:ambiguity:set}, \(\feasflow\) is a compact convex polytope. 
Moreover, we have the representation
\begin{align}\label{eq:amb:set:feas}
\ambset(\mu,\eps)=\aset[\bigg]{\nu^{-j}\in \feasflow \suchthat \dtv{\nu_t^k,\mu_t^k}\le\eps\ \text{for all }\,k\neq j,\ t=0,\ldots,T}. 
\end{align}
 We first show that \(\mu\mapsto \ambset(\mu,\eps)\), restricted to dynamically feasible nominal flows, has closed graph.
The graph \cite[Chapter 2, Section 2.14]{ref:AliBor-06} of the set-valued map \(\ambset\) is
\begin{align*}
\graph(\ambset)\Let \aset[\big]{(\mu,\nu^{-j}) \suchthat \nu^{-j}\in \ambset(\mu,\eps)}.
\end{align*}
Since \(\ambset\) has a closed graph it follows that, whenever
\begin{equation}
(\mu_n)_{n\in \Nz} \to\mu,\,\, (\nu_n^{-j})_{n\in \Nz}\to\nu^{-j},\,\, (\nu_n^{-j})_{n\in \Nz}\in \ambset(\mu_n,\eps), \,\,\text{we have}\,\, \nu^{-j}\in \ambset(\mu,\eps). \nn
\end{equation}
To this end, let \((\mu_n)_{n\in \Nz} \to \mu\) be a convergent sequence of feasible nominal flows, and let \((\nu_n^{-j})_{n\in \Nz} \to \nu^{-j}\) with \(\nu_n^{-j} \in \mathcal{W}_j(\mu_n,\eps)\) for every \(n \in \Nz\). By \eqref{eq:amb:set:feas}, for every \(n\), we have \(\nu_n^{-j} \in \feasflow\), and since \(\feasflow\) is closed, \(\nu^{-j} \in \feasflow\) (closed sets contains all of its limit points \cite[Chapter 2]{ref:Rud-Analysis}). We also have for every \(k \neq j\) and every \(t \in \aset[]{0,\ldots,T}\), \(\dtv{\nu^k_{n,t}, \mu^k_{n,t}}\le \eps.\) Since the TV distance is continuous over the finite-dimensional simplex, passing to the limit, gives us 
\begin{align}
    \dtv{\nu_t^k,\mu_t^k} \le \eps \quad \text{for all }\, k \neq j \text{ and }\, t \in \aset[]{0,\ldots,T}. \nn
\end{align}
Thus, \(\nu^{-j} \in \ambset (\mu,\eps)\) and the graph \(\mu \mapsto \ambset (\mu,\eps)\) is closed. 

%%%%%%%%%%%%%%%%%%%%%%%%%%%%%%%%
%%%%%%%%%% step five %%%%%%%%%%%
%%%%%%%%%%%%%%%%%%%%%%%%%%%%%%%%

\noindent\proofstep{step:five}{Upper hemicontinuity} We next prove upper hemicontinuity of the correspondence \(\mu\mapsto \ambset(\mu,\eps)\) at an arbitrary dynamically feasible nominal flow \(\mu\).
We use the open-set definition of upper hemicontinuity \cite[Chapter~17, Section~17.2, Definition~17.2]{ref:AliBor-06}, together with the sequential characterization for compact-valued correspondences \cite[Chapter~17, Theorem~17.16]{ref:AliBor-06}.\footnote{We note that the context in \cite[Chapter 17]{ref:AliBor-06} is much more general compared to our setting. All the set-valued mapping there have habitats as topological space and thus the authors employ the notion of `nets', which for our finite-dimensional setting reduces to `sequences'.} 
In the present finite-dimensional setting, it is enough to show that whenever \((\mu_n)_{n \in \Nz}\to\mu\) and \(\nu_n^{-j}\in\ambset(\mu_n,\eps)\) for all \(n\in \Nz\), the sequence \((\nu_n^{-j})_{n \in \Nz}\) has a subsequence converging to some element of \(\ambset(\mu,\eps)\).

Fix an open set \(O\) containing \(\ambset(\mu,\eps)\). Suppose, by contradiction, that \(\mu\mapsto \ambset(\mu,\eps)\) is not upper hemicontinuous at \(\mu\).
Then, there exist a sequence \((\mu_n)_{n \in \Nz}\to\mu\) of dynamically feasible nominal flows and elements \(\nu_n^{-j}\in \ambset(\mu_n,\eps)\) such that
\begin{align*}
\Nz \ni n \mapsto \nu_n^{-j}\notin O.
\end{align*}
Since \(\ambset(\mu_n,\eps)\subseteq\feasflow\) for every \(n\), and since \(\feasflow\) is compact, there exists a subsequence, still denoted by \((\nu_n^{-j})_{n \in \Nz}\), and some \(\nu^{-j}\in\feasflow\), such that \((\nu_n^{-j})_{n \in \Nz}\to\nu^{-j}.\)

By the closed-graph property proved in \ref{step:four}, we see that \(\nu^{-j}\in \ambset(\mu,\eps)\). 
Since \(\ambset(\mu,\eps)\subset O\) and \(O\) is open, we have \(\nu^{-j}\in O\), and hence \(\nu_n^{-j}\in O\) for all sufficiently large \(n\). This contradicts the fact that \(\Nz \ni n \mapsto \nu_n^{-j}\notin O\).
Therefore, \(\mu\mapsto \ambset(\mu,\eps)\) is upper hemicontinuous at \(\mu\).

%%%%%%%%%%%%%%%%%%%%%%%%%%%%%%%%
%%%%%%%%%% step seven %%%%%%%%%%
%%%%%%%%%%%%%%%%%%%%%%%%%%%%%%%%

\noindent\proofstep{step:six}{Lower hemicontinuity} We now prove lower hemicontinuity of the correspondence \(\mu\mapsto \ambset(\mu,\eps)\) \cite[Chapter~17, Section~17.2, Definition~17.2]{ref:AliBor-06}. 
By the sequential characterization of lower hemicontinuity \cite[Theorem~17.19]{ref:AliBor-06}, it suffices to show that whenever \((\mu_n)_{n\in \Nz}\to\mu\) and \(\nu^{-j}\in \ambset(\mu,\eps)\), there exist a subsequence \(\bigl(\mu_{n_\ell}\bigr)_{\ell \in \Nz}\) and points \(\nu_{n_\ell}^{-j}\in \ambset(\mu_{n_\ell},\eps)\) such that \(\bigl(\nu_{n_\ell}^{-j}\bigr)_{\ell \in \Nz}\to\nu^{-j}\).
In fact, we prove the stronger statement by constructing such a sequence along the whole sequence \((\mu_n)_{n\in \Nz}\).
Let \((\mu_n)_{n\in \Nz}\to\mu\) be dynamically feasible nominal flows, and fix \(\nu^{-j}\in \ambset(\mu,\eps)\). 
We construct \(\nu_n^{-j}\in \ambset(\mu_n,\eps)\) such that \((\nu_n^{-j})_{n\in \Nz}\to\nu^{-j}\).
If \(\eps=0\), then \(d_{\TV}(\nu_t^k,\mu_t^k)=0\) for all \(k\neq j\) and \(t=0,\ldots,T\), so \(\nu^{-j}=\mu^{-j}\). Since each \((\mu_n)_{n\in \Nz}\) is dynamically feasible, \(\mu_n^{-j}\in \feasflow\). Thus, choosing \(\nu_n^{-j}\Let \mu_n^{-j}\) gives \(\nu_n^{-j}\in \ambset(\mu_n,0)\) and \((\nu_n^{-j})_{n\in \Nz}\to\nu^{-j}\).

Now suppose \(\eps>0\). Define
\begin{align*}
\Nz \ni n \mapsto \Delta_n \Let \max_{k\neq j,\ 0\le t\le T}d_{\TV}(\nu_t^k,\mu_{n,t}^k).
\end{align*}
Since \((\mu_n)_{n\in \Nz}\to\mu\) and \(\nu^{-j}\in \ambset(\mu,\eps)\), we have \(\limsup_{n\to+\infty}\Delta_n\le\eps\).
If \(\Delta_n\le\eps\), set \(\lambda_n\Let 0\). If \(\Delta_n>\eps\), set \(\lambda_n\Let 1-\eps/\Delta_n\). Then, \(0\le\lambda_n<1\) and \(\lambda_n\to 0\) for all \(n\in \Nz\). 
We now employ an interpolation argument which in spirit is analogous to the convex-mixture construction in \cite[Proposition 3.1]{ref:neufeld2023markov}. To this end, define the sequence
\begin{align*}
\Nz \ni n \mapsto \nu_n^{-j} \Let (1-\lambda_n)\nu^{-j}+\lambda_n\mu_n^{-j}.
\end{align*}
Because \(\feasflow\) is convex and both \(\nu^{-j}\) and \(\mu_n^{-j}\) belong to \(\feasflow\), we have \(\nu_n^{-j}\in \feasflow\).
Moreover, for every \(k\neq j\) and \(t=0,\ldots,T\),
\begin{align*}
d_{\TV}(\nu_{n,t}^k,\mu_{n,t}^k)
&=d_{\TV}\bigl((1-\lambda_n)\nu_t^k+\lambda_n\mu_{n,t}^k,\mu_{n,t}^k\bigr)\\
&=(1-\lambda_n)d_{\TV}(\nu_t^k,\mu_{n,t}^k)\le (1-\lambda_n)\Delta_n \le\eps.
\end{align*}
Hence, \(\nu_n^{-j}\in \ambset(\mu_n,\eps)\). Moreover, we know that, \(\nu_n^{-j}-\nu^{-j}=\lambda_n(\mu_n^{-j}-\nu^{-j}).\)
Since \(\mu_n\to\mu\), the sequence \(\{\mu_n^{-j}\}_{n\in\mathbb{N}}\) is bounded. 
Therefore, because \(\lambda_n\to 0\), we get \(\nu_n^{-j}\to\nu^{-j}.\) 
Thus, for every sequence \((\mu_n)_{n\in \Nz}\to\mu\) and every \(\nu^{-j}\in\ambset(\mu,\eps)\), we have constructed \(\nu_n^{-j}\in\ambset(\mu_n,\eps)\) such that \((\nu_n^{-j})_{n\in \Nz}\to\nu^{-j}\).
Hence, \(\mu\mapsto\ambset(\mu,\eps)\) is lower hemicontinuous.

%%%%%%%%%%%%%%%%%%%%%%%%%%%%%%%%
%%%%%%%%%% step eight %%%%%%%%%%
%%%%%%%%%%%%%%%%%%%%%%%%%%%%%%%%

\noindent\proofstep{step:seven}{Continuity of \(\Phi_j(\cdot)\)} Combining upper hemicontinuity and lower hemicontinuity, the correspondence \(\mu\mapsto\ambset(\mu,\eps)\) is continuous \cite[Chapter 17, Section 17.2, Definition 17.2]{ref:AliBor-06}. Moreover, by Proposition~\ref{prop:ambiguity:set}, \(\ambset(\mu,\eps)\) is nonempty and compact for every dynamically feasible nominal flow \(\mu\). 

We now have the required ingredients; let us return to the regularity of \(\Phi_j(\cdot)\), as asserted in \ref{step:three}.
Recall the definitions of \(\tilde{o}^j\) and \(o\) and define
\begin{align}\label{eq:g:obj:func}
\bigl(\tilde{o}^j,o,\widehat\mu^{-j}\bigr) \mapsto \mathcal{G}_j\bigl(\tilde{o}^j,o,\widehat\mu^{-j}\bigr) \Let
\sum_{t=0}^T\sum_{x\in\X}\tilde \xi_t^j(x)r_t^j(x,\xi_t^j,\widehat\mu_t^{-j}).
\end{align}
Since the reward functions are Lipschitz continuous in the mean-field argument and the horizon and state space are finite, \((\tilde{o}^j,o,\widehat\mu^{-j}) \mapsto \mathcal{G}_j\bigl(\tilde{o}^j,o,\widehat\mu^{-j}\bigr)\) is continuous. Moreover,
\begin{align*}
\Phi_j(\tilde{o}^j;o) = \min_{\widehat\mu^{-j}\in\ambset(\xi(o),\eps)}
\mathcal{G}_j(\tilde{o}^j,o,\widehat\mu^{-j}),
\end{align*}
where the `infimum' is a `min' because \(\ambset(\xi(o),\eps)\) is compact and \(\mathcal{G}_j(\cdot)\) is continuous. 
Since \(o\mapsto \xi(o)\) is continuous and \(\mu\mapsto\ambset(\mu,\eps)\) is a nonempty compact-valued continuous correspondence, Berge's maximum theorem implies that \(\Phi_j(\cdot)\) is continuous on \(\Occ^j\times\Occ\); see \cite[Chapter 17]{ref:AliBor-06}.

%%%%%%%%%%%%%%%%%%%%%%%%%%%%%%%%
%%%%%%%%%% step nenw %%%%%%%%%%%
%%%%%%%%%%%%%%%%%%%%%%%%%%%%%%%%

\noindent\proofstep{step:eight}{Existence via the Kakutani fixed-point theorem} We first show that, for a fixed \(o \in \Occ\), the mapping \(\tilde{o}^j \mapsto \Phi_j(\tilde{o}^j;o)\) is concave.
Indeed, for every fixed \(\widehat{\mu}^{-j} \in \ambset(\xi(o),\eps)\), the map \(\tilde{o}^j \mapsto \mathcal{G}\bigl(\tilde{o}^j ,o,\widehat{\mu}^{-j}\bigr)\) is affine in \(\tilde{\xi}^j\). 
Thus, \(\Phi_j(\cdot;o)\), being the pointwise infimum of affine mappings, is concave \cite{ref:Roc:ConvAnal-70}. 

Define the \emph{best response} correspondence \(\bestresp_j: \Occ \rightrightarrows \Occ^j\) by
\begin{align*}
\bestresp_j(o)\Let \argmax_{\tilde{o}^j\in\Occ^j}\Phi_j(\tilde{o}^j;o).
\end{align*}
Since \(\Occ^j\) is compact and \(\Phi_j(\cdot)\) is continuous, Weierstrass' theorem implies that \(\bestresp_j(o)\) is nonempty and compact for every \(o\in\Occ\).
Since \(\Phi_j(\cdot;o)\) is concave and \(\Occ^j\) is convex, the set \(\bestresp_j(o)\) is convex. 
Moreover, since the feasible set \(\Occ^j\) is fixed and compact and the objective \(\Phi_j(\cdot)\) is continuous, Berge's maximum theorem implies that \(\bestresp_j\) is upper hemicontinuous; see \cite[Chapter~17]{ref:AliBor-06}.

Define the product correspondence \(\bestresp:\Occ\rightrightarrows\Occ\) by \(\bestresp(o)\Let\prod_{j=1}^M\bestresp_j(o)\). 
Because \(M\) is finite, \(\bestresp\) has nonempty, compact, convex values and is upper hemicontinuous.
Since \(\Occ\) is a nonempty compact convex subset of a finite-dimensional Euclidean space, Kakutani's fixed-point theorem \cite[Chapter 17, Corollary 17.55]{ref:AliBor-06} implies that there exists \(o^\star\in\Occ\) such that \(o^\star\in\BR(o^\star).\)

%%%%%%%%%%%%%%%%%%%%%%%%%%%%%%%%
%%%%%%%%%% step ten %%%%%%%%%%%%
%%%%%%%%%%%%%%%%%%%%%%%%%%%%%%%%

\noindent\proofstep{step:nine}{Consistency and optimality} Finally, we verify the conditions (i) and (ii) in Definition \ref{def:RA:MP:MFE}:
\begin{enumerate}[leftmargin=*, label=\(\circ\)]
\item Let \(o^{j,\star} \Let (\xi^{j,\star},\eta^{j,\star})\), and define \(\mu^\star\Let(\xi^{1,\star},\ldots,\xi^{M,\star})\). 
Following \eqref{eq:markov:policy}, construct \(\pi^{j,\star}\in\Pi^j\) from \(o^{j,\star}\) by setting \(\pi_t^{j,\star}(u|x)\Let \eta_t^{j,\star}(x,u)/\xi_t^{j,\star}(x)\) whenever \(\xi_t^{j,\star}(x)>0\), and choosing \(\pi_t^{j,\star}(\cdot|x)\) arbitrarily in \(\P(\U)\) whenever \(\xi_t^{j,\star}(x)=0\). 
By construction, this policy generates the state-flow \(\xi^{j,\star}\). Hence, we have \(\mu^{j,\star}=\B_\prop^j(\pi^{j,\star})\) for all \(j\in[M]\), and thus the condition in Definition~\ref{def:RA:MP:MFE} holds.

\item Fix \(j\in[M]\) and take any \(\pi^j\in\Pi^j\).
Let \(\tilde{o}^j\in\Occ^j\) be the occupation measure induced by \(\pi^j\). 
Since \(o^\star\in \bestresp(o^\star)\), we have \(o^{j,\star}\in\bestresp_j(o^\star)\). Therefore \(\Phi_j(o^{j,\star};o^\star)\ge \Phi_j(\tilde{o}^j;o^\star).\)
By the equivalence between policies and occupation measures, this inequality implies
\begin{align*}
\costj(\pi^{j,\star};\mu^\star)\ge \costj(\pi^j;\mu^\star).
\end{align*}
Since \(j\in[M]\) and \(\pi^j\in\Pi^j\) were arbitrary, the risk-averse optimality condition holds for every population. 
\end{enumerate}
Therefore \((\pi^\star,\mu^\star)\) is a risk-averse multi-population mean-field equilibrium. Our assertion of existence stands established.
\end{proof}

\textbf{Non-contractivity:} As we noted in Section \ref{subsec:risk:aversion}: contractivity, as in many instances in the mean-field literature \cite{cui2021approximately,ref:learning:MFG:Lau,ref:Cianes:Huang}, in general does not follow. The following result is similar in spirit to \cite[Theorem 2]{cui2021approximately}. 
% ========================================
% ========= Thrm: Non Contrtvty ==========
% ========================================

\begin{myOCP}
\begin{proposition}[Non-contractivity of the fixed point operator]\label{thm:non-contractivity:RA:MP:MFG}
Let \(\briskopt:\M\to\Pi\) be a single-valued selector of the risk-averse best-response correspondence, i.e., for each \(\mu\in\M\), let \(\briskopt(\mu)=(\pi^1,\ldots,\pi^M)\in\Pi,\) where, for every \(j\in[M]\),
\begin{align}\label{eq:b-risk-opt:component:RA:MP:MFE}
\pi^j \in \argmax_{\pi^j\in\Pi^j}
\inf_{\widehat\mu^{-j}\in\ambset(\mu,\eps)} \E^{\pi^j}\left[
\sum_{t=0}^{T} r_t^j\bigl(x_t^j,\mu_t^j,\widehat\mu_t^{-j}\bigr)
\right].
\end{align}
Define the mean-field fixed-point map \(\Gamma:\M\to\M\) by
\begin{align}\label{eq:Gamma:RA:MP:MFE}
\M \ni \mu \mapsto\Gamma(\mu)\Let \bprop\bigl(\briskopt(\mu)\bigr).
\end{align}
Suppose that the image \(\Gamma(\M)\) is finite. Then, either \(\Gamma\) is constant, or \(\Gamma\) is not Lipschitz continuous and hence it cannot be a contraction. 
\end{proposition}
\end{myOCP}

% \begin{remark}
% The single-valuedness of \(\briskopt\) is a mild condition \cite{cui2021approximately}.
% %
% It should be understood as fixing a selector of the risk-averse best-response correspondence.
% %
% This \emph{does not} impose uniqueness of the risk-averse optimizer. 
% %
% The condition implies that, for each mean-field flow \(\mu\in\M\), \newedit{a single} admissible risk-averse best response is selected.
% %
% This is the object naturally generated by a fixed-point implementation with a deterministic tie-breaking rule. 
% %
% Theorem~\ref{thm:non-contractivity:RA:MP:MFG} studies the induced mean-field map \(\Gamma\Let\B_\prop\circ\briskopt\) and asserts that, whenever \(\Gamma(\M)\) is finite, the resulting map is either constant or not Lipschitz continuous.
% \end{remark}

\begin{remark}
Proposition~\ref{thm:non-contractivity:RA:MP:MFG} plays an auxiliary role in our analysis: it illustrates why contractivity of the unregularized mean-field map cannot generally be assumed and motivates the use of entropy regularization. Here, the single-valuedness of \(\briskopt\) is understood as fixing a selector of the risk-averse best-response correspondence, for example through deterministic tie-breaking.
\end{remark}

\begin{proof}[\highlight{Proof of Proposition \ref{thm:non-contractivity:RA:MP:MFG}}]
The proof follows the finite-image non-contractivity argument of \cite[Theorem 2]{cui2021approximately}, applied to the mean-field map \(\Gamma\) defined in \eqref{eq:Gamma:RA:MP:MFE}. 
If \(\Gamma\) is constant there is nothing to show. Let \(\Gamma\) be non-constant. Then, there exist \(\mu,\mu' \in \M\) with \(\mu \neq \mu'\) such that \(\Gamma(\mu) \neq \Gamma(\mu')\). 
Since, by assumption, \(\Gamma(\M)\) is finite, the quantity
\begin{align*}
\delta_{\Gamma} \Let
\min_{\substack{\nu,\nu'\in\M\\ \Gamma(\nu)\neq\Gamma(\nu')}}
\dtv{\Gamma(\nu),\Gamma(\nu')},\nn
\end{align*}
is well-defined and \(\delta_{\Gamma}>0\) by definition. Hence, for any \(\nu,\nu' \in \M\) such that \(\Gamma(\nu) \neq \Gamma(\nu')\), we have \(\dtv{\Gamma(\nu),\Gamma(\nu')} \ge \delta_{\Gamma}\).

To show that \(\Gamma\) cannot be Lipschitz continuous, assume, by contradiction, that \(\Gamma\) is Lipschitz with a Lipschitz constant \(L_c>0\). 
Since \(\M\) is convex, for every \(N\in\Nz\) and every \(i\in\{0,\ldots,N\}\), the combination \(\mu^i \Let \frac{i}{N}\mu+\frac{N-i}{N}\mu'\) belongs to \(\M\). 
Moreover, \(\mu_0=\mu'\) and \(\mu^N=\mu,\) and \(\dtv{\mu^i,\mu^{i+1}}=\dtv{\mu,\mu'}/N\) for all \(i\in\{0,\ldots,N-1\}\). 
Choose \(N\) sufficiently large so that
\begin{align*}
L_c\,\dtv{\mu^i,\mu^{i+1}} = L_c \frac{\dtv{\mu,\mu'}}{N}
< \delta_{\Gamma}
\end{align*}
for all \(i\in\{0,\ldots,N-1\}\). 
By the triangle inequality,
\begin{align*}
\dtv{\Gamma(\mu),\Gamma(\mu')} =
\dtv{\Gamma(\mu^N),\Gamma(\mu_0)}\le \sum_{i=0}^{N-1}
\dtv{\Gamma(\mu^i),\Gamma(\mu^{i+1})}.
\end{align*}
Since \(\Gamma(\mu)\neq\Gamma(\mu')\), there exists some \(i\in\{0,\ldots,N-1\}\) such that \(\Gamma(\mu^i)\neq\Gamma(\mu^{i+1}).\)
Otherwise, \emph{all terms} along the chain would have identical images under \(\Gamma\), which would imply \(\Gamma(\mu)=\Gamma(\mu')\), a contradiction.

For the index \(i\), the preceding estimate gives \(\dtv{\Gamma(\mu^i),\Gamma(\mu^{i+1})} \ge \delta_{\Gamma}.\)
On the other hand, since \(\Gamma\) is Lipschitz with constant \(L_c\), we have
\begin{align*}
\dtv{\Gamma(\mu^i),\Gamma(\mu^{i+1})} \le L_c \dtv{\mu^i,\mu^{i+1}} < \delta_{\Gamma},
\end{align*}
a contradiction. Thus, \(\Gamma\) cannot be Lipschitz continuous and consequently, \(\Gamma\) cannot be a contraction.
\end{proof}
\section{Proof of Theorem \ref{thm:existence:entropic:RA:MP:MFE}}\label{sec:existence:ent-reg:appen}

We now restate and provide a proof of Theorem \ref{thm:existence:entropic:RA:MP:MFE} from Section \ref{subsec:Entropy:Regularization}. 
% ========================================
% ======== Thrm: Existence EntReg ========
% ========================================

\begin{myOCP}
\begin{theorem}[Existence of entropy-regularized RA-MP-MFE]\label{thm:existence:entropic:RA:MP:MFE:appen}
Recall the problem data from Section \ref{sec:prob:form} and assume that \(\mu_0^j\in\P(\X)\) are fixed for all \(j\in[M]\). 
Let Assumption~\ref{assmpt:lipschitiz-model} be in force; then, for every fixed \(\eps\in \lcrc{0}{1}\) and \(\entparam>0\), there exists an entropy-regularized RA-MP-MFE \((\pi^\star_{\entparam},\mu^\star_{\entparam})\in\Pi\times\M\) in the sense of Definition~\ref{def:entropic:RA:MP:MFE}.
More precisely, there exists \(\pi^\star_{\entparam}\) such that, with \(\mu^\star_{\entparam}\Let\B_\prop(\pi^\star_{\entparam})\), for every \(j\in[M]\),
\begin{align}\label{eq:existence:entropic:RA:MP:MFE}
\pi^{j,\star}_{\entparam} \in \argmax_{\pi^j\in\Pi^j}
\inf_{\widehat\mu^{-j}\in\ambset(\mu^\star_{\entparam},\eps)}
\E^{\pi^j}\left[ \sum_{t=0}^{T}
r_t^j\bigl(x_t^j,\mu_{t,\entparam}^{j,\star},\widehat\mu_t^{-j}\bigr)
-\entparam\sum_{t=0}^{T-1} \log \pi_t^j(u_t^j|x_t^j) \right],
\end{align}
and \newedit{in addition}, i.e., \(\mu_{\entparam}^{j,\star}=\B_\prop^j(\pi_{\entparam}^{j,\star})\quad\text{for all }\,j\in[M].\)
\end{theorem}
\end{myOCP}

\begin{proof}
Our proof is a modification of the arguments presented in Theorem~\ref{thm:existence:RA:MP:MFE}. We only verify the additional regularity introduced by the entropy regularization step, i.e., continuity and concavity.

To this end, recall \(\Occ^j\), \(\Occ\), and \(\xi(o)\) from the definition of occupation measures and properties \eqref{occ:constraint:1}--\eqref{occ:constraint:2}. For \(\tilde{o}^j\Let(\tilde \xi^j,\tilde\eta^j)\in\Occ^j\), define
\begin{align}\label{eq:entropy:occ:regularizer}
\tilde{o}^j \mapsto \Psi_{j,\entparam}(\tilde{o}^j)
\Let -\entparam\sum_{t=0}^{T-1}\sum_{x\in\X}\sum_{u\in\U}
\tilde\eta_t^j(x,u) \log\frac{\tilde\eta_t^j(x,u)}{\tilde \xi_t^j(x)}.
\end{align}
Note that, when \(\tilde \xi_t^j(x)=0\), the occupation constraints imply \(\tilde\eta_t^j(x,u)=0\) for every \(u\in\U\), so the corresponding summand is defined to be zero.
Also, define the \emph{entropy-regularized} robust occupation-measure payoff by
\begin{align}\label{eq:entropic:robust:occ:payoff}
(\tilde{o}^j,o) \mapsto \Phi_{j,\entparam}(\tilde{o}^j;o) \Let
\inf_{\widehat\mu^{-j}\in\ambset(\xi(o),\eps)}
\left[ \sum_{t=0}^{T}\sum_{x\in\X}
\tilde \xi_t^j(x)r_t^j(x,\xi_t^j,\widehat\mu_t^{-j})
+ \Psi_{j,\entparam}(\tilde{o}^j) \right].
\end{align}
Since \(\Psi_{j,\entparam}(\cdot)\) in \eqref{eq:entropy:occ:regularizer} does not depend on \(\widehat\mu^{-j}\), we can write \eqref{eq:entropic:robust:occ:payoff} as
\begin{align}\label{eq:entropic:Phi:decomp}
\Phi_{j,\entparam}(\tilde{o}^j;o) = \Phi_j(\tilde{o}^j;o)+\Psi_{j,\entparam}(\tilde{o}^j),
\end{align}
where \(\Phi_j(\cdot)\) is the unregularized robust occupation-measure payoff defined in \eqref{robust:occ:payoff}.

In Theorem~\ref{thm:existence:RA:MP:MFE}, we have established that the map \((\tilde{o}^j,o)\mapsto\Phi_j(\tilde{o}^j;o)\) is continuous on \(\Occ^j\times\Occ\), and, for fixed \(o\in\Occ\), the map \(\tilde{o}^j\mapsto\Phi_j(\tilde{o}^j;o)\) is concave. We now show that the same continuity and concavity properties hold for \(\Psi_{j,\entparam}(\cdot)\).

\textbf{Continuity of \(\Psi_{j,\entparam}(\cdot)\) on \(\Occ^j\):} Indeed, on the region where \(\tilde \xi_t^j(x)>0\), this follows directly from continuity of the function \((a,b)\mapsto a\log(a/b)\). We need to check the regularity when \(\tilde \xi_t^j(x)\to0\). Whenever \(\tilde \xi_t^j(x)>0\), define
\(\tilde\pi_t^j(u|x)\Let\tilde\eta_t^j(x,u)/\tilde \xi_t^j(x).\)
\newedit{Then,}
\begin{align*}
0 \le -\sum_{u\in\U}\tilde\eta_t^j(x,u) \log\frac{\tilde\eta_t^j(x,u)}{\tilde \xi_t^j(x)} = \tilde \xi_t^j(x) \left(-\sum_{u\in\U}\tilde\pi_t^j(u|x)\log\tilde\pi_t^j(u|x) \right)\le \tilde \xi_t^j(x)\log|\U|,
\end{align*}
and thus, whenever \(\tilde \xi_t^j(x)\to0\), the entropy contribution converges to zero.
Hence, \(\Psi_{j,\entparam}(\cdot)\) is continuous on \(\Occ^j\).

\noindent \textbf{Concavity of \(\Psi_{j,\entparam}(\cdot)\) on \(\Occ^j\):}
For each fixed \(t\) and \(x\), define 
\begin{align*}
\mathcal{D}_\U\Let \aset[\bigg]{(\xi,\zeta)\in\Real_+\times\Real[|\U|]_{+} \suchthat \sum_{u\in\U}\zeta(u)=\xi}.
\end{align*}
Then, the map \(\mathcal{D}_\U\ni(\xi,\zeta)\mapsto -\sum_{u\in\U}\zeta(u)\log\frac{\zeta(u)}{\xi}\in\Real\)
is concave in the joint variable $(\xi,\zeta)$. Therefore \(\Psi_{j,\entparam}(\cdot)\), being a finite sum of concave functions, is concave on \(\Occ^j\).\footnote{In fact, more is true; see Lemma \ref{lem:strong:concavity:entreg}, assertion \ref{lem:strong:concavity:entreg:I} ahead.}

Thus, it follows from \eqref{eq:entropic:Phi:decomp} that \(\Phi_{j,\entparam}(\cdot)\) is continuous on \(\Occ^j\times\Occ\), and that, for every fixed \(o\in\Occ\), the map \(\tilde{o}^j\mapsto\Phi_{j,\entparam}(\tilde{o}^j;o)\) is concave.
In a similar manner as before, define the entropy-regularized best-response correspondence 
\begin{align}\label{eq:ent:best:response}
o \mapsto \bestresp_{j,\entparam}(o)
\Let \argmax_{\tilde{o}^j\in\Occ^j} \Phi_{j,\entparam}(\tilde{o}^j;o).
\end{align}
Since \(\Occ^j\) is compact and \(\Phi_{j,\entparam}(\cdot)\) is continuous, Weierstrass' theorem implies that \(\bestresp_{j,\entparam}(o)\) is nonempty and compact.
Since \(\Phi_{j,\entparam}(\cdot;o)\) is concave and \(\Occ^j\) is convex, \(\bestresp_{j,\entparam}(o)\) is convex.
Moreover, since the feasible set is fixed and compact and \(\Phi_{j,\entparam}(\cdot)\) is jointly continuous, Berge's maximum principle~\cite[Chapter 17, Theorem 17.31]{ref:AliBor-06} implies that \(\bestresp_{j,\entparam}(\cdot)\) is upper hemicontinuous.

Then, the correspondence \(\bestresp_{\entparam}:\Occ\rightrightarrows\Occ\), defined by \(\bestresp_{\entparam}(o) \Let \prod_{j=1}^M\bestresp_{j,\entparam}(o)\), has nonempty, compact, convex values and is upper hemicontinuous. 
Since \(\Occ\) is nonempty, compact, and convex, Kakutani's fixed-point theorem \cite[Chapter~17, Corollary~17.55]{ref:AliBor-06} implies that there exists \(o^\star\in\Occ\) such that \(o^\star\in\bestresp_{\entparam}(o^\star).\)\footnote{More is true; see Lemma \ref{lem:strong:concavity:entreg}, assertion \ref{lem:strong:concavity:entreg:II}, for details.}

Finally, reconstruct Markov policies from the fixed point \(o^\star\) as in the proof of Theorem~\ref{thm:existence:RA:MP:MFE}. 
The policy--occupation-measure equivalence gives the consistency condition \(\mu^{j,\star}=\B_\prop^j(\pi^{j,\star})\), while \(o^\star\in\bestresp_{\entparam}(o^\star)\) gives the entropy-regularized risk-averse optimality inequality in \eqref{eq:existence:entropic:RA:MP:MFE}.
Our proof is complete.
\end{proof}

\begin{remark}
In fact the best response map is a singleton and the optimizer is unique. This follows from Lemma \ref{lem:strong:concavity:entreg} ahead. 
\end{remark}

% ========================================
% ===== Thrm: Approx RA-MFE EntReg =======
% ========================================

\textbf{An approximation result:} Our next (additional) result concerns \newedit{the} approximation of the RA-MFE under entropy regularization.

\begin{myOCP}
\begin{proposition}[Approximation and vanishing-regularization limit]\label{thm:approx:vanishing:entropic:RA:MP:MFE}
Suppose that \edit{the assumptions of} Theorem \ref{thm:existence:entropic:RA:MP:MFE} hold. 
Fix \(\eps \in \lcrc{0}{1}\), and for any \(\pi\Let(\pi^1,\ldots,\pi^M)\in\Pi\), let \(\mu^\pi\Let\B_\prop(\pi)\), and recall the definition of (population-wise) risk-averse exploitability by
\begin{align}\label{eq:population:RA:exploitability}
\phi_{\RA}^{j}(\pi) \Let \sup_{\tilde\pi^j\in\Pi^j}
\bar{\objective}_{j}(\tilde\pi^j;\mu^\pi) - \bar{\objective}_{j}(\pi^j;\mu^\pi) \quad\text{for } j\in[M],
\end{align}
and the total risk-averse exploitability by
\begin{align}\label{eq:total:RA:exploitability}
\phi_{\RA}(\pi)\Let\sum_{j=1}^M\phi_{\RA}^{j}(\pi).
\end{align}
Then, the following assertions hold:
\begin{enumerate}[label=\textup{(\ref{thm:approx:vanishing:entropic:RA:MP:MFE}-\alph*)}, leftmargin=*, widest=ii, align=left]
\item \label{approx:thrm:conseq:1} For every \(\entparam>0\), let \((\pi^\star_\entparam,\mu^\star_\entparam)\in\Pi\times\M\) be an entropy-regularized RA-MP-MFE in the sense of Definition~\ref{def:entropic:RA:MP:MFE}. Then, for every \(j\in[M]\), we have the bound
\begin{align}\label{eq:population:RA:exploitability:bound}
\phi_{\RA}^{j}(\pi^\star_\entparam)\le \entparam T\log|\U|,
\end{align}
and consequently, \(\phi_{\RA}(\pi^\star_\entparam)\le M\entparam T\log|\U|\).
In particular, every entropy-regularized RA-MP-MFE is a (population-wise) \(\entparam T\log|\U|\)-approximate RA-MP-MFE of the unregularized problem.

\item \label{approx:thrm:conseq:2} Let \(\{\entparam_n\}_{n\in\Nz}\subset (0, +\infty)\) be such that \(\entparam_n\to 0\) as \(n\to +\infty\). 
For each \(n\in\Nz\), let \(\bigl(\pi^\star_{\entparam_n},\mu^\star_{\entparam_n}\bigr)\in\Pi\times\M\) be an entropy-regularized RA-MP-MFE, and let \(o^\star_{\entparam_n}\in\Occ\) be an occupation-measure representation induced by \(\pi^\star_{\entparam_n}\). 
If \(o^\star\in\Occ\) is an accumulation point of the sequence \(\{o^\star_{\beta_n}\}_{n\in\Nz}\), then \(o^\star\) induces a policy \(\pi^\star\in\Pi\) and a mean-field flow \(\mu^\star\Let \xi(o^\star)\) such that \((\pi^\star,\mu^\star)\) is a RA-MP-MFE in the sense of Definition~\ref{def:RA:MP:MFE}.
\end{enumerate}
\end{proposition}
\end{myOCP}

\begin{remark}\label{rem:on:trade:off}
Theorem~\ref{thm:approx:vanishing:entropic:RA:MP:MFE} highlights a natural trade-off for the choice of the parameter \(\entparam\) used for regularization. As \(\entparam\) tends to \(0\), the (entropy-regularized) equilibrium becomes an increasingly accurate approximation of the original risk-averse equilibrium, with the approximation error vanishing at a rate \(\mathcal{O}(\entparam)\).
On the other hand, stronger entropy regularization improves the stability of the best-response mapping and, as shown below, facilitates contractivity. Thus, smaller values of \(\entparam\) favor fidelity to the unregularized game, whereas larger values favor contractive behavior and computational convergence. This is common phenomenon in entropy regularized mean-field games~\cite{cui2021approximately}. This is also visible in the multi-population chasing example presented in Section~\ref{sec:numexp}.
\end{remark}

\begin{proof}[\highlight{Proof of Proposition \ref{thm:approx:vanishing:entropic:RA:MP:MFE}}]
We start with the proof of the first assertion, i.e., \ref{approx:thrm:conseq:1}.
To this end, fix \(\entparam>0\), and let \((\pi^\star_\entparam,\mu^\star_\entparam)\in\Pi\times\M\) be an entropy-regularized RA-MP-MFE. For each \(j\in[M]\) and \(\pi^j\in\Pi^j\), define the entropy
\begin{align*}
\pi^j \mapsto \entropy_j(\pi^j) \Let \E^{\pi^j}\left[-\sum_{t=0}^{T-1}\log\pi_t^j(u_t^j|x_t^j)\right].
\end{align*}
Since the Shannon entropy of a probability vector on \(\U\) is bounded between \(0\) and \(\log|\U|\), it follows that
\begin{align}\label{eq:entropy:policy:bound}
0\le \entropy_j(\pi^j)\le T\log|\U|\quad \text{for every } \pi^j\in\Pi^j.
\end{align}
By Definition~\ref{def:entropic:RA:MP:MFE}, for every \(\tilde\pi^j\in\Pi^j\), we have \(\bar{\objective}_{j,\entparam}(\pi^{j,\star}_\entparam;\mu^\star_\entparam)\ge \bar{\objective}_{j,\entparam}(\tilde\pi^j;\mu^\star_\entparam).\)
Since the entropy term does not depend on \(\widehat\mu^{-j}\), we have \(\bar{\objective}_{j,\entparam}(\pi^j;\mu)=\bar{\objective}_{j}(\pi^j;\mu)+\entparam\,\entropy_j(\pi^j).\)
Therefore, for every \(\tilde\pi^j\in\Pi^j\), we have the bound
\begin{align*}
\bar{\objective}_{j}(\pi^{j,\star}_\entparam;\mu^\star_\entparam)+\entparam\,\entropy_j(\pi^{j,\star}_\entparam)\ge \bar{\objective}_{j}(\tilde\pi^j;\mu^\star_\entparam)+\entparam\,\entropy_j(\tilde\pi^j).
\end{align*}
Rearranging and using \eqref{eq:entropy:policy:bound}, we obtain
\begin{align*}
\bar{\objective}_{j}(\tilde\pi^j;\mu^\star_\entparam)-\bar{\objective}_{j}(\pi^{j,\star}_\entparam;\mu^\star_\entparam)\le \entparam\left[\entropy_j(\pi^{j,\star}_\entparam)-\entropy_j(\tilde\pi^j)\right]\le \entparam T\log|\U|.
\end{align*}
Taking the supremum over \(\tilde\pi^j\in\Pi^j\), and using \(\mu^\star_\entparam=\B_\prop(\pi^\star_\entparam)\), yields \(\phi_{\RA}^{j}(\pi^\star_\entparam)\le \entparam T\log|\U|.\)
This proves \eqref{eq:population:RA:exploitability:bound}, and summing over \(j\in[M]\) gives \(\phi_{\RA}(\pi^\star_\entparam)\le M\entparam T\log|\U|\); as desired.

We now prove \ref{approx:thrm:conseq:2}. 
Let \(o^\star\in\Occ\) be an accumulation point of \(\{o^\star_{\entparam_n}\}_{n\in\Nz}\). 
Passing to a subsequence if necessary (and for convinience, not relabeling) suppose that \(o^\star_{\entparam_n}\to o^\star\) as \(n \to +\infty\).
We write \(o^\star_{\entparam_n}\Let(o^{1,\star}_{\entparam_n},\ldots,o^{M,\star}_{\entparam_n})\) and \(o^\star\Let(o^{1,\star},\ldots,o^{M,\star})\).
Fixing \(j\in[M]\) and \(\tilde{o}^j\in\Occ^j\) and noting the fact that \(o^\star_{\entparam_n}\) is an entropy-regularized fixed point in occupation measures, we have \(\Phi_{j,\entparam_n}(o^{j,\star}_{\entparam_n};o^\star_{\entparam_n})\ge \Phi_{j,\entparam_n}(\tilde{o}^j;o^\star_{\entparam_n}).\)
Using the fact that \(\Phi_{j,\entparam_n}=\Phi_j+\Psi_{j,\entparam_n}\), we have
\begin{align*}
\Phi_j(\tilde{o}^j;o^\star_{\entparam_n})-\Phi_j(o^{j,\star}_{\entparam_n};o^\star_{\entparam_n})\le \Psi_{j,\entparam_n}(o^{j,\star}_{\entparam_n})-\Psi_{j,\entparam_n}(\tilde{o}^j).
\end{align*}
By the entropy bound in occupation-measure form, \(0\le \Psi_{j,\entparam_n}(o^j)\le \entparam_n T\log|\U|\) for every \(o^j\in\Occ^j\). Hence,
\begin{align}\label{eq:vanishing:approx:occ:ineq}
\Phi_j(\tilde{o}^j;o^\star_{\entparam_n})-\Phi_j(o^{j,\star}_{\entparam_n};o^\star_{\entparam_n})\le \entparam_n T\log|\U|.
\end{align}
Letting \(n\to+\infty\) in \eqref{eq:vanishing:approx:occ:ineq}, the right-hand side converges to zero. 
Moreover, by the proof of Theorem~\ref{thm:existence:RA:MP:MFE}, the map \((\tilde{o}^j,o)\mapsto\Phi_j(\tilde{o}^j;o)\) is continuous on \(\Occ^j\times\Occ\).
Therefore, \(\Phi_j(\tilde{o}^j;o^\star)-\Phi_j(o^{j,\star};o^\star)\le 0.\)
Since \(\tilde{o}^j\in\Occ^j\) was arbitrary, it follows that
\begin{align*}
o^{j,\star}\in\argmax_{\tilde{o}^j\in\Occ^j}\Phi_j(\tilde{o}^j;o^\star).
\end{align*}
Since \(j\in[M]\) was arbitrary, \(o^\star\) satisfies the unregularized risk-averse optimality condition for every population.

Finally, since \(o^\star\in\Occ\), the policy--occupation-measure equivalence used in the proof of Theorem~\ref{thm:existence:RA:MP:MFE} gives a policy \(\pi^\star\in\Pi\) inducing \(o^\star\), and the corresponding MF-flow is \(\mu^\star\Let \xi(o^\star)\). 
The occupation constraints give \(\mu^{j,\star}=\B_\prop^j(\pi^{j,\star})\) for every \(j\in[M]\), while the optimality condition above gives the risk-averse best-response condition. 
Hence \((\pi^\star,\mu^\star)\) is a RA-MP-MFE in the sense of Definition~\ref{def:RA:MP:MFE}. 
Our proof is complete.
\end{proof}

\section{Proofs of Theorems \ref{thm:contractivity:regularized:BR} and \ref{thm:contractivity:inner:regularized:BR}}\label{sec:contraction:ent-reg:appen}

% ========================================
% ======= Lem: Strong concavity ==========
% ========================================

\black{In Theorem \ref{thm:existence:entropic:RA:MP:MFE} we showed that the mapping 
\(\Psi_{j,\entparam}(\cdot)\) is concave. We start by showing a stronger result, namely, that \(\Psi_{j,\entparam}(\cdot)\) is \emph{strongly concave}; this is one of the main ingredients for showing contractivity in the sequel.}
\begin{myOCP}
\begin{lemma}[Strong concavity of \(\Phi_{j,\entparam}(
 \cdot,o)\)]\label{lem:strong:concavity:entreg}
 Fix \(j\in[M]\), \(T \in \Nz\), and \(\entparam>0\). For \(o^j=(\xi^j,\eta^j)\) and \(\bar o^j=(\bar{\xi}^j,\bar{\eta}^j)\) in \(\Occ^j\), recall the norm
\begin{align}
\|o^j-\bar{o}^j\|_{\onorm} \Let \max \aset[\Big]{\max_{0\le t \le  T}\|\xi_t^j-\bar{\xi}_t^j\|_1,\,\max_{0 \le t<T}\|\eta_t^j-\bar{\eta}_t^j\|_1}. \nn
\end{align}
Consider the mapping \(R_j:\Occ^j\to\Real\), given by
\begin{align}\label{eq:negative:causal:entropy}
\Occ^j \ni o^j \mapsto R_j(o^j) \Let \sum_{t=0}^{T-1}\sum_{x\in\X}\sum_{u\in\U}
\eta_t^j(x,u)\log\frac{\eta_t^j(x,u)}{\xi_t^j(x)} \in \Real.
\end{align}
Then \(R_j\) is \(1\)-strongly convex on \(\Occ^j\) with respect to \(\|\cdot\|_{\onorm}\), i.e., for every \(o^j,\bar o^j\in\Occ^j\) and \(\lambda\in \lcrc{0}{1}\),
\begin{align}\label{eq:negative:causal:entropy:strong:convexity}
R_j\bigl(\lambda o^j+(1-\lambda)\bar o^j\bigr)\le \lambda R_j(o^j)+(1-\lambda)R_j(\bar o^j) -\frac{\lambda(1-\lambda)}{2}\|o^j-\bar o^j\|_{\onorm}^2.
\end{align}
Consequently, we have the following assertions:
\begin{enumerate}[label=\textup{(\ref{lem:strong:concavity:entreg}-\alph*)}, leftmargin=*, widest=b, align=left]

\item \label{lem:strong:concavity:entreg:I} The mapping \(\Psi_{j,\entparam}(\cdot) = -\entparam R_j(\cdot)\) in \eqref{eq:entropy:occ:regularizer} is \(\entparam\)-strongly concave on \(\Occ^j\) with respect to \(\|\cdot\|_{\onorm}\). 
Moreover, for every fixed \(o\in\Occ\), the map \(\tilde{o}^j\mapsto\Phi_{j,\entparam}(\tilde{o}^j;o)\) in \eqref{eq:entropic:Phi:decomp} is \(\entparam\)-strongly concave on \(\Occ^j\). 

\item \label{lem:strong:concavity:entreg:II} The best response \(\bestresp_{j,\entparam}(o)\) defined in \eqref{eq:ent:best:response} is a singleton for every \(o\in\Occ\). 
\end{enumerate}
\end{lemma}
\end{myOCP}

\begin{proof}
The proof is inspired by a \edit{nice} argument presented in \cite{ref:neu2017} relating Bregman and Kullback–Leibler divergences along with some standard facts from convex optimization \cite{ref:beck2017first}.

Fix \(j \in [M]\).
We will simplify notation and denote \(\occset = \Occ^j\). Also, we suppress the population subcript \(j\) by writing \(o = (\xi,\eta)\), and \(\bar{o} \Let (\bar{\xi},\bar{\eta})\), and \(R \Let R_j\).
We first establish strong convexity of \(o \mapsto R(o)\) on the relative interior \(\relint(\occset)\) (the boundary case will be handled later).
To this end, take \(o = (\xi,\eta)\) and \(\bar{o} = (\bar{\xi},\bar{\eta})\) in \(\relint(\occset)\). 
For each fixed \(t\) and \(x\), define \((\xi,\zeta) \mapsto \dummyent(\xi,\zeta) \Let \sum_{u \in \U}\zeta(u)\log \frac{\zeta(u)}{\xi}\). For \(\xi>0\) and \(\zeta(u)>0\), we have that
\begin{align*}
    \frac{\partial \dummyent}{\partial \zeta(u)} = \log \frac{\zeta(u)}{\xi} + 1 \quad \text{and} \quad \frac{\partial \dummyent}{\partial \xi} = - \frac{\sum_{u \in \U}\zeta(u)}{\xi} = -1,
\end{align*}
where in the last equality we used the fact that \(\sum_{u \in \U}\zeta(u) = \xi\) (which is the occupation constraint). We consider the Bregman divergence of \(R\), which is given by 
\begin{align*}
(o,\bar{o}) \mapsto \bregdiv(o,\bar{o})\Let R(o)-R(\bar{o})-\left\langle\nabla R(\bar{o}),o-\bar{o}\right\rangle.
\end{align*}
Substituting the preceding derivatives gives
\begin{align*}
\bregdiv(o,\bar{o})&=R(o)-R(\bar{o})+\sum_{t=0}^{T-1}\sum_{x\in\X}\bigl(\xi_t(x)-\bar{\xi}_t(x)\bigr)- \nn\\& \qquad\sum_{t=0}^{T-1}\sum_{x\in\X}\sum_{u\in\U}\left(\log\frac{\bar{\eta}_t(x,u)}{\bar{\xi}_t(x)}+1\right)\bigl(\eta_t(x,u)-\bar{\eta}_t(x,u)\bigr).
\end{align*}
Using the occupation constraints \(\sum_{u\in\U}\bigl(\eta_t(x,u)-\bar{\eta}_t(x,u)\bigr)=\xi_t(x)-\bar{\xi}_t(x)\), and simplifying, we get
\begin{align*}
\bregdiv(o,\bar{o})&=R(o)-R(\bar{o})-\sum_{t=0}^{T-1}\sum_{x\in\X}\sum_{u\in\U}\log\frac{\bar{\eta}_t(x,u)}{\bar{\xi}_t(x)}\bigl(\eta_t(x,u)-\bar{\eta}_t(x,u)\bigr)\\
&=\sum_{t=0}^{T-1}\sum_{x\in\X}\sum_{u\in\U}\eta_t(x,u)\log\frac{\eta_t(x,u)/\xi_t(x)}{\bar{\eta}_t(x,u)/\bar{\xi}_t(x)}.
\end{align*}
Let \(\pi\) and \(\bar{\pi}\) be the Markov policies induced by \(o\) and \(\bar{o}\), respectively, so that
\begin{align*}
\pi_t(u|x)=\frac{\eta_t(x,u)}{\xi_t(x)}\quad \text{and} \quad \bar{\pi}_t(u|x)=\frac{\bar{\eta}_t(x,u)}{\bar{\xi}_t(x)}.
\end{align*}
Since \(\eta_t(x,u)=\xi_t(x)\pi_t(u|x)\), we obtain
\begin{align}\label{eq:breg:KL}
\bregdiv(o,\bar{o})=\sum_{t=0}^{T-1}\sum_{x\in\X}\xi_t(x)\KL\bigl(\pi_t(\cdot|x)\,\|\,\bar{\pi}_t(\cdot|x)\bigr).
\end{align}
We now express the right-hand side of \eqref{eq:breg:KL} as a path-space relative entropy. 
Let \(\Omega_{\mathrm{path}}\Let\X^{T+1}\times\U^T\) be the path space of finite trajectories, and we write a typical element of \(\Omega_{\mathrm{path}}\) as \(\omega=(x_0,u_0,\ldots,u_{T-1},x_T)\). 
For all \(\omega \in \Omega_{\mathrm{path}}\), define the path laws induced by \(o\) and \(\bar{o}\) by
\begin{align*}
P_o(\omega)\Let\mu_0^j(x_0)\prod_{t=0}^{T-1}\pi_t(u_t|x_t)f_t^j(x_{t+1}|x_t,u_t)
\end{align*}
and 
\begin{align*}
P_{\bar{o}}(\omega)\Let\mu_0^j(x_0)\prod_{t=0}^{T-1}\bar{\pi}_t(u_t|x_t)f_t^j(x_{t+1}|x_t,u_t).
\end{align*}
By the policy--occupation-measure equivalence (see \ref{step:two} in the proof of Theorem \ref{thm:existence:entropic:RA:MP:MFE:appen}) and the flow-conservation constraints, the \(x_t\)-marginals of \(P_o\) and \(P_{\bar{o}}\) are \(\xi_t\) and \(\bar{\xi}_t\), respectively, while their \((x_t,u_t)\)-marginals are \(\eta_t\) and \(\bar{\eta}_t\), respectively.
Indeed, \(P_o(x_t=x,u_t=u)=\xi_t(x)\pi_t(u|x)=\eta_t(x,u)\), and
\begin{align*}
P_o(x_{t+1}=y)=\sum_{x\in\X}\sum_{u\in\U}\eta_t(x,u)f_t^j(y|x,u)=\xi_{t+1}(y),
\end{align*}
with the analogous identities holding under \(P_{\bar{o}}\).

% \blue{Since \(o,\bar{o}\in\relint(\occset)\), their occupation measures have the same support on every state-action coordinate; hence the induced policies have the same support on all reachable state-action pairs, and, since the initial distribution and transition kernels are common, and thus \(P_o \ll P_{\bar{o}}\).}\sid{I have modified this part a little, check} 

Since \(o,\bar{o}\in\relint(\occset)\), the occupation measures \(\eta\) and \(\bar\eta\) have the same support. Indeed, for each coordinate \((t,x,u)\), either \(\eta_t(x,u)=0\) for every \(o\in\mathcal{O}\), or \(\eta_t(x,u)>0\) for every \(o\in\relint(\mathcal{O})\).
Let us verify that \(P_o \ll P_{\bar o}\).\footnote{Recall the ``\(\mu \ll \nu\)'' denotes the notion of absolute continuity between the measures \(\mu\) and \(\nu\).} Let \(\omega \Let (x_0,u_0,\ldots,u_{T-1},x_T)\in\Omega_{\mathrm{path}}\) be such that \(P_o(\omega)>0\). Then, for every \(t\in\{0,\ldots,T-1\}\), \(\eta_t(x_t,u_t) =  P_o(X_t=x_t,U_t=u_t)>0\). Since \(\eta\) and \(\bar\eta\) have the same support, it follows that \(\bar\eta_t(x_t,u_t)>0\). Consequently,
\begin{align*}
 \bar{\xi}_t(x_t)=\sum_{u\in \U}\bar{\eta}_t(x_t,u)>0 \quad \text{and} \quad \bar{\pi}_t(u_t\mid x_t) = \frac{\bar{\eta}_t(x_t,u_t)}{\bar{\xi}_t(x_t)}>0.   
\end{align*}
Also, \(P_o(\omega)>0\) implies \(\mu_0^j(x_0)>0\) and
\(f_t^j(x_{t+1}\mid x_t,u_t)>0\) for every \(t\).
Since \(P_o\) and \(P_{\bar o}\) have the same initial distribution and transition kernels, every factor defining \(P_{\bar o}(\omega)\) is therefore strictly positive. Hence, \(P_o(\omega)>0\) implies that \(P_{\bar o}(\omega)>0,\) and thus \(P_o\ll P_{\bar o}\). Therefore, for every path in the support of \(P_o\),
\begin{align*}
\frac{P_o(x_0,u_0,\ldots,u_{T-1},x_T)}{P_{\bar{o}}(x_0,u_0,\ldots,u_{T-1},x_T)}=\prod_{t=0}^{T-1}\frac{\pi_t(u_t|x_t)}{\bar{\pi}_t(u_t|x_t)}.
\end{align*}
It follows that
\begin{align*}
\KL(P_o\|P_{\bar{o}})&=\E^{P_o}\left[\log\frac{P_o}{P_{\bar{o}}}\right] =\sum_{t=0}^{T-1}\E^{P_o}\left[\log\frac{\pi_t(U_t|X_t)}{\bar{\pi}_t(U_t|X_t)}\right]\nn \\&=\sum_{t=0}^{T-1}\sum_{x\in\X}\sum_{u\in\U}\eta_t(x,u)\log\frac{\pi_t(u|x)}{\bar{\pi}_t(u|x)}\\
&=\sum_{t=0}^{T-1}\sum_{x\in\X}\xi_t(x)\KL\bigl(\pi_t(\cdot|x)\,\|\,\bar{\pi}_t(\cdot|x)\bigr).
\end{align*}
Comparing the preceding identity with \eqref{eq:breg:KL} yields
\begin{align}\label{eq:bregman:path:KL}
\bregdiv(o,\bar{o})=\KL(P_o\|P_{\bar{o}}).
\end{align}

Every \(\xi_t\) and \(\eta_t\) is a marginal of \(P_o\), and every \(\bar{\xi}_t\) and \(\bar{\eta}_t\) is the corresponding marginal of \(P_{\bar{o}}\). Since marginalization cannot increase the \(\ell_1\)-distance, for every \(t=0,\ldots,T\),
\begin{align*}
\|\xi_t-\bar{\xi}_t\|_1&=\sum_{x\in\X}\left|\sum_{\omega:X_t(\omega)=x}\bigl(P_o(\omega)-P_{\bar{o}}(\omega)\bigr)\right| \nn \\& \le \sum_{\omega\in\Omega_{\mathrm{path}}}\left|P_o(\omega)-P_{\bar{o}}(\omega)\right|=\|P_o-P_{\bar{o}}\|_1.
\end{align*}
Similarly, for every \(t=0,\ldots,T-1\), \(\|\eta_t-\bar{\eta}_t\|_1\le\|P_o-P_{\bar{o}}\|_1.\) Therefore,
\begin{align}\label{eq:occupation:path:distance}
\|o-\bar{o}\|_{\onorm}\le\|P_o-P_{\bar{o}}\|_1.
\end{align}
By Pinsker's inequality \cite[Section 1, Lemma 2]{ref:canonne2023} and \eqref{eq:bregman:path:KL}--\eqref{eq:occupation:path:distance},
\begin{align}\label{eq:breg:srcv:bound}
\bregdiv(o,\bar{o})=\KL(P_o\|P_{\bar{o}})\ge\frac{1}{2}\|P_o-P_{\bar{o}}\|_1^2\ge\frac{1}{2}\|o-\bar{o}\|_{\onorm}^2.
\end{align}
Thus, for every \(o,\bar{o}\in\relint(\occset)\), 
% we have the strong convexity of \(R(\cdot)\); i.e., the following holds
%
\begin{align}\label{eq:breg:gradient}
R(o)\ge R(\bar{o})+\left\langle\nabla R(\bar{o}),o-\bar{o}\right\rangle+\frac{1}{2}\|o-\bar{o}\|_{\onorm}^2.
\end{align}
From the equivalence between the first-order and inequality characterizations of strong convexity \cite[Theorem 5.24]{ref:beck2017first}, inequality \eqref{eq:breg:gradient} implies that, for every \(\lambda\in\lcrc{0}{1}\) and every \(o,\bar{o}\in\relint(\occset)\), 
\begin{align}\label{eq:strong:convexity:relative:interior}
R\bigl(\lambda o+(1-\lambda)\bar{o}\bigr)\le\lambda R(o)+(1-\lambda)R(\bar{o})-\frac{\lambda(1-\lambda)}{2}\|o-\bar{o}\|_{\onorm}^2.
\end{align}

We now extend the result to the boundary of \(\occset\). Choose \(o^\circ\in\relint(\occset)\). For arbitrary \(o,\bar{o}\in\occset\) and \(\delta\in\lcrc{0}{1}\), define
\begin{align*}
o_{\delta}\Let(1-\delta)o+\delta o^\circ \quad \text{and} \quad \bar{o}_\delta\Let(1-\delta)\bar{o}+\delta o^\circ.
\end{align*}
Then, \(o_{\delta},\bar{o}_\delta\in\relint(\occset)\), and hence
\begin{align*}
R\bigl(\lambda o_{\delta}+(1-\lambda)\bar{o}_\delta\bigr)\le\lambda R(o_{\delta})+(1-\lambda)R(\bar{o}_\delta)-\frac{\lambda(1-\lambda)}{2}\|o_{\delta}-\bar{o}_\delta\|_{\onorm}^2.
\end{align*}
Since \(o_{\delta}-\bar{o}_\delta=(1-\delta)(o-\bar{o}),\) we have \(\|o_{\delta}-\bar{o}_\delta\|_{\onorm}=(1-\delta)\|o-\bar{o}\|_{\onorm}.\)
Letting \(\delta\to0^+\) and using the continuity of \(R\), we obtain
\begin{align*}
R\bigl(\lambda o+(1-\lambda)\bar{o}\bigr)\le\lambda R(o)+(1-\lambda)R(\bar{o})-\frac{\lambda(1-\lambda)}{2}\|o-\bar{o}\|_{\onorm}^2.
\end{align*}
Therefore, \(R_j\) is \(1\)-strongly convex on \(\Occ^j\) with respect to \(\|\cdot\|_{\onorm}\). 
Moreover, \(\Psi_{j,\entparam}=-\entparam R_j\), and thus, it follows that \(\Psi_{j,\entparam}\) is \(\entparam\)-strongly concave on \(\Occ^j\).
We also know that, for every fixed \(o\in\Occ\), the mapping \(\tilde{o}^j\mapsto\Phi_j(\tilde{o}^j;o)\) is concave, and \(\Phi_{j,\entparam}(\tilde{o}^j;o)=\Phi_j(\tilde{o}^j;o)+\Psi_{j,\entparam}(\tilde{o}^j).\)
Hence, \(\Phi_{j,\entparam}(\cdot;o)\) is \(\entparam\)-strongly concave on \(\Occ^j\).

Finally, compactness of \(\Occ^j\) and continuity of \(\Phi_{j,\entparam}(\cdot;o)\) imply that \(\bestresp_{j,\entparam}(o)\) is nonempty. Suppose that two distinct points \(o_1^j,o_2^j\in\Occ^j\) are both maximizers, with common maximum value \(V\). Strong concavity with \(\lambda=\frac{1}{2}\) gives
\begin{align*}
\Phi_{j,\entparam}\Biggl(\frac{o_1^j+o_2^j}{2};o\Biggr) \ge V+\frac{\entparam}{8}\|o_1^j-o_2^j\|_{\onorm}^2 > V,
\end{align*}
which contradicts maximality. Therefore, \(\bestresp_{j,\entparam}(o)\) is a singleton for every \(o\in\Occ\). Our proof is complete.
\end{proof}

% ========================================
% ======= Thrm: eps = 1, contractn =======
% ========================================

We restate Theorem \ref{thm:contractivity:regularized:BR} in detail and provide a detailed proof.

\begin{myOCP}
\begin{theorem}[Contractivity of BR mapping when \(\eps=1\)]
\label{thm:contractivity:regularized:BR:appen}
Let \(\eps=1\), and fix \(\entparam>0\). Suppose that Assumption \ref{assum:eps:one:case} holds and let \(q_{\entparam} \Let \max_{j\in[M]}L_j/\entparam.\) Recall that, for \(o\in\Occ\), we have
\begin{align*}
\bestresp_{j,\entparam}(o)\Let\argmax_{\widetilde{o}^j\in\Occ^j}\Phi_{j,\entparam}(\widetilde o^j;o),
\end{align*}
and define \(\bestresp_{\entparam}(o)  \Let \left(\bestresp_{1,\entparam}(o),\ldots,\bestresp_{M,\entparam}(o)\right).\) Then~\newedit{,}
\begin{enumerate}[label=\textup{(\ref{thm:contractivity:regularized:BR}-\alph*)}, leftmargin=*, widest=b, align=left]

\item \newedit{For} every \( o,\bar o\in\Occ\) and \(j\in[M]\), \newedit{the following holds}
\begin{align*}
\left\|\bestresp_{j,\entparam}(o)-\bestresp_{j,\entparam}(\bar o)\right\|_{\onorm}\le\frac{L_j}{\entparam}\left\|o^j-\bar o^j\right\|_{\onorm}.
\end{align*}

\item Consequently, with \(\|o-\bar o\|_{\onorm,\infty}\Let\max_{j\in[M]}\|o^j-\bar o^j\|_{\onorm},\)
we have the estimate \(\left\|\bestresp_{\entparam}(o)-\bestresp_{\entparam}(\bar o)\right\|_{\onorm,\infty}\le q_\entparam\|o-\bar o\|_{\onorm,\infty}.\)
Therefore, if \(\entparam>\max_{j\in[M]} L_j,\) then \(\bestresp_{\entparam}\) is a contraction on \(\Occ\). Consequently, for every feasible initial flow \(\mu_0\), the iteration \(\mu(n+1)=\Gamma_{\entparam}(\mu(n))\) converges to the unique equilibrium mean-field flow \(\mu_{\entparam}^\star\) in Definition \ref{def:entropic:RA:MP:MFE}.
\end{enumerate}
\end{theorem}
\end{myOCP}

\begin{proof}
Fix \(j\in[M]\) and \(o,\bar{o}\in\Occ\) and let us write \(o^j\Let(\xi^j,\eta^j)\) and \(\bar{o}^j\Let(\bar{\xi}^j,\bar{\eta}^j)\).
From Lemma \ref{lem:strong:concavity:entreg}, the mapping \(\tilde{o}^j\mapsto\Phi_{j,\entparam}(\tilde{o}^j;o)\) is \(\entparam\)-strongly concave on \(\Occ^j\) with respect to \(\|\cdot\|_{\onorm}\), and \(\bestresp_{j,\entparam}(o)\) is a singleton for every \(o\in\Occ\).
Thus, we may identify \(\bestresp_{j,\entparam}(o)\) with its unique element and define
\begin{align*}
z^j\Let\bestresp_{j,\entparam}(o)\quad\text{and}\quad \bar z^j\Let\bestresp_{j,\entparam}(\bar{o}),
\end{align*}
and \(z^j\Let(n^j,\zeta^j)\) and \(\bar z^j\Let(\bar n^j,\bar\zeta^j)\).\footnote{\(n\) and \(\zeta\) are dummies for \(\xi\) and \(\eta\).}

Recall the expression of the ambiguity set in \eqref{eq:amb:set:feas}. Since \(\dtv{\nu,\mu}\le 1\) for every \(\nu,\mu\in\P(\X)\), we see that 
\begin{align}\label{eq:fixed:amb:set}
\ambset(\xi(o),1)=\feasflow\quad\text{for every }\,o\in\Occ.
\end{align}
In particular, the ambiguity set is independent of the nominal occupation profile.
Now, let \(\tilde{o}^j\Let(\tilde \xi^j,\tilde\eta^j)\in\Occ^j\). By \eqref{eq:fixed:amb:set} and the structure of the reward \(r_t^j(\cdot)\), we have
\begin{align*}
\Phi_j(\tilde{o}^j;o)
&=\inf_{\widehat\mu^{-j}\in\feasflow}\sum_{t=0}^{T}\sum_{x\in\X}\tilde \xi_t^j(x)\left[a_t^j(x,\widehat\mu_t^{-j})+b_t^j(x,\xi_t^j)\right]\\
&=\inf_{\widehat\mu^{-j}\in\feasflow}\sum_{t=0}^{T}\sum_{x\in\X}\tilde \xi_t^j(x)a_t^j(x,\widehat\mu_t^{-j})
+\sum_{t=0}^{T}\sum_{x\in\X}\tilde \xi_t^j(x)b_t^j(x,\xi_t^j).
\end{align*}
Suppose that, in the above expression,
\begin{align*}
A_j(\tilde{o}^j) \Let \inf_{\widehat\mu^{-j}\in\feasflow}\sum_{t=0}^{T}\sum_{x\in\X}\tilde \xi_t^j(x)a_t^j(x,\widehat\mu_t^{-j}),\,\,\text{and}\,\,
C_j(\tilde{o}^j;\xi^j) \Let\sum_{t=0}^{T}\sum_{x\in\X}\tilde \xi_t^j(x)b_t^j(x,\xi_t^j).
\end{align*}
The mapping \(A_j(\cdot)\) is independent of the nominal profile \(o\), while \(C_j(\tilde{o}^j;\xi^j)\) depends on \(o\) only through the nominal state-flow \(\xi^j\). Consequently, we can  write
\begin{align}\label{eq:cPhi:decom}
\Phi_j(\tilde{o}^j;o)=A_j(\tilde{o}^j)+C_j(\tilde{o}^j;\xi^j).
\end{align}

We first record a consequence of the strong concavity established in Lemma \ref{lem:strong:concavity:entreg}. 
Let \(v^j\in\Occ^j\) be arbitrary and let \(\lambda\in(0, 1)\). Since \(\Occ^j\) is convex, \((1-\lambda)z^j+\lambda v^j\in\Occ^j.\)
By the \(\entparam\)-strong concavity of \(\Phi_{j,\entparam}(\cdot;o)\),
\begin{align*}
\Phi_{j,\entparam}\bigl((1-\lambda)z^j+\lambda v^j;o\bigr)
&\ge(1-\lambda)\Phi_{j,\entparam}(z^j;o)+\lambda\Phi_{j,\entparam}(v^j;o)+\frac{\entparam\lambda(1-\lambda)}{2}\|z^j-v^j\|_{\onorm}^2.
\end{align*}
Since \(z^j\) maximizes \(\Phi_{j,\entparam}(\cdot;o)\) over \(\Occ^j\), we also have \(\Phi_{j,\entparam}(z^j;o)\ge\Phi_{j,\entparam}\bigl((1-\lambda)z^j+\lambda v^j;o\bigr).\)
Combining the preceding inequalities and dividing by \(\lambda>0\) and letting \(\lambda\to0^+\), it follows that
\begin{align}\label{eq:contractivity:quadratic:growth}
\Phi_{j,\entparam}(z^j;o)-\Phi_{j,\entparam}(v^j;o)
\ge\frac{\entparam}{2}\|z^j-v^j\|_{\onorm}^2.
\end{align}
Taking \(v^j=\bar z^j\) in \eqref{eq:contractivity:quadratic:growth} gives \(\Phi_{j,\entparam}(z^j;o)-\Phi_{j,\entparam}(\bar z^j;o)
\ge\frac{\entparam}{2}\|z^j-\bar z^j\|_{\onorm}^2.\). Applying the same argument to \(\Phi_{j,\entparam}(\cdot;\bar{o})\), whose maximizer is \(\bar z^j\), gives \(\Phi_{j,\entparam}(\bar z^j;\bar{o})-\Phi_{j,\entparam}(z^j;\bar{o})
\ge\frac{\entparam}{2}\|z^j-\bar z^j\|_{\onorm}^2.\)
Adding the preceding two quantities, we obtain
\begin{align}\label{eq:contractivity:quadratic:growth:sum}
\entparam\|z^j-\bar z^j\|_{\onorm}^2
\le\Phi_{j,\entparam}(z^j;o)-\Phi_{j,\entparam}(\bar z^j;o)+\Phi_{j,\entparam}(\bar z^j;\bar{o})-\Phi_{j,\entparam}(z^j;\bar{o}).
\end{align}

Using \eqref{eq:entropic:Phi:decomp}, the right-hand side of \eqref{eq:contractivity:quadratic:growth:sum} is equal to
\begin{align*}
&\Phi_j(z^j;o)-\Phi_j(\bar z^j;o)+\Phi_j(\bar z^j;\bar{o})-\Phi_j(z^j;\bar{o}) \nn \\& +\Psi_{j,\entparam}(z^j)-\Psi_{j,\entparam}(\bar z^j)
+\Psi_{j,\entparam}(\bar z^j)-\Psi_{j,\entparam}(z^j)\\
&=\Phi_j(z^j;o)-\Phi_j(\bar z^j;o)+\Phi_j(\bar z^j;\bar{o})-\Phi_j(z^j;\bar{o}).
\end{align*}
Using the decomposition \eqref{eq:cPhi:decom}, basic algebra yields 
\begin{align*}
&\Phi_j(z^j;o)-\Phi_j(\bar z^j;o)+\Phi_j(\bar z^j;\bar{o})-\Phi_j(z^j;\bar{o})=C_j(z^j;\xi^j) \nn \\& - C_j(\bar z^j;\xi^j)-C_j(z^j;\bar{\xi}^j)+C_j(\bar z^j;\bar{\xi}^j).
\end{align*}
Therefore, \eqref{eq:contractivity:quadratic:growth:sum} implies that
\begin{align}\label{eq:contractivity:cross:increment}
\entparam\|z^j-\bar z^j\|_{\onorm}^2 \le C_j(z^j;\xi^j)-C_j(\bar z^j;\xi^j)
-C_j(z^j;\bar{\xi}^j)+C_j(\bar z^j;\bar{\xi}^j).
\end{align}
By the definition of \(C_j(\cdot;\cdot)\), the right-hand side of \eqref{eq:contractivity:cross:increment} is equal to
\begin{align*}
\sum_{t=0}^{T}\sum_{x\in\X}\bigl(n_t^j(x)-\bar n_t^j(x)\bigr)
\bigl(b_t^j(x,\xi_t^j)-b_t^j(x,\bar{\xi}_t^j)\bigr).
\end{align*}
Consequently,
\begin{align*}
\entparam\|z^j-\bar z^j\|_{\onorm}^2
&\le\left|\sum_{t=0}^{T}\sum_{x\in\X}\bigl(n_t^j(x)-\bar n_t^j(x)\bigr)
\bigl(b_t^j(x,\xi_t^j)-b_t^j(x,\bar{\xi}_t^j)\bigr)\right|  \nn \\& 
\le\sum_{t=0}^{T}\|n_t^j-\bar n_t^j\|_1
\left\|b_t^j(\cdot,\xi_t^j)-b_t^j(\cdot,\bar{\xi}_t^j)\right\|_\infty\\
&\le\sum_{t=0}^{T}\ell_{j,t}\|n_t^j-\bar n_t^j\|_1\|\xi_t^j-\bar{\xi}_t^j\|_1.
\end{align*}
Since every occupation measure in \(\Occ^j\) has the prescribed initial distribution \(\mu_0^j\), we have \(\xi_0^j=\bar{\xi}_0^j=n_0^j=\bar n_0^j=\mu_0^j.\) Hence, the term corresponding to \(t=0\) vanishes, and therefore
\begin{align*}
\entparam\|z^j-\bar z^j\|_{\onorm}^2
&\le\sum_{t=1}^{T}\ell_{j,t}\|n_t^j-\bar n_t^j\|_1\|\xi_t^j-\bar{\xi}_t^j\|_1  \le\|z^j-\bar z^j\|_{\onorm}\left(\sum_{t=1}^{T}\ell_{j,t}\right)
\max_{1\le t\le T}\|\xi_t^j-\bar{\xi}_t^j\|_1\\
&=L_j\|z^j-\bar z^j\|_{\onorm}\max_{1\le t\le T}\|\xi_t^j-\bar{\xi}_t^j\|_1.
\end{align*}
If \(z^j=\bar z^j\), then the desired estimate holds trivially. Otherwise, dividing the preceding inequality by \(\entparam\|z^j-\bar z^j\|_{\onorm}\) gives us the estimate 
\begin{align}\label{eq:contract:lipschitz:final}
\left\|\bestresp_{j,\entparam}(o)-\bestresp_{j,\entparam}(\bar{o})\right\|_{\onorm}
\le\frac{L_j}{\entparam}\max_{1\le t\le T}\|\xi_t^j-\bar{\xi}_t^j\|_1  \le \frac{L_j}{\entparam}\|o^j-\bar{o}^j\|_{\onorm},
\end{align}
where the last inequality follows from the fact that \(\|\cdot\|_{\onorm}\), \(\max_{1\le t\le T}\|\xi_t^j-\bar{\xi}_t^j\|_1\le\|o^j-\bar{o}^j\|_{\onorm}.\) Finally, taking the maximum over \(j\in[M]\) in \eqref{eq:contract:lipschitz:final}, we obtain
\begin{align*}
\left\|\bestresp_{\entparam}(o)-\bestresp_{\entparam}(\bar{o})\right\|_{\onorm,\infty}
&=\max_{j\in[M]}\left\|\bestresp_{j,\entparam}(o)-\bestresp_{j,\entparam}(\bar{o})\right\|_{\onorm}\nn \\& \le \max_{j\in[M]}\frac{L_j}{\entparam}\|o^j-\bar{o}^j\|_{\onorm} = q_{\entparam}\|o-\bar{o}\|_{\onorm,\infty}.
\end{align*}
If \(\entparam>\max_{j\in[M]}L_j\), then \(q_{\entparam}<1\). Therefore, \(\bestresp_{\entparam}:\Occ\to\Occ\) is a contraction with respect to \(\|\cdot\|_{\onorm,\infty}\).

Since \(\Occ\) is nonempty and compact and \(q_{\entparam}<1\), Banach's fixed-point theorem \cite{ref:AliBor-06} gives a unique fixed point \(o_{\entparam}^\star\), to which all iterates of \(\bestresp_{\entparam}\) converge. 
By Lemma~\ref{lem:strong:concavity:entreg}, the best-response occupation measure is unique. Together with the policy--occupation-measure equivalence, this makes \(\Gamma_{\entparam}\) single-valued and thus
\begin{align*}
\xi(\bestresp_{\entparam}(o))=\Gamma_{\entparam}(\xi(o))\quad \text{for all }o\in\Occ.
\end{align*}
Hence, by continuity of \(\xi\), the mean-field iteration converges from every feasible initial flow to \(\mu_{\entparam}^\star\Let\xi(o_{\entparam}^\star)\). 
The same equivalence shows that \(o_{\entparam}^\star\) yields an equilibrium in Definition \ref{def:entropic:RA:MP:MFE} and that every such equilibrium induces a fixed point of \(\bestresp_{\entparam}\). Uniqueness of \(o_{\entparam}^\star\) therefore gives uniqueness of the equilibrium mean-field flow.
\end{proof}

% ========================================
% ===== Thrm: 0 < eps < 1, contractn =====
% ========================================

Let us recall Assumptions \ref{assum:eps:one:case} and \ref{assum:eps:0:1} along with several quantities defined therein. We defined \(\ell_{j,t}^a\Let\sum_{k\neq j}\max_{x\in\X} \|A_t^{j,k}(x,\cdot)\|_\infty\) and 
\(L_j^a\Let\sum_{t=1}^T\ell_{j,t}^a.\) Further, recall \(\Occ^j\), \(\Occ\), and \(\xi(o)\) from the definition of occupation measures and properties \eqref{occ:constraint:1}--\eqref{occ:constraint:2}. For \(\overline{o}^j\Let(\overline{\xi}^j,\overline{\eta}^j)\in\Occ^j\) and \(\entparam>0\), define
\begin{align}
\overline{o}^j \mapsto \Psi_{j,\entparam}(\overline{o}^j)
\Let -\entparam\sum_{t=0}^{T-1}\sum_{x\in\X}\sum_{u\in\U}
\overline{\eta}_t^j(x,u) \log\frac{\overline{\eta}_t^j(x,u)}{\overline{\xi}_t^j(x)}. \nn
\end{align}
We now restate Theorem \ref{thm:contractivity:inner:regularized:BR} and then record a detailed proof.

\begin{myOCP}
\begin{theorem}[Contractivity of BR mapping when \(\eps \in (0,1)\)]\label{thm:contractivity:inner:regularized:BR:appen}
Fix \(\eps\in(0,1)\) and \(\entparam>0\), and suppose that Assumptions~\ref{assum:eps:one:case} and~\ref{assum:eps:0:1} hold.
For \(\sigma>0\) and \(\widehat\mu^{-j}=\{\widehat\mu_t^k\}_{k\neq j,\;t=0}^T\), let \(\|\widehat\mu^{-j}\|_2^2\Let\sum_{k\neq j}\sum_{t=0}^T\sum_{x\in\X}|\widehat\mu_t^k(x)|^2\), and define the inner-regularized payoff
\begin{align}\label{eq:inner:reg:appen}
\Phi_{j,\entparam}^{\sigma}(\widetilde{o}^j;o)&\Let\min_{\widehat\mu^{-j}\in\ambset(\xi(o),\eps)}\left[\sum_{t=0}^T\sum_{x\in\X}\widetilde\xi_t^j(x)r_t^j\bigl(x,\xi_t^j,\widehat\mu_t^{-j}\bigr)+\frac{\sigma}{2}\|\widehat\mu^{-j}\|_2^2\right] \nn \\& \qquad\qquad+\Psi_{j,\entparam}(\widetilde{o}^j).
\end{align}
Define the best response mappings \(\bestresp_{j,\entparam}^{\sigma}(o)\Let\argmax_{\widetilde{o}^j\in\occset^j}\Phi_{j,\entparam}^{\sigma}(\widetilde{o}^j;o)\) and \(\bestresp_{\entparam}^{\sigma}(o)\Let(\bestresp_{1,\entparam}^{\sigma}(o),\ldots,\bestresp_{M,\entparam}^{\sigma}(o))\).
Then, for every \(j\in[M]\), there exists a constant \(\Lambda_{j,\eps}\ge0\), independent of \(\sigma\), such that, for every \(\sigma>0\) and \(o,\bar{o}\in\Occ\),
\begin{align}\label{eq:inner:regularized:BR:lipschitz:appen}
\left\|\bestresp_{j,\entparam}^{\sigma}(o)-\bestresp_{j,\entparam}^{\sigma}(\bar{o})\right\|_{\onorm}\le\frac{L_j+2L_j^a\Lambda_{j,\eps}}{\entparam}\|o-\bar{o}\|_{\onorm,\infty}.
\end{align}
Consequently, defining \(q_{\entparam,\eps}\Let\max_{j\in[M]}(L_j+2L_j^a\Lambda_{j,\eps})/\entparam\), we have
\begin{align*}
\left\|\bestresp_{\entparam}^{\sigma}(o)-\bestresp_{\entparam}^{\sigma}(\bar{o})\right\|_{\onorm,\infty}\le q_{\entparam,\eps}\|o-\bar{o}\|_{\onorm,\infty}.
\end{align*}
If \(\entparam>\max_{j\in[M]}(L_j+2L_j^a\Lambda_{j,\eps})\), then \(\bestresp_{\entparam}\) and \(\bestresp_{\entparam}^{\sigma}\), for every \(\sigma>0\), are contractions on \(\Occ\) in \(\|\cdot\|_{\onorm,\infty}\) with the common factor \(q_{\entparam,\eps}<1\). Their unique fixed points \(o_{\entparam}^\star\) and \(o_{\entparam,\sigma}^\star\), respectively, satisfy
\begin{align*}
\lim_{\sigma\to0^+}\|o_{\entparam,\sigma}^\star-o_{\entparam}^\star\|_{\onorm,\infty}=0.
\end{align*}
Moreover, for every \(\mu_0\), the iteration \(\mu(n+1)=\Gamma_{\entparam}(\mu(n))\) converges to \(\mu_{\entparam}^\star\Let\xi(o_{\entparam}^\star)\), the unique equilibrium mean-field flow in Definition~\ref{def:entropic:RA:MP:MFE}.
\end{theorem}
\end{myOCP}

\begin{proof}

Fix \(j\in[M]\). We first establish the regularity of the inner minimizer; our proof relies on techniques from variational analysis \cite{ref:DonRoc-14} and polyhedral partition \cite{ref:morari2015predictive}. To this end, we first establish uniqueness. Throughout the proof, \(\Phi_{j,\entparam}^{0}\Let\Phi_{j,\entparam}\) and \(\bestresp_{j,\entparam}^{0}\Let\bestresp_{j,\entparam}\), with the same fixed \(\eps\in(0,1)\).

\setcounter{proofstep}{0}
\noindent\proofstep{step:one:cont}{Establishing uniqueness}
By Assumption \ref{assum:eps:0:1}, for every \(\widetilde \xi^j=\{\widetilde \xi_t^j\}_{t=0}^T\), define the vector \(c_j(\widetilde \xi^j)\) by
\begin{align*}
\bigl[c_j(\widetilde \xi^j)\bigr]_{k,t,y}\Let\sum_{x\in\X}\widetilde \xi_t^j(x)A_t^{j,k}(x,y)\quad \text{for } k\neq j,\, t=0,\ldots,T,\text{ and } y\in\X.
\end{align*}
Note that, we view \(c_j(\widetilde{\xi}^j)\) and \(\widehat{\mu}^{-j}\) \edit{as} vectors indexed by \((k,t,y)\) with \(k\neq j\) so that \(\inprod{c_j(\widetilde \xi^j)}{\widehat\mu^{-j}}=\sum_{t=0}^{T}\sum_{k\neq j}\sum_{y\in\X}[c_j(\widetilde \xi^j)]_{k,t,y}\widehat{\mu}_t^k(y)\).
Then, by Assumptions~\ref{assum:eps:one:case} and~\ref{assum:eps:0:1}, the terms in the inner objective of \eqref{eq:inner:reg:appen} that depend on \(\widehat\mu^{-j}\) are given by \(\inprod{c_j(\widetilde \xi^j)}{\widehat\mu^{-j}}+\frac{\sigma}{2}\left\|\widehat\mu^{-j}\right\|_2^2.\)
Consequently, the inner minimizers of \eqref{eq:inner:reg:appen} coincide with the minimizers of
\begin{align}\label{eq:proof:inner:regularized:QP}
\min_{\widehat\mu^{-j}\in\ambset(\xi(o),\eps)}\left[\left\langle c_j(\widetilde \xi^j),\widehat\mu^{-j}\right\rangle+\frac{\sigma}{2}\left\|\widehat\mu^{-j}\right\|_2^2\right].
\end{align}
The ambiguity set \(\ambset(\xi(o),\eps)\) is nonempty, compact, and convex (Proposition \ref{prop:ambiguity:set}).
Since \(\sigma>0\), the objective in \eqref{eq:proof:inner:regularized:QP} is \(\sigma\)-strongly convex in \(\widehat\mu^{-j}\). Therefore, \eqref{eq:proof:inner:regularized:QP} admits a \emph{unique} minimizer.
We will denote it by \(\widehat\mu_{\sigma}^{-j,\star}(\widetilde{o}^j;o)=\{\widehat\mu_{\sigma,t}^{-j,\star}(\widetilde{o}^j;o)\}_{t=0}^T\), where \(\widehat\mu_{\sigma,t}^{-j,\star}(\widetilde{o}^j;o)\) denotes the tuple of opponents' state distributions at time \(t\).

\noindent\proofstep{step:two:cont}{Establishing structure and regularity}
We next establish a Lipschitz estimate for the inner minimizer with a constant independent of \(\sigma\).
By Proposition~\ref{prop:ambiguity:set}, the occupation constraints and the auxiliary-variable representation of the total-variation constraints are jointly linear in the nominal and adversarial flows.
Since a linear projection of a polyhedron is again a polyhedron \cite[Theorem~19.3]{ref:Roc:ConvAnal-70}, there exist fixed matrices \(E_j,F_j,C_j,G_j\) and fixed vectors \(e_j,d_j\), depending only on \(j\) and \(\eps\), such that, for every \(\mu\in\xi(\Occ)\),
\begin{align}\label{eq:proof:ambiguity:polyhedral:representation}
\ambset(\mu,\eps)=\aset[\big]{\nu\suchthat E_j\nu+F_j\mu=e_j,\quad C_j\nu+G_j\mu\le d_j}.
\end{align}
Here \(\nu\) denotes the vector obtained by stacking the opponent flows.

Recall that \(\widehat\mu_{\sigma}^{-j,\star}(\widetilde{o}^j;o)\) is the unique minimizer of \eqref{eq:proof:inner:regularized:QP} (established in the previous step). For an arbitrary vector \(v\) of the same dimension as \(\nu\), define the projection operator \cite[Chapter 1]{ref:parikh2014proximal}
\begin{align}\label{eq:proof:projection:def}
P_j(v,\mu)\Let\operatorname*{argmin}_{\nu\in\ambset(\mu,\eps)}\frac{1}{2}\left\|\nu-v\right\|_2^2.
\end{align}
First, note that, since \(\ambset(\mu,\eps)\) is nonempty, closed, and convex, \(P_j(v,\mu)\) is well defined and single-valued. The \emph{primary idea} of this step is to first, interpret the minimizer \(\widehat\mu_{\sigma}^{-j,\star}(\widetilde{o}^j;o)\) as a projection along the lines of \cite[Chapter 6]{ref:parikh2014proximal} and then, following the active-set analysis\footnote{This is a well-known route to tackle multiparametric quadratic programs and have been applied to the so-called `explicit model predictive control' in control theory \cite{ref:morari2015predictive,ref:bemporad2002explicit}.} of \cite[Section 4]{ref:bemporad2002explicit}, to establish a piecewise-affine representation of \(P_j\) and to show that \(P_j(v,\cdot)\) is Lipschitz on \(\xi(\Occ)\), with a Lipschitz constant independent of \(v\).

To this end, let \(\mathcal R_j\) denote the row-index set of \(C_j\).
For every \(I\subseteq\mathcal R_j\), let \(C_{j,I}\) and \(G_{j,I}\) denote the submatrices of \(C_j\) and \(G_j\), respectively, formed by the rows indexed by \(I\), let \(d_{j,I}\) denote the corresponding subvector of \(d_j\), and let \(I^c\Let\mathcal R_j\setminus I\).
Define
\begin{align}\label{eq:proof:HI:RI:hi}
H_I\Let\begin{bmatrix}E_j\\ C_{j,I}\end{bmatrix},\quad R_I\Let\begin{bmatrix}F_j\\ G_{j,I}\end{bmatrix},\quad h_I\Let\begin{bmatrix}e_j\\ d_{j,I}\end{bmatrix}.
\end{align}

Fix \(v\) and \(\mu\), and write \(\nu=P_j(v,\mu)\).
By the variational-inequality characterization of Euclidean projection \cite[Theorem~2A.7]{ref:DonRoc-14} together with the normal-cone representation for a polyhedron \cite[Theorem~2E.3]{ref:DonRoc-14}, there exist a subset \(I\subseteq\mathcal R_j\), an equality multiplier \(\zeta\), and a multiplier \(\lambda_I\ge0\) such that
\begin{subequations}\label{eq:proof:projection:optimality}
\begin{align}
E_j\nu+F_j\mu&=e_j, \label{eq:proof:projection:optimality:a}\\
C_{j,I}\nu+G_{j,I}\mu&=d_{j,I}, \label{eq:proof:projection:optimality:b}\\
C_{j,I^c}\nu+G_{j,I^c}\mu&\le d_{j,I^c}, \label{eq:proof:projection:optimality:c}\\
\nu-v+E_j^\top\zeta+C_{j,I}^\top\lambda_I&=0,\qquad \lambda_I\ge0. \label{eq:proof:projection:optimality:d}
\end{align}
\end{subequations}
%
%\bhavs{do we fix $e_j$ and $d_j$? do they depend on $\nu$ or $v$?}\sid{No, they are fixed.}
Conversely, if \(\nu,\zeta,\lambda_I\) satisfy \eqref{eq:proof:projection:optimality} for some \(I\subseteq\mathcal R_j\), then \(\nu=P_j(v,\mu)\). Equations \eqref{eq:proof:projection:optimality:a}--\eqref{eq:proof:projection:optimality:b} give
\begin{align}\label{eq:proof:HI:equality}
H_I\nu=h_I-R_I\mu.
\end{align}
Moreover, \eqref{eq:proof:projection:optimality:d} implies \(\nu-v\in\operatorname{range}(H_I^\top),\) which together with \eqref{eq:proof:HI:equality} uniquely determine the Euclidean projection of \(v\) onto the affine set \(\aset[]{\eta \suchthat H_I\eta=h_I-R_I\mu}\).
Therefore \cite[Chapter 6]{ref:parikh2014proximal},
\begin{align}\label{eq:proof:projection:affine:formula}
P_j(v,\mu)=v+H_I^\dagger\bigl(h_I-R_I\mu-H_Iv\bigr),
\end{align}
where \(H_I^\dagger\) denotes the Moore--Penrose pseudoinverse.
For each \(I\subseteq\mathcal R_j\), define
\begin{align*}
\mathcal{Q}_I\Let \aset[\big]{(v,\mu,\nu,\zeta,\lambda_I) \suchthat  \mu\in\xi(\Occ) \text{ and }\eqref{eq:proof:projection:optimality:a}-\eqref{eq:proof:projection:optimality:d} \text{ holds}}
\end{align*}
Since \(\Occ\) is a polytope and \(\xi\) is linear, \(\xi(\Occ)\) is a polytope %\bhavs{confirming: this is basically a slice of $o$ right?} \sid{Yes, it is kind of a slice, i.e., we select the state flow of the occ. meas.}.
Hence \(\mathcal Q_I\) is a polyhedron. Let
\begin{align*}
\mathcal{P}_I\Let\aset[\big]{(v,\mu)\suchthat \text{there exist }\nu,\zeta,\lambda_I\text{ with }(v,\mu,\nu,\zeta,\lambda_I)\in\mathcal Q_I}.
\end{align*}
Note that, the set \(\mathcal{P}_I\) is the coordinate projection of the polyhedron \(\mathcal{Q}_I\), and is therefore itself a polyhedron \cite[Theorem~19.3]{ref:Roc:ConvAnal-70}.
%
% The finitely many sets \(\mathcal P_I\), \(I\subseteq\mathcal R_j\), cover \(\mathbb R^{\dim(\nu)}\times\xi(\Occ)\) \bhavs{is this trivial?} \sid{it follow due to the optimality conditions}.
\edit{For every \((v,\mu)\in\Real[\dim(\nu)]\times\xi(\Occ)\), the projection \(P_j(v,\mu)\) exists and satisfies \eqref{eq:proof:projection:optimality} for some \(I\subseteq\mathcal{R}_j\) and suitable multipliers. By the definitions of \(\mathcal{Q}_I\) and \(\mathcal{P}_I\), this implies \((v,\mu)\in\mathcal{P}_I\). Hence, these finitely many sets cover the entire parameter domain; see also \cite[Section 4.2, Theorem 4]{ref:bemporad2002explicit}.}
Moreover, on each \(\mathcal{P}_I\), the projection mapping is given by the affine formula \eqref{eq:proof:projection:affine:formula}.
\edit{Notice that, for fixed \(v\), the expression in \eqref{eq:proof:projection:affine:formula} is affine in \(\mu\), with coefficient matrix \(-H_I^\dagger R_I\)}.  Let \(\|A\|_{2\to2} \Let \sup_{\|z\|_2=1}\|Az\|_2\) denotes the operator norm induced by the Euclidean norms on the domain and codomain; see \cite{ref:guth2025estimating}. Consequently, we set
\begin{align}\label{eq:proof:kappa:def}
\kappa_{j,\eps}\Let\max_{I\subseteq\mathcal R_j}\left\|H_I^\dagger R_I\right\|_{2\to2}.
\end{align}
Since \(\mathcal R_j\) is finite, \(\kappa_{j,\eps}<+\infty\). Fix \(v\) and \(\mu,\bar\mu\in\xi(\Occ)\), and consider the combination \(\mu_s\Let(1-s)\mu+s\bar\mu\) for all \(s\in[0,1]\).
Because \(\xi(\Occ)\) is convex, \(\mu_s\in\xi(\Occ)\) for every \(s\in[0,1]\).
For fixed \(v\), each intersection \(\aset[\big]{s\in[0,1]\suchthat (v,\mu_s)\in\mathcal P_I}\) is an interval (possibly empty) because \(\mathcal{P}_I\) is convex. Since only finitely many sets \(\mathcal P_I\) occur and they cover the segment, there exists a finite partition \(0=s_0<s_1<\cdots<s_N=1\) such that, on every interval \([s_{\ell-1},s_\ell]\), one affine representation \eqref{eq:proof:projection:affine:formula} is valid. \edit{Now notice that \( P_j(v,\mu_{s_\ell})-P_j(v,\mu_{s_{\ell-1}}) = -H_{I_{\ell}}^{\dagger}R_{I_{\ell}}  (\mu_{s_\ell}-\mu_{s_{\ell-1}})\); taking norm and expanding it further, we obtain the estimate}
\begin{align*}
\left\|P_j(v,\mu_{s_\ell})-P_j(v,\mu_{s_{\ell-1}})\right\|_2
&\le \left\|H_{I_\ell}^\dagger R_{I_\ell}\right\|_{2\to2}
\left\|\mu_{s_\ell}-\mu_{s_{\ell-1}}\right\|_2 \le \kappa_{j,\eps}\left\|\mu_{s_\ell}-\mu_{s_{\ell-1}}\right\|_2.
\end{align*}
Summing over \(\ell=1,\ldots,N\) and using \(\sum_{\ell=1}^N\left\|\mu_{s_\ell}-\mu_{s_{\ell-1}}\right\|_2
=\left\|\mu-\bar\mu\right\|_2,\) we obtain
\begin{align}\label{eq:proof:projection:uniform:lipschitz}
\left\|P_j(v,\mu)-P_j(v,\bar\mu)\right\|_2
\le\kappa_{j,\eps}\left\|\mu-\bar\mu\right\|_2.
\end{align}

We now return to the QP \eqref{eq:proof:inner:regularized:QP} and employ the standard completion of square, which yields
\begin{align*}
\left\langle c_j(\widetilde\xi^j),\widehat\mu^{-j}\right\rangle+\frac{\sigma}{2}\left\|\widehat\mu^{-j}\right\|_2^2
=\frac{\sigma}{2}\left\|\widehat\mu^{-j}+\frac{c_j(\widetilde\xi^j)}{\sigma}\right\|_2^2-\frac{1}{2\sigma}\left\|c_j(\widetilde\xi^j)\right\|_2^2.
\end{align*}
Therefore, we now have the characterization
\begin{align}\label{eq:proof:inner:projection:representation}
\widehat\mu_{\sigma}^{-j,\star}(\widetilde{o}^j;o)
=P_j\left(-\frac{c_j(\widetilde\xi^j)}{\sigma},\xi(o)\right).
\end{align}
For fixed \(\widetilde{o}^j\), the first argument of \(P_j\) in \eqref{eq:proof:inner:projection:representation} is the same when \(o\) is replaced by \(\bar{o}\).
Hence \eqref{eq:proof:projection:uniform:lipschitz} yields
\begin{align}\label{eq:proof:inner:optimizer:euclidean:lipschitz}
\left\|\widehat\mu_{\sigma}^{-j,\star}(\widetilde{o}^j;o)-\widehat\mu_{\sigma}^{-j,\star}(\widetilde{o}^j;\bar{o})\right\|_2
\le\kappa_{j,\eps}\left\|\xi(o)-\xi(\bar{o})\right\|_2.
\end{align}
The constant in \eqref{eq:proof:inner:optimizer:euclidean:lipschitz} is therefore independent of \(\sigma\). Finally, from definition of occupation measure, we have
\begin{align*}
\left\|\xi(o)-\xi(\bar{o})\right\|_2
\le\sqrt{M(T+1)}\left\|o-\bar{o}\right\|_{\onorm,\infty}.
\end{align*}
Moreover, for every \(k\neq j\) and \(t=0,\ldots,T\),
\begin{align*}
\dtv{\widehat\mu_{\sigma,t}^{k,\star}(\widetilde{o}^j;o),\widehat\mu_{\sigma,t}^{k,\star}(\widetilde{o}^j;\bar{o})}
&=\frac{1}{2}\left\|\widehat\mu_{\sigma,t}^{k,\star}(\widetilde{o}^j;o)-\widehat\mu_{\sigma,t}^{k,\star}(\widetilde{o}^j;\bar{o})\right\|_1 \nn \\&\le\frac{\sqrt{|\X|}}{2}\left\|\widehat\mu_{\sigma}^{-j,\star}(\widetilde{o}^j;o)-\widehat\mu_{\sigma}^{-j,\star}(\widetilde{o}^j;\bar{o})\right\|_2.
\end{align*}
Thus, with
\begin{align}\label{eq:proof:Lambda:def}
\Lambda_{j,\eps}\Let\frac{1}{2}\sqrt{|\X|M(T+1)}\,\kappa_{j,\eps},
\end{align}
we obtain, for every \(\sigma>0\), \(\widetilde{o}^j\in\Occ^j\), and \(o,\bar{o}\in\Occ\),
\begin{align}\label{eq:proof:inner:optimizer:lipschitz}
\max_{k\neq j}\max_{0\le t\le T}
\dtv{\widehat\mu_{\sigma,t}^{k,\star}(\widetilde{o}^j;o),\widehat\mu_{\sigma,t}^{k,\star}(\widetilde{o}^j;\bar{o})}
\le\Lambda_{j,\eps}\left\|o-\bar{o}\right\|_{\onorm,\infty}.
\end{align}

\noindent\proofstep{step:three:cont}{Lipschitz estimate and contractivity of \(\bestresp\)}
We now derive the Lipschitz estimate for the best-response mapping. Define the non-entropic component of the inner-regularized payoff by
\begin{align}\label{eq:nonentropic:cost}
H_{j,\sigma}(\widetilde \xi^j;o)\Let\min_{\widehat\mu^{-j}\in\ambset(\xi(o),\eps)}\left[\sum_{t=0}^T\sum_{x\in\X}\widetilde \xi_t^j(x)r_t^j\bigl(x,\xi_t^j,\widehat\mu_t^{-j}\bigr)+\frac{\sigma}{2}\left\|\widehat\mu^{-j}\right\|_2^2\right],
\end{align}
so that \(\Phi_{j,\entparam}^{\sigma}(\widetilde{o}^j;o)=H_{j,\sigma}(\widetilde \xi^j;o)+\Psi_{j,\entparam}(\widetilde{o}^j).\)
\edit{For fixed \(\tilde\xi^j\) and \(o\), the objective in \eqref{eq:nonentropic:cost} is the sum of an affine function of \(\widehat\mu^{-j}\) and the \(\sigma\)-strongly convex quadratic term. Since the feasible set is nonempty, compact, and convex, the minimizer exists and is unique.}
% The feasible set of \eqref{eq:nonentropic:cost} is independent of \(\widetilde \xi^j\), and the minimizer is unique since the objective is affine in \(\widehat\mu^{-j}\) apart from the strictly convex quadratic term 
%\bhavs{read weird, maybe we can just say that the objective is affine + strictly convex therefor unique minimizer?}.\sid{Yes, rewrote.}
%
For fixed \(o\), compactness of \(\ambset(\xi(o),\eps)\), joint continuity of the objective in \((\widetilde{\xi}^j,\widehat\mu^{-j})\), and uniqueness of the minimizer in \eqref{eq:nonentropic:cost} imply that this minimizer depends continuously on \(\widetilde \xi^j\). Therefore, Danskin's theorem \cite{ref:danskin2012theory} gives continuous differentiability and yields
\begin{align}\label{eq:proof:danskin:gradient}
\nabla_{\widetilde \xi_t^j}H_{j,\sigma}(\widetilde \xi^j;o)=a_t^j\left(\cdot,\widehat\mu_{\sigma,t}^{-j,\star}(\widetilde{o}^j;o)\right)+b_t^j(\cdot,\xi_t^j)\text{ for all }t=0,\ldots,T.
\end{align}
Fix \(o,\bar{o}\in\Occ\), and define \(\Delta_j(\widetilde{o}^j)\Let\Phi_{j,\entparam}^{\sigma}(\widetilde{o}^j;\bar{o})-\Phi_{j,\entparam}^{\sigma}(\widetilde{o}^j;o).\)
The entropy terms cancel in \(\Delta_j\), and hence, by \eqref{eq:proof:danskin:gradient},
\begin{align}\label{eq:proof:delta:gradient}
\nabla_{\widetilde \xi_t^j}\Delta_j(\widetilde{o}^j)=\,&a_t^j\left(\cdot,\widehat\mu_{\sigma,t}^{-j,\star}(\widetilde{o}^j;\bar{o})\right)-a_t^j\left(\cdot,\widehat\mu_{\sigma,t}^{-j,\star}(\widetilde{o}^j;o)\right)+b_t^j(\cdot,\bar \xi_t^j)-b_t^j(\cdot,\xi_t^j).
\end{align}
It follows from Assumption~\ref{assum:eps:0:1} and the definition of \(\ell_{j,t}^a\) that
\begin{align}\label{eq:proof:affine:reward:lipschitz}
\left\|a_t^j(\cdot,\nu^{-j})-a_t^j(\cdot,\bar\nu^{-j})\right\|_\infty\le\ell_{j,t}^a\max_{k\neq j}\left\|\nu^k-\bar\nu^k\right\|_1,
\end{align}
for every \(\nu^{-j},\bar\nu^{-j}\in\P(\X)^{M-1}\). Combining \eqref{eq:proof:inner:optimizer:lipschitz}, \eqref{eq:proof:affine:reward:lipschitz}, and the identity \(\left\|\mu-\bar{\mu}\right\|_1=2\dtv{\mu,\bar{\mu}}\),  we obtain %\bhavs{the way you define total variation as max over k and t is different from how we define it in the main text - not an issue, we should just be clear (in the main text we sum over k and t)} \sid{Right, it is not using some different notion, its the same.}
\begin{align}\label{eq:proof:opponent:reward:gradient:bound}
\left\|a_t^j\left(\cdot,\widehat\mu_{\sigma,t}^{-j,\star}(\widetilde{o}^j;\bar{o})\right)-a_t^j\left(\cdot,\widehat\mu_{\sigma,t}^{-j,\star}(\widetilde{o}^j;o)\right)\right\|_\infty\le2\ell_{j,t}^a\Lambda_{j,\eps}\left\|o-\bar{o}\right\|_{\onorm,\infty}.
\end{align}
Moreover, Assumption~\ref{assum:eps:one:case} gives
\begin{align}\label{eq:proof:self:reward:gradient:bound}
\left\|b_t^j(\cdot,\bar \xi_t^j)-b_t^j(\cdot,\xi_t^j)\right\|_\infty\le\ell_{j,t}\left\|\bar \xi_t^j-\xi_t^j\right\|_1\le\ell_{j,t}\left\|o-\bar{o}\right\|_{\onorm,\infty}.
\end{align}
Every occupation measure in \(\Occ\) has the prescribed initial distributions, and every element of the ambiguity set has the same prescribed opponent initial distributions. Therefore,
\begin{align*}
\xi_0^j=\bar \xi_0^j=\mu_0^j,\,\,\text{and}\,\,\widehat\mu_{\sigma,0}^{k,\star}(\widetilde{o}^j;o)=\widehat\mu_{\sigma,0}^{k,\star}(\widetilde{o}^j;\bar{o})=\mu_0^k\text{ for all }k\neq j,
\end{align*}
and hence \(\nabla_{\widetilde \xi_0^j}\Delta_j(\widetilde{o}^j)=0\). It follows from \eqref{eq:proof:delta:gradient}--\eqref{eq:proof:self:reward:gradient:bound} that, for \(t=1,\ldots,T\),
\begin{align}\label{eq:proof:delta:gradient:bound}
\left\|\nabla_{\widetilde \xi_t^j}\Delta_j(\widetilde{o}^j)\right\|_\infty\le\left(\ell_{j,t}+2\ell_{j,t}^a\Lambda_{j,\eps}\right)\left\|o-\bar{o}\right\|_{\onorm,\infty}.
\end{align}
Let \(\widetilde{o}^j,\underline{\widetilde{o}}^j\in\occset^j\), and let \(\widetilde{o}_s^j\Let(1-s)\widetilde{o}^j+s\underline{\widetilde{o}}^j\) for \(s\in\lcrc{0}{1}.\)
Since \(\Occ^j\) is convex, \(\widetilde{o}_s^j\in\Occ^j\) for every \(s\in\lcrc{0}{1}\). Using the fundamental theorem of calculus, H\"older's inequality, and \eqref{eq:proof:delta:gradient:bound}, we have
\begin{align*}
\left|\Delta_j(\widetilde{o}^j)-\Delta_j(\underline{\widetilde{o}}^j)\right|&\le\sum_{t=1}^T\left\|\widetilde \xi_t^j-\underline{\widetilde \xi}_t^j\right\|_1\sup_{s\in\lcrc{0}{1}}\left\|\nabla_{\widetilde \xi_t^j}\Delta_j(\widetilde{o}_s^j)\right\|_\infty\\ &\le\left(\sum_{t=1}^T\ell_{j,t}+2\Lambda_{j,\eps}\sum_{t=1}^T\ell_{j,t}^a\right)\left\|o-\bar{o}\right\|_{\onorm,\infty}\left\|\widetilde{o}^j-\underline{\widetilde{o}}^j\right\|_{\onorm}\\ &=\left(L_j+2L_j^a\Lambda_{j,\eps}\right)\left\|o-\bar{o}\right\|_{\onorm,\infty}\left\|\widetilde{o}^j-\underline{\widetilde{o}}^j\right\|_{\onorm}.
\end{align*}
Thus, the Lipschitz constant is
\begin{align}\label{eq:proof:delta:lipschitz}
\operatorname{Lip}(\Delta_j)\le\left(L_j+2L_j^a\Lambda_{j,\eps}\right)\left\|o-\bar{o}\right\|_{\onorm,\infty}.
\end{align}
For fixed \(o\in\occset\) and \(\widehat\mu^{-j}\in\ambset(\xi(o),\eps)\), the mapping
\begin{align*}
\widetilde{o}^j\longmapsto\sum_{t=0}^T\sum_{x\in\X}\widetilde \xi_t^j(x)r_t^j\bigl(x,\xi_t^j,\widehat\mu_t^{-j}\bigr)+\frac{\sigma}{2}\left\|\widehat\mu^{-j}\right\|_2^2
\end{align*}
is affine. Therefore, \(H_{j,\sigma}(\cdot;o)\), being the pointwise infimum of affine functions, is concave on \(\Occ^j\). By Lemma~\ref{lem:strong:concavity:entreg}, \(\Psi_{j,\entparam}\) is \(\entparam\)-strongly concave on \(\Occ^j\) with respect to \(\left\|\cdot\right\|_{\onorm}\).
Hence, \(\Phi_{j,\entparam}^{\sigma}(\cdot;o)\) is \(\entparam\)-strongly concave on \(\Occ^j\). Since \(\Occ^j\) is compact and \(\Phi_{j,\entparam}^{\sigma}(\cdot;o)\) is continuous, \(\bestresp_{j,\entparam}^{\sigma}(o)\) is nonempty, and strong concavity implies that it is a singleton.

Let \(z^j\Let\bestresp_{j,\entparam}^{\sigma}(o),\) and \(\bar{z}^j\Let\bestresp_{j,\entparam}^{\sigma}(\bar{o}).\)
The quadratic-growth consequence of strong concavity at a maximizer gives
\begin{align*}
\Phi_{j,\entparam}^{\sigma}(z^j;o)-\Phi_{j,\entparam}^{\sigma}(\bar z^j;o)\ge\frac{\entparam}{2}\left\|z^j-\bar z^j\right\|_{\onorm}^2,
\end{align*}
and
\begin{align*}
\Phi_{j,\entparam}^{\sigma}(\bar z^j;\bar{o})-\Phi_{j,\entparam}^{\sigma}(z^j;\bar{o})\ge\frac{\entparam}{2}\left\|z^j-\bar z^j\right\|_{\onorm}^2.
\end{align*}
Adding the preceding inequalities and recalling the definition of \(\Delta_j\), we obtain
\begin{align*}
\entparam\left\|z^j-\bar z^j\right\|_{\onorm}^2&\le\Phi_{j,\entparam}^{\sigma}(z^j;o)-\Phi_{j,\entparam}^{\sigma}(\bar z^j;o)+\Phi_{j,\entparam}^{\sigma}(\bar z^j;\bar{o})-\Phi_{j,\entparam}^{\sigma}(z^j;\bar{o})\\ &=\Delta_j(\bar z^j)-\Delta_j(z^j)\le\left(L_j+2L_j^a\Lambda_{j,\eps}\right)\left\|o-\bar{o}\right\|_{\onorm,\infty}\left\|z^j-\bar z^j\right\|_{\onorm},
\end{align*}
where the last inequality follows from \eqref{eq:proof:delta:lipschitz}. If \(z^j=\bar z^j\), then \eqref{eq:inner:regularized:BR:lipschitz:appen} holds trivially. Otherwise, dividing the preceding inequality by \(\entparam\left\|z^j-\bar z^j\right\|_{\onorm}\) yields
\begin{align*}
\left\|\bestresp_{j,\entparam}^{\sigma}(o)-\bestresp_{j,\entparam}^{\sigma}(\bar{o})\right\|_{\onorm}\le\frac{L_j+2L_j^a\Lambda_{j,\eps}}{\entparam}\left\|o-\bar{o}\right\|_{\onorm,\infty},
\end{align*}
which proves \eqref{eq:inner:regularized:BR:lipschitz:appen}. Finally, defining \(q_{\entparam,\eps}\Let\max_{j\in[M]}(L_j+2L_j^a\Lambda_{j,\eps})/\entparam\) and taking the maximum over \(j\in[M]\), we obtain
\begin{align}\label{eq:proof:BR:uniform:lipschitz}
&\left\|\bestresp_{\entparam}^{\sigma}(o)-\bestresp_{\entparam}^{\sigma}(\bar{o})\right\|_{\onorm,\infty}=\max_{j\in[M]}\left\|\bestresp_{j,\entparam}^{\sigma}(o)-\bestresp_{j,\entparam}^{\sigma}(\bar{o})\right\|_{\onorm}\\ &\le\max_{j\in[M]}\frac{L_j+2L_j^a\Lambda_{j,\eps}}{\entparam}\left\|o-\bar{o}\right\|_{\onorm,\infty}=q_{\entparam,\eps}\left\|o-\bar{o}\right\|_{\onorm,\infty}.
\end{align}
If \(\entparam>\max_{j\in[M]}\left(L_j+2L_j^a\Lambda_{j,\eps}\right),\) then \(q_{\entparam,\eps}<1\), and hence \(\bestresp_{\entparam}^{\sigma}:\Occ\to\Occ\) is a contraction with respect to \(\left\|\cdot\right\|_{\onorm,\infty}\), with the same contraction factor for every \(\sigma>0\).

\noindent\proofstep{step:four:cont}{Passing to \(\sigma=0\)}
Let \(D\Let(M-1)(T+1)\). Since every opponent state marginal is a probability vector, \(\left\|\widehat\mu^{-j}\right\|_2^2\le D\). Consequently, for every \(j\), \(o\in\Occ\), \(\widetilde{o}^j\in\Occ^j\), and \(\sigma>0\),
\begin{align}\label{eq:proof:payoff:uniform:approximation}
0\le\Phi_{j,\entparam}^{\sigma}(\widetilde{o}^j;o)-\Phi_{j,\entparam}(\widetilde{o}^j;o)\le\frac{\sigma D}{2}.
\end{align}
Set \(z_\sigma^j\Let\bestresp_{j,\entparam}^{\sigma}(o)\) and \(z_0^j\Let\bestresp_{j,\entparam}(o)\). By Lemma~\ref{lem:strong:concavity:entreg}, the original best response is unique.
Strong concavity of the original payoff, optimality of \(z_\sigma^j\) for the perturbed payoff, and \eqref{eq:proof:payoff:uniform:approximation} imply %\bhavs{how is $\Phi_{j,\entparam}(z_0^j;o) \leq \Phi_{j,\entparam}^{\sigma}(z_\sigma^j;o)$?}\sid{Ah, so there is term in between, i.e., \(\Phi_{j,\beta}(z_0^j;o) \le \Phi^\sigma_{j,\beta}(z_0^j;o) \le \Phi^\sigma_{j,\beta}(z_\sigma^j;o).\) Does this make sense?}
\begin{align*}
\frac{\entparam}{2}\left\|z_\sigma^j-z_0^j\right\|_{\onorm}^2&\le\Phi_{j,\entparam}(z_0^j;o)-\Phi_{j,\entparam}(z_\sigma^j;o) \le\Phi_{j,\entparam}^{\sigma}(z_\sigma^j;o)-\Phi_{j,\entparam}(z_\sigma^j;o)\le\frac{\sigma D}{2}.
\end{align*}
Thus,
\begin{align}\label{eq:proof:BR:uniform:approximation}
\sup_{o\in\Occ}\left\|\bestresp_{\entparam}^{\sigma}(o)-\bestresp_{\entparam}(o)\right\|_{\onorm,\infty}\le\delta_\sigma\Let\sqrt{\frac{\sigma D}{\entparam}}\lra0\text{ as }\sigma\to0^+.
\end{align}
Combining this with \eqref{eq:proof:BR:uniform:lipschitz} gives
\begin{align*}
\left\|\bestresp_{\entparam}(o)-\bestresp_{\entparam}(\bar{o})\right\|_{\onorm,\infty}\le2\delta_\sigma+q_{\entparam,\eps}\left\|o-\bar{o}\right\|_{\onorm,\infty}.
\end{align*}
Letting \(\sigma\to0^+\) shows that \(\bestresp_{\entparam}\) satisfies the same Lipschitz estimate.

\noindent\proofstep{step:five:cont}{Fixed points and the equilibrium mean-field flow}
Suppose now that \(\entparam>\max_{j\in[M]}(L_j+2L_j^a\Lambda_{j,\eps})\), so \(q_{\entparam,\eps}<1\). The set \(\Occ\) is nonempty and compact, hence complete under \(\left\|\cdot\right\|_{\onorm,\infty}\).
Banach's fixed-point theorem gives unique fixed points \(o_{\entparam}^{\sigma,\star}\) and \(o_{\entparam}^{\star}\) of \(\bestresp_{\entparam}^{\sigma}\) and \(\bestresp_{\entparam}\), respectively. Moreover, \eqref{eq:proof:BR:uniform:approximation} and the contraction estimate yield %\bhavs{wait LHS and RHS are the same terms? shouldn't RHS just be $\delta$ based on 73?}\sid{No its correct right? I added another step, check.}
\begin{align*}
\left\|o_{\entparam}^{\sigma,\star}-o_{\entparam}^{\star}\right\|_{\onorm,\infty} & = \left\| \bestresp_{\entparam}^{\sigma} \bigl(o_{\entparam}^{\sigma,\star}\bigr) - \bestresp_{\entparam}\bigl(o_{\entparam}^{\star}\bigr) \right\|  \nn \\& \le 
\left\| \bestresp_{\entparam}^{\sigma} \bigl(o_{\entparam}^{\sigma,\star}\bigr) - \bestresp_{\entparam}\bigl(o_{\entparam}^{\sigma,\star}\bigr) \right\| + \left\| \bestresp_{\entparam}\bigl(o_{\entparam}^{\sigma,\star}\bigr) - \bestresp_{\entparam}\bigl(o_{\entparam}^{\star}\bigr) \right\| \nn \\&
\le\delta_\sigma+q_{\entparam,\eps}\left\|o_{\entparam}^{\sigma,\star}-o_{\entparam}^{\star}\right\|_{\onorm,\infty},
\end{align*}
and therefore
\begin{align*}
\left\|o_{\entparam}^{\sigma,\star}-o_{\entparam}^{\star}\right\|_{\onorm,\infty}\le\frac{1}{1-q_{\entparam,\eps}}\sqrt{\frac{\sigma D}{\entparam}}\lra0\text{ as }\sigma\to0^+.
\end{align*}
The corresponding mean-field flows converge as well, since \(\xi\) is linear and continuous.

By the policy--occupation-measure equivalence and the uniqueness of optimal occupations from Lemma~\ref{lem:strong:concavity:entreg}, the composite operator \(\Gamma_{\entparam}\) is single-valued on feasible flows and satisfies
\begin{align*}
\xi\bigl(\bestresp_{\entparam}(o)\bigr)=\Gamma_{\entparam}\bigl(\xi(o)\bigr)\text{ for every }o\in\Occ.
\end{align*}
For any feasible initial flow \(\mu_0\), choose \(o^0\in\Occ\) with \(\xi(o^0)=\mu_0\), and set \(o^{n+1}=\bestresp_{\entparam}(o^n)\). Banach's theorem gives \(o^n\to o_{\entparam}^{\star}\); the preceding identity then shows that \(\mu^n\Let\xi(o^n)\) satisfies \(\mu^{n+1}=\Gamma_{\entparam}(\mu^n)\) and converges to \(\mu_{\entparam}^{\star}\Let\xi(o_{\entparam}^{\star})\).

Reconstructing a policy from \(o_{\entparam}^{\star}\), feasibility gives mean-field consistency and the fixed-point condition gives entropy-regularized risk-averse optimality. Thus the reconstructed policy and \(\mu_{\entparam}^{\star}\) satisfy Definition \ref{def:entropic:RA:MP:MFE}.
Conversely, every equilibrium in Definition \ref{def:entropic:RA:MP:MFE} induces a fixed point of \(\bestresp_{\entparam}\). Uniqueness of \(o_{\entparam}^{\star}\) therefore gives uniqueness of the equilibrium occupation profile and mean-field flow.
\end{proof}

\begin{remark}[Interpretation of the parameters and sensitivity constants]\label{rem:params}
The parameters \(\eps\), \(\entparam\), and \(\sigma\) serve different purposes. Recall the expression of the ambiguity set \(\ambset(\mu,\eps)\) in \eqref{eq:ambiguity:set}. The ambiguity radius \(\eps\) specifies the admissible deviations of the opponents' flows. For a fixed policy and nominal flow, increasing \(\eps\) enlarges the ambiguity set and can only decrease the worst-case reward. Thus, \(\eps\) controls the conservatism of the decision criterion. 
The regularization parameter \(\entparam\) in Definition \ref{def:entropic:RA:payoff}, weights the entropy reward, encouraging \emph{randomized policies} and providing the strong concavity used in the contraction analysis. As emphasized before entropy regularization models bounded rationality in mean-field games \cite{mazumdar2025tractable}. 

The constants \(L_j\) and \(L_j^a\) in Assumptions \ref{assum:eps:one:case} and \ref{assum:eps:0:1} respectively, quantify the sensitivity of population \(j\)'s reward to its own population distribution and the opponents' distributions, respectively. The constant \(\Lambda_{j,\eps}\) bounds the sensitivity of the inner minimizing flow to changes in the nominal profile, with the candidate occupation measure held fixed. It depends on the geometry of the dynamically feasible ambiguity sets and is independent of \(\sigma\). These are problem-dependent bounds. Under its stated assumptions, Theorem \ref{thm:contractivity:inner:regularized:BR} certifies contraction when
\begin{align*}
q_{\entparam,\eps} \Let \frac{\max_{j\in[M]}(L_j+2L_j^a\Lambda_{j,\eps})}{\entparam}<1.
\end{align*}
Increasing \(\entparam\) therefore improves the certified contraction factor, but increases the entropy perturbation of the original game. The displayed condition is sufficient and may be \emph{conservative}. When \(\eps=1\), the ambiguity set is independent of the nominal opponents' flows, and Theorem \ref{thm:contractivity:regularized:BR} gives the simpler factor \(\max_j L_j/\entparam\).

The parameter \(\sigma\) in \eqref{eq:inner:reg} makes the inner problem strongly convex, ensuring a unique minimizing flow and enabling differentiation of its optimal value. 
Positive \(\sigma\) also perturbs the payoff; taking \(\sigma\) small reduces this perturbation. \
The contraction estimate is uniform in \(\sigma>0\) and remains valid at \(\sigma=0\); the theorem therefore imposes no requirement to increase \(\entparam\) when \(\sigma\) increases.
\end{remark}

% ========================================
% ===== Thrm: Fic. Play Convergence ======
% ========================================
\section{Fictitious Play and Proof of Theorem \ref{thm:fp_FH}}\label{sec:fic-play:conv:appen}

We provide some additional details about the risk-averse fictitious play algorithm introduced in Section \ref{sec:fic:play}.
Recall from Section \ref{sec:fic:play} that the FP updates iteratively compute a best response to the population distribution induced by the average of the best responses from previous iterations.
Specifically, for each population \(j\in[M]\), each time step \(t=0,\ldots,T\), and iteration index \(\ell\ge1\), we have~\cite{perrin2020fictitiousplaymeanfield}:
\begin{equation}\label{eq:discrete-bar-mu}
\bar\mu_t^{j,\ell}=\frac{\ell-1}{\ell}\bar\mu_t^{j,\ell-1}+\frac{1}{\ell}\mu_t^{\pi^{j,\ell}}\newedit{.}
\end{equation}
The policy generating this average distribution is, for \(t=0,\ldots,T-1\), given by~\cite{perrin2020fictitiousplaymeanfield}:
\begin{equation}\label{eq:discrete-bar-pi}
\bar\pi_t^{j,\ell}(u|x)=\frac{\sum_{i=1}^{\ell}\mu_t^{\pi^{j,i}}(x)\pi_t^{j,i}(u|x)}{\sum_{i=1}^{\ell}\mu_t^{\pi^{j,i}}(x)},
\end{equation}
where $\bar\pi_t^{j,\ell}(u|x)$ is the weighted average of the best responses of population $j$ from iteration $1$ up to iteration $\ell$. The ratio is defined whenever the denominator is positive; otherwise, $\bar\pi_t^{j,\ell}(\cdot|x)$ may be chosen arbitrarily.

Fix an arbitrary population $j\in[M]$. At the starting point, $0\le s<1$, let us fix an initial policy for population $j$ as $\bar\pi^{j,s<1}=\aset[\big]{\bar\pi_t^{j,s<1}}_t=\aset[\big]{\pi_t^{j,s<1}}_t$ with induced distribution $\bar\mu^{j,s<1}=\mu^{j,s<1}=\mu^{\pi^{j,s<1}}=\aset[\big]{\mu_t^{\pi^{j,s<1}}}_t$~\cite{perrin2020fictitiousplaymeanfield}.
Moving forward, we drop the superscript $j$ for brevity since the analysis remains consistent across all populations $j\in[M]$.
The continuous-time fictitious play process is defined for all $s\ge1$ and $t\in\{0,\ldots,T\}$. For each fictitious play time \(s\ge1\), let \(\pi^{\BR,s}\) denote the $j^{\textrm{th}}$ component of a measurably selected risk-averse best response, i.e., \(\pi^{\BR,s}\in\briskopt^j(\bar\mu^s)\), where \(\bar\mu^s\) in this argument denotes the joint population flow.
Let \(\mu^{\BR,s}=\bprop^j(\pi^{\BR,s})\) denote the corresponding induced mean-field flow. The averaged flow satisfies
\begin{align}\label{eq:mean-field-derivative}
\frac{\dd}{\dd s}\bar\mu_t^s(x)=\frac{1}{s}\left(\mu_t^{\BR,s}(x)-\bar\mu_t^s(x)\right)\,\,\textrm{\newedit{equivalently},}\,\,\bar\mu_t^s(x)=\frac{1}{s}\int_0^s\mu_t^{\BR,z}(x)\dd z,
\end{align}
where the differential identity holds for almost every $s\ge1$, while the integral identity holds for every $s\ge1$. Here $\mu_t^{\BR,s}$ is induced by a best response policy $\{\pi_t^{\BR,s}\}_t$ computed against the joint averaged population flow.
From~\cite{perrin2020fictitiousplaymeanfield}, it also follows that
\begin{align}\label{eq:policy-derivative}
&\text{for }t=0,\ldots,T-1,\;\bar\mu_t^s(x)\frac{\dd}{\dd s}\bar\pi_t^s(u|x)=\frac{1}{s}\mu_t^{\BR,s}(x)\left[\pi_t^{\BR,s}(u|x)-\bar\pi_t^s(u|x)\right]\\
&\text{equivalently,}\;\bar\pi_t^s(u|x)\int_0^s\mu_t^{\BR,z}(x)\dd z=\int_0^s\mu_t^{\BR,z}(x)\pi_t^{\BR,z}(u|x)\dd z\nn.
\end{align}
For $0\le z<1$, $\pi^{\BR,z}$ and $\mu^{\BR,z}$ are taken to be the fixed initial policy and its induced flow, respectively.

We now restate and provide a detailed proof of Theorem \ref{thm:fp_FH}. Recall the expression of the risk-averse exploitability in \eqref{eq:population:RA:exploitability} and \eqref{eq:total:RA:exploitability}.

\begin{myOCP}
\begin{theorem}\label{thm:fp_FH:appen}
Let the RA-MP-MFG satisfy Assumptions~\ref{assmpt:lipschitiz-model} and \ref{assum:eps:one:monotone:reward}, and let \(\eps=1\) in \eqref{eq:amb:set:feas}. Consider the continuous-time fictitious play process \(\aset[]{\bar\pi^s,\bar\mu^s}_{s\ge1}\) defined in \eqref{eq:mean-field-derivative} and \eqref{eq:policy-derivative}. Then, for every \(j\in[M]\), the map \(s\mapsto\phi_{\RA}^j(\bar\pi^s)\) is locally absolutely continuous and satisfies
\begin{align*}
\frac{\dd}{\dd s}\phi_{\RA}^j(\bar\pi^s)\le-\frac{1}{s}\phi_{\RA}^j(\bar\pi^s)\quad\text{for almost every }s\ge1.
\end{align*}
Consequently, the exploitability is a \emph{strong Lyapunov function} along these trajectories and \(\frac{\dd}{\dd s}\phi_{\RA}(\bar\pi^s)\le-\frac{1}{s}\phi_{\RA}(\bar\pi^s)\) for almost every \(s\ge1\) and \(\phi_{\RA}(\bar\pi^s)=\mathcal{O}(1/s)\).
\end{theorem}
\end{myOCP}

\begin{proof}
Fix \(j\in[M]\). Write \(\pi^s \Let \bar\pi^{j,s}\) and \(\mu^s \Let \bar\mu^{j,s}\), and suppress the population index in \(\pi^{\BR,s}\), \(\mu^{\BR,s}\), \(r_t\), \(a_t\), and \(b_t\). The symbols \(\bar\pi^s\) and \(\bar\mu^s\) continue to denote the joint profiles. From \eqref{eq:mean-field-derivative}, we have
\begin{align}\label{eq:d-pi-tmp}
\frac{\dd}{\dd s}\mu_t^s(x)=\frac{1}{s}\left[\mu_t^{\BR,s}(x)-\mu_t^s(x)\right]\quad\text{for almost every }s\ge1.
\end{align}
Since \(\eps=1\), the ambiguity set is independent of the nominal flow, and hence \(\ambset(\bar\mu^s,1)=\ambset(\bar\mu^1,1)\) for every \(s \ge 1.\)
For each \(s\ge1\), let \(\pi^{\BR,s}\in\briskopt^j(\bar\mu^s)\) be the selected risk-averse best response, with induced flow \(\mu^{\BR,s}=\bprop^j(\pi^{\BR,s})\), and choose
%%%%
\begin{align*}
\nu^{\BR,s}&\in \argmin_{\widehat\mu^{-j}\in\ambset(\bar\mu^1,1)}\sum_{t=0}^{T}\sum_{x\in\X}\mu_t^{\BR,s}(x)r_t\bigl(x,\mu_t^s,\widehat\mu_t^{-j}\bigr).
\end{align*}
Similarly, choose a worst-case trajectory for the current policy
\begin{align*}
\nu^s&\in \argmin_{\widehat\mu^{-j}\in\ambset(\bar\mu^1,1)}\sum_{t=0}^{T}\sum_{x\in\X}\mu_t^s(x)r_t\bigl(x,\mu_t^s,\widehat\mu_t^{-j}\bigr).
\end{align*}
These optimizers exist by compactness of the feasible sets and continuity under Assumption~\ref{assmpt:lipschitiz-model}. Consequently,
\begin{align*}
\phi_{\RA}^j(\bar\pi^s)&=\max_{\tilde\pi\in\Pi^j}\bar\objective_j(\tilde\pi;\bar\mu^s)-\bar\objective_j(\pi^s;\bar\mu^s) \nn \\&= \sum_{t=0}^{T}\sum_{x\in\X}\left[\mu_t^{\BR,s}(x)r_t\bigl(x,\mu_t^s,\nu_t^{\BR,s}\bigr)-\mu_t^s(x)r_t\bigl(x,\mu_t^s,\nu_t^s\bigr)\right].
\end{align*}

We first justify the \emph{differentiation} of these optimized values. We record them via the following steps:
\begin{enumerate}[label=\textup{(Diff-\alph*)}, leftmargin=*, widest=ii, align=left]

\item \label{diff:1} The map \(z\mapsto\mu_t^{\BR,z}\) is measurable and bounded, since each \(\mu_t^{\BR,z}\) is a probability distribution. Therefore, its indefinite integral is locally absolutely continuous \cite[Chapter 3, Theorem 3.35]{ref:FolReal-99}.
Multiplication by the continuously differentiable function \(s\mapsto1/s\) preserves absolute continuity on compact intervals in \([1,+\infty)\). Thus, the integral representation in \eqref{eq:mean-field-derivative} implies that \(\mu^s\) is locally absolutely continuous.

\item \label{diff:2} Moreover, \eqref{eq:d-pi-tmp} gives
\begin{align*}
\left\|\frac{\dd}{\dd s}\mu_t^s\right\|_1=\frac{1}{s}\left\|\mu_t^{\BR,s}-\mu_t^s\right\|_1\le\frac{2}{s}\quad\text{for almost every }s\ge1.
\end{align*}
Using the fundamental theorem of calculus for absolutely continuous functions \cite[Chapter 3]{ref:FolReal-99}, we consequently obtain, for \(s\ge r\ge s_0\ge1\),
\begin{equation}\label{eq:fp:flow:lipschitz}
\left\|\mu_t^s-\mu_t^r\right\|_1\le\int_r^s\left\|\frac{\dd}{\dd z}\mu_t^z\right\|_1\dd z\le\frac{2}{s_0}|s-r|.
\end{equation}
Hence, \(\mu^s\) is locally Lipschitz.

\item \label{diff:3} Next, let \(R_j<+\infty\) be a uniform bound on \(|r_t(\cdot)|\), which exists by Assumption~\ref{assmpt:lipschitiz-model} and compactness, and recall that \(L_r^j\) was the Lipschitz constant in Assumption \ref{assmpt:lipschitiz-model}.
Since \(\eps=1\), the ambiguity set is fixed, and the nominal population flow enters the payoff only through \(\mu^s\). 
For a fixed candidate policy \(\tilde\pi\), its induced flow is independent of the nominal flow. The reward Lipschitz bound therefore applies uniformly over all candidate policies and feasible opponent flows. 
Taking an infimum (or supremum) over a fixed set preserves such a uniform bound, and consequently, 
\begin{equation}\label{eq:fp:values:lipschitz}
\begin{aligned}
\left|\max_{\tilde\pi\in\Pi^j}\bar\objective_j(\tilde\pi;\bar\mu^s)-\max_{\tilde\pi\in\Pi^j}\bar\objective_j(\tilde\pi;\bar\mu^r)\right|&\le L_r^j\sum_{t=0}^{T}\left\|\mu_t^s-\mu_t^r\right\|_1,\\
\left|\bar\objective_j(\pi^s;\bar\mu^s)-\bar\objective_j(\pi^r;\bar\mu^r)\right|&\le(R_j+L_r^j)\sum_{t=0}^{T}\left\|\mu_t^s-\mu_t^r\right\|_1.
\end{aligned}
\end{equation}
For the second inequality, we use \(\mu^s=\bprop^j(\pi^s)\) and compare the trajectory objectives at a fixed feasible opponent flow. Changing the state probabilities contributes at most \(R_j\sum_{t=0}^{T}\|\mu_t^s-\mu_t^r\|_1\), while changing the reward arguments contributes at most \(L_r^j\sum_{t=0}^{T}\|\mu_t^s-\mu_t^r\|_1\).
\end{enumerate}

Combining \eqref{eq:fp:flow:lipschitz} and \eqref{eq:fp:values:lipschitz} shows that both optimized values are locally Lipschitz in \(s\). 
Their difference, \(\phi_{\RA}^j(\bar\pi^s)\), is also locally Lipschitz by the triangle inequality.
Thus, both optimized values and their difference are locally absolutely continuous, since every Lipschitz function on a compact interval is absolutely continuous \cite[Section 6.4, Proposition 7]{ref:royden1988real}.

Fix a time \(s>1\) at which \eqref{eq:d-pi-tmp} holds and the flow and both optimized values are differentiable. Keeping \(\pi^{\BR,s}\) fixed and using \(r_t(x,\mu,\nu)=a_t(x,\nu)+b_t(x,\mu)\), we obtain, for sufficiently small \(h\) of either sign,
\begin{align*}
&\max_{\tilde\pi\in\Pi^j}\bar\objective_j(\tilde\pi;\bar\mu^{s+h})-\max_{\tilde\pi\in\Pi^j}\bar\objective_j(\tilde\pi;\bar\mu^s) \ge \bar\objective_j(\pi^{\BR,s};\bar\mu^{s+h})-\bar\objective_j(\pi^{\BR,s};\bar\mu^s)\\
&\qquad=\sum_{t=0}^{T}\sum_{x\in\X}\mu_t^{\BR,s}(x)\left[b_t(x,\mu_t^{s+h})-b_t(x,\mu_t^s)\right].
\end{align*}
Similarly, keeping \(\nu^s\) fixed gives 
\begin{align*}
&\bar\objective_j(\pi^{s+h};\bar\mu^{s+h})-\bar\objective_j(\pi^s;\bar\mu^s)  \nn \\& \qquad\quad\le \sum_{t=0}^{T}\sum_{x\in\X}\left[\mu_t^{s+h}(x)r_t(x,\mu_t^{s+h},\nu_t^s)-\mu_t^s(x)r_t(x,\mu_t^s,\nu_t^s)\right].
\end{align*}

Dividing these inequalities by \(h\) and taking limits from both sides justifies differentiating the optimized values while holding the selected optimizers fixed. Using \(\nabla_\mu r_t(x,\mu,\nu)=\nabla_\mu b_t(x,\mu)\) and \eqref{eq:d-pi-tmp}, we therefore obtain, for almost every \(s\ge1\),
\begin{align*}
&\frac{\dd}{\dd s}\phi_{\RA}^j(\bar\pi^s)=\sum_{t=0}^{T}\sum_{x\in\X}\Bigg[\mu_t^{\BR,s}(x)\inprod{\nabla_\mu r_t(x,\mu_t^s,\nu_t^{\BR,s})}{\frac{\dd}{\dd s}\mu_t^s}\\
&\qquad-\mu_t^s(x)\inprod{\nabla_\mu r_t(x,\mu_t^s,\nu_t^s)}{\frac{\dd}{\dd s}\mu_t^s}-r_t(x,\mu_t^s,\nu_t^s)\frac{\dd}{\dd s}\mu_t^s(x)\Bigg]\\
&=s\sum_{t=0}^{T}\sum_{x\in\X}\inprod{\nabla_\mu b_t(x,\mu_t^s)}{\frac{\dd}{\dd s}\mu_t^s}\frac{\dd}{\dd s}\mu_t^s(x) \nn \\& \qquad \quad- \frac{1}{s}\sum_{t=0}^{T}\sum_{x\in\X}r_t(x,\mu_t^s,\nu_t^s)\left[\mu_t^{\BR,s}(x)-\mu_t^s(x)\right]\\
&\stackrel{\mathclap{(\dag)}}{=}-\frac{1}{s}\phi_{\RA}^j(\bar\pi^s)+s\sum_{t=0}^{T}\sum_{x\in\X}\inprod{\nabla_\mu b_t(x,\mu_t^s)}{\frac{\dd}{\dd s}\mu_t^s}\frac{\dd}{\dd s}\mu_t^s(x) \nn \\& \qquad \quad + \frac{1}{s}\sum_{t=0}^{T}\sum_{x\in\X}\mu_t^{\BR,s}(x)\left[a_t(x,\nu_t^{\BR,s})-a_t(x,\nu_t^s)\right].
\end{align*}
In \((\dag)\), we employed the expression for the exploitability and our reward decomposition.
The second term in \((\dag)\) is nonpositive by Assumption \ref{assum:eps:one:monotone:reward}; we refer to \cite{perrin2020fictitiousplaymeanfield} for more details.
The third term in \((\dag)\) is also nonpositive by definition of the `inf' operator.
Consequently, \(\frac{\dd}{\dd s}\phi_{\RA}^j(\bar\pi^s)\le-\frac{1}{s}\phi_{\RA}^j(\bar\pi^s)\) for almost every \(s\ge1.\) By local absolute continuity, this implies \(\frac{\dd}{\dd s}\left[s\phi_{\RA}^j(\bar\pi^s)\right]\le0\) a.e., and summing over \(j\in[M]\), we obtain \(\frac{\dd}{\dd s}\phi_{\RA}(\bar\pi^s)\le-\frac{1}{s}\phi_{\RA}(\bar\pi^s)\) a.e. This proves our claim.
\end{proof}

\begin{remark}
Observe that Assumption~\ref{assum:eps:one:monotone:reward} imposes a monotonicity condition on the \emph{own-population} component \(b_t^{j}\) of the reward function and \emph{does not} restrict the reward component depending on the opponent-dependent populations \(a^j_t\).
It is therefore weaker (compared to the existing literature) than requiring monotonicity of the full multi-population reward with respect to the entire mean-field profile.
\end{remark}

% ========================================
% ======== Numerical details =============
% ========================================

\section{Numerical Experiments: Additional Details}\label{sec:add:num:det:appen}

In this section, we provide additional details about the numerical examples presented in the main text (Section~\ref{sec:numexp}).
In particular, we discuss in detail the dual formulation of the risk-averse objective that simplifies computation by enabling the use of off-the-shelf optimization solvers and the hyperparameters used for the experiments. 
Note that all computations are performed on the CPU, in particular, macOS 14.1.
Further, the pseudocodes of the proposed risk-averse entropy-regularized fixed point iteration and risk-averse fictitious play algorithms are presented in Algorithms~\ref{alg:fpi-risk-multi-population} and~\ref{alg:fic-play-risk-multi-population} respectively.
Note that we use Polyak averaging for the policy update in Algorithm~\ref{alg:fpi-risk-multi-population} to ensure smooth and stable updates~\cite{ref:learning:MFG:Lau}.

\begin{algorithm2e}[ht]
\caption{Fixed Point Iteration for RA-MP-MFGs}
\label{alg:fpi-risk-multi-population}
\DontPrintSemicolon
\KwIn{MP-MFG instance, entropy-regularization parameter \(\beta\), damping factor $\tau$}
\KwOut{Policy profile \(\pi\)}

Initialize
\(\pi^{\ell=0}=\{\pi_t^{j,\ell=0}\}_{t\in\{0,\ldots,T-1\},\,j\in[M]}\)\;

\For{\(\ell=0,1,\ldots\)}{
\(\mu^{\ell}=\bprop(\pi^\ell)\)\;
   \(\pi^{\ell+1}=\briskoptbeta(\mu^\ell)\)\;
   \(\pi^{\ell+1}=\tau\pi'+(1-\tau)\pi^{\ell}\)\;
    }

\Return{\(\pi\)}\;
\end{algorithm2e}

\begin{algorithm2e}[ht]
\caption{Fictitious Play for RA-MP-MFGs}
\label{alg:fic-play-risk-multi-population}
\DontPrintSemicolon
\KwIn{MP-MFG instance}
\KwOut{Policy profile \(\bar\pi\)}

Initialize \(\pi^{\ell=0}=\{\pi_t^{j,\ell=0}\}_{t\in\{0,\ldots,T-1\},\,j\in[M]}\)\;
$\bar\pi^{\ell=0} = \pi^{\ell=0}$\;
$\bar\mu^{\ell=0} = \bprop(\pi^\ell)$\;
\For{\(\ell=0,1,\ldots\)}{
\(\pi^{\ell+1} =  \briskopt(\bar\mu^\ell)\)\;
\(\mu^{\ell+1} =  \bprop(\pi^{\ell+1})\)\;
\(\bar\pi^{\ell+1} = \textrm{avg}(\pi^{0}, \ldots,\pi^{\ell+1})\) using~\eqref{eq:discrete-bar-pi}\;
\(\bar\mu^{\ell+1} = \textrm{avg}(\mu^{0}, \ldots,\mu^{\ell+1})\) using~\eqref{eq:discrete-bar-mu}\;
}

\Return{\(\bar\pi\)}\;
\end{algorithm2e}

\subsection{Grid Navigation}

In this example, both populations traverse a $5\times 5$ grid with an action space corresponding to $\texttt{\{LEFT, RIGHT, UP, DOWN, STAY\}}$ and $T=12$.
The populations transition deterministically.

Population 1, at a given state $x^1_t\in\X$, receives $r_t^1(x_t^1,\mu_t^1,\mu_t^2) \Let \mu_t^2(x_t^1)$.
We assume that Population 1 is risk-averse to deviations in Population 2's trajectory. 
The reward \(r^1_t(\cdot)\) is linear in \(\mu^2_t\) and \(\eps=1\), i.e., the ambiguity set is independent of the nominal trajectory. 
We seek to find a policy that maximizes the risk-averse payoff (Definition~\ref{def:RA:payoff}) for a fixed joint mean-field flow $\mu$, i.e., $\argmax_{\pi^1\in\Pi^1}\bar{\objective}_1(\pi^1;\mu)$.
The ensuing optimization problem is
\begin{align}\label{eq:grid-nav-opt}
\argmax_{\pi^1\in\Pi^1}\bar{\objective}_1(\pi^1;\mu)=\argmax_{\pi^1\in\Pi^1}\inf_{\hat\mu^{2}\in\mathcal{W}_1(\mu,1)}\E^{\pi^1}\left[\sum_{t=0}^{T} \hat\mu_t^2(x_t^1)\right]
\end{align}

By exploiting duality together with the occupation-measure representation,~\eqref{eq:grid-nav-opt} can be reformulated using the dual variables $V \Let \{V_t\}_{t=0}^{T}$ such that $ V_t\in\mathbb{R}^{|\X|}$, and consequently we have
\begin{equation}\label{grid-nav-opt-duality}
\begin{aligned}
& \max_{\eta^1 \ge 0,\;V} 
&& \sum_{x} \mu^2_0(x) V_0(x)  \\ &  \sbjto && \begin{cases}
\sum_u \eta^1_0(x,u)=\mu^1_0(x), & \text{for all } x,\\[2mm]
\sum_u \eta^1_{t+1}(y,u) = \sum_{x,u}\eta^1_t(x,u)f^1_t(y\mid x,u),
& \text{for all } y,t,\\[2mm] V_T(x)\leq \sum_u \eta^1_T(x,u),
& \text{for all } x,\\[2mm] V_t(x)\leq \sum_u \eta^1_t(x,u) +\sum_y f^2_t(y\mid x,u)V_{t+1}(y), & \text{for all }x,u,t.
\end{cases}
\end{aligned}\nn
\end{equation}

We use \texttt{MOSEK}~\cite{mosek} to solve~\eqref{grid-nav-opt-duality}.
The policy-occupation-measure equivalence allows us to recover the optimal policy from $\eta^{1, *}$, the maximizer of~\eqref{grid-nav-opt-duality}. 
The optimal policy for the risk-neutral Population 2 can be computed using existing algorithms for MP-MFGs, and in our work we employ online mirror descent (OMD)~\cite{perolat2021scalingmeanfieldgames}.
In our example, for Population 2, the OMD parameter $\alpha=0.05$.
Under entropy regularization,~\eqref{grid-nav-opt-duality} can be modified to include the regularization (under parameter $\beta$) as
\begin{equation}\label{grid-nav-opt-duality-beta}
\begin{aligned}
& \max_{\eta^1 \ge 0,\;V} 
&& \sum_{x} \mu^2_0(x) V_0(x) -\beta \sum_{t=0}^{T-1}\sum_{x,u}\eta^1_t(x,u) \log\left(\frac{\eta^1_t(x,u)}
{\sum_{u'}\eta^1_t(x,u')} \right) \\ &  \sbjto && \begin{cases}
\sum_u \eta^1_0(x,u)=\mu^1_0(x), & \text{for all } x,\\[2mm]
\sum_u \eta^1_{t+1}(y,u) = \sum_{x,u}\eta^1_t(x,u)f^1_t(y\mid x,u),
& \text{for all } y,t,\\[2mm] V_T(x)\leq \sum_u \eta^1_T(x,u),
& \text{for all } x,\\[2mm] V_t(x)\leq \sum_u \eta^1_t(x,u) +\sum_y f^2_t(y\mid x,u)V_{t+1}(y), & \text{for all }x,u,t.
\end{cases}
\end{aligned}
\end{equation}
As before, we employ \texttt{MOSEK}~\cite{mosek} to solve~\eqref{grid-nav-opt-duality-beta} and we can recover the optimal policy from $\eta^{1, *}$.
For Population 2, we utilize Boltzamnn Iteration~\cite{cui2021approximately}.
Table~\ref{tab:sensitivity} computes the ER-RA-MFE with $\beta=0.1$ and a damping factor $\tau=0.2$ (see Algorithm~\ref{alg:fpi-risk-multi-population}) for both populations while Figure~\ref{fig:navigation} presents an ablation over the values of $\beta$.

\subsection{Multi-Population Chasing}
In both variants of multi-population chasing ($M=3$ and $4$) populations traverse a $5\times 5$ grid and $T=5$, similar to the Navigation example. 
The population dynamics follow ``noisy" transition dynamics similar to the original examples presented in~\cite{perolat2021scalingmeanfieldgames, LanctotEtAl2019OpenSpiel}.
We use \texttt{CLARABEL}~\cite{Clarabel_2024} as the solver for the optimization problems associated with the multi-population chasing problem.

\subsubsection{3 Population Chasing} In this example, the fully expanded reward obtained by Population 1 is given by $r_t^1(x_t^1,\mu_t^1,\mu_t^2,\mu^3_t) \Let -\log(\mu^1_t(x)) + \mu^2_t(x) + \mu^3_t(x)$, which is again, linear in $\mu^2_t$.
The corresponding dual formulation using the dual variables $V_t \in\mathbb{R}^{|\X|}$ for all $t = 0, \ldots, T$ as follows
\begin{equation}\label{3-population-chasing-opt-duality}
\begin{aligned}
& \max_{\eta^1 \ge 0,\;V} 
&& \sum_{x} \mu^2_0(x) V_0(x) +\sum_{t, x}\xi_t^1(x)\Bigl(\mu^3_t(x) -\log(\mu^1_t(x))\Bigr) \\ &  \sbjto && \begin{cases}
\sum_u \eta^1_0(x,u)=\mu^1_0(x), & \text{for all } x,\\[2mm]
\sum_u \eta^1_t(x,u)=\xi^1_t(x), & \text{for all } x, t,\\[2mm]
\sum_u \eta^1_{t+1}(y,u) = \sum_{x,u}\eta^1_t(x,u)f^1_t(y\mid x,u),
& \text{for all } y,t,\\[2mm] V_T(x)\leq \sum_u \eta^1_T(x,u),
& \text{for all } x,\\[2mm] V_t(x)\leq \sum_u \eta^1_t(x,u) +\sum_y f^2_t(y\mid x,u)V_{t+1}(y), & \text{for all }x,u,t.
\end{cases}
\end{aligned}
\end{equation}

The optimal policy for the risk-neutral Populations 2 and 3 (i.e., maximize expected reward instead of risk-averse payoff) are computed using OMD with $\alpha = 0.01$.
To obtain the entropy-regularized problem, we append~\eqref{3-population-chasing-opt-duality} with the entropy term similar to~\eqref{grid-nav-opt-duality-beta}.
We report the ER-RA-MFE value in Table~\ref{tab:sensitivity} under $\beta=1$ and a damping factor $\tau=0.2$.

\subsubsection{4 Population Chasing} In this example, the fully expanded reward obtained by Population 1 is given by $r_t^1(x_t^1,\mu_t^1,\mu_t^2,\mu^3_t, \mu^4_t) \Let -\log(\mu^1_t(x)) + \mu^2_t(x) + \frac{1}{2}\mu^3_t(x) -\mu^4_t(x)$.
Population 1 is now risk-averse to Populations 2 \emph{and} 3.
The reward $r_t^1$ is linear in $\mu^2_t$ and $\mu^3_t$ and independent of cross-coupling terms like $\mu^2_t(x)\mu^3_t(x)$.
The corresponding dual formulation using the dual variables $V^2=\{V^2_t\} \in\mathbb{R}^{|\X|}$ and $V^3=\{V^3_t\} \in\mathbb{R}^{|\X|}$ for all $t = 0, \ldots, T$ is as follows

\begin{equation}\label{4-population-chasing-opt-duality}
\begin{aligned}
& \max_{\eta^1 \ge 0,\;V} 
&& \sum_{x} \mu^2_0(x) V^2_0(x) + \mu^3_0(x) V^3_0(x) +\sum_{t, x}\xi^1_t(x)\Bigl(\mu^4_t(x) -\log(\mu^1_t(x))\Bigr)\\ &  \sbjto && \begin{cases}
\sum_u \eta^1_0(x,u)=\mu^1_0(x), & \text{for all } x,\\[2mm]
\sum_u \eta^1_t(x,u)=\xi^1_t(x), & \text{for all } x, t\\[2mm]
\sum_u \eta^1_{t+1}(y,u) = \sum_{x,u}\eta^1_t(x,u)f^1_t(y\mid x,u),
& \text{for all } y,t,\\[2mm] V^2_T(x)\leq \sum_u \eta^1_T(x,u),
& \text{for all } x,\\[2mm] V^2_t(x)\leq \sum_u \eta^1_t(x,u) +\sum_y f^2_t(y\mid x,u)V^2_{t+1}(y), & \text{for all }x,u,t,
\\[2mm] V^3_T(x)\leq \sum_u \frac{1}{2}\eta^1_T(x,u),
& \text{for all } x,\\[2mm] V^3_t(x)\leq \sum_u \frac{1}{2}\eta^1_t(x,u) +\sum_y f^3_t(y\mid x,u)V^3_{t+1}(y), & \text{for all }x,u,t.
\end{cases}
\end{aligned}
\end{equation}

The optimal policy for the risk-neutral populations are computed using OMD with $\alpha = 0.01$.
To obtain the entropy-regularized problem, we append~\eqref{4-population-chasing-opt-duality} with the entropy term similar to~\eqref{grid-nav-opt-duality-beta}.
We report the ER-RA-MFE value in Table~\ref{tab:sensitivity} under $\beta=1$ and $\tau=0.2$.

\subsection{Stock Market Trading}
The discrete inventory positions are specified as: $\X \Let \{-K, \ldots, 0, \ldots, K\}$ for some integer $K>0$ representing the holdings of a population.
The actions available to each population are $\U \Let \{-1, 0, 1\}$ representing \texttt{SELL, HOLD} and \texttt{BUY} respectively.
The inventory transitions deterministically, i.e., $x_{t+1} = \textrm{clip}_{(-K, K)}(x_t + a_t)$.
The change in the stock price at time $t$ is denoted by $\Delta p_t$ and follows  $\Delta p_t = \gamma_t + \theta_t z_t$, where $\gamma_t$ represents the price drift, $\theta_t$ represents the volatility and $z\sim \textrm{Discrete}(-1, 0, 1)$, where $\E[z_t]= 0$ and $\textrm{var}(z_t) = 1$.
The rewards for the AI population follows $r^1_t(x^1_t, \mu^1_t, \mu^2_t) = \gamma_tx^1_t + c_0x^1_t\cdot \left\langle q, \mu^2_t\right\rangle - c_1(x^1_t)^2$, where $q\t = [-K, \ldots, 0, \ldots K]$ and $c_0, c_1>0$.
The term $c_0x^1_t\cdot \left\langle q, \mu^2_t\right\rangle$ represents the price pressure created by aggregate human inventory shifts and $c_1(x_t)^2$ penalizes large AI inventory positions due to volatility and holding costs.
The rewards for the human population is given by $r^2_t(x^2_t,  \mu^1_t, \mu^2_t) = \gamma_tx^2_t -w_0(\theta_t x_t)^2+ w_1x^2_t\cdot \left\langle q, \mu^2_t\right\rangle - w_2\mathds{1}_{\gamma_t < -\chi}(x^2_t + K)$, where  $-w_0(\theta_t x_t)^2$, $w_1x^2_t\cdot \left\langle q, \mu^2_t\right\rangle$ and   $w_2\mathds{1}_{\gamma_t < -\chi}(x^2_t + K)$ represent the penalty due to volatility in the price, herding behavior and penalty due to panic selling when the stock price drift falls below $-\chi$.

We calibrate the price model (in particular, $\theta$ and $\gamma$) to mimic the dynamics of Apple Inc. (AAPL)  over September 2025 ($T=20$), using data retrieved through \texttt{yfinance}~\cite{aroussi2026yfinance} and analyze the effect of risk-aversion of the AI population to the deviations in the human population's mean-field trajectory. 
We assume $K=5$, which implies $|\X| = 11$.
Other parameters are given as follows:
\begin{align*}
     c_0 &= 0.1 &  c_1 &= 0.05 &  \chi &= 0.02 \\
     w_0 &= 0.02 &  w_1 &= 0.1 & w_2 &= 0.15
\end{align*}

The reward $r^1_t$ linear in $\mu^2_t$.
However, it does not satisfy Assumption~\ref{assum:eps:one:monotone:reward}, and consequently the convergence guarantee for risk-averse fictitious play does not apply. Instead, we employ entropy-regularized fixed-point iteration to obtain a close approximation of the RA-MFE, using a ``small'' regularization parameter \(\beta=10^{-6}\) and $\tau=0.2$.
The corresponding dual formulation using the dual variables $V_t \in\mathbb{R}^{|\X|}$ for all $t = 0, \ldots, T$ as follows
\begin{equation}\label{stock-trading-opt-duality-beta}
\begin{aligned}
& \max_{\eta^1 \ge 0,\;V} 
&& J_{\rm sm} \\ &  \sbjto && \begin{cases}
\sum_u \eta^1_t(x,u)=\xi^1_t(x), & \text{for all } x, t,\\[2mm]
\sum_u \eta^1_{t+1}(y,u) = \sum_{x,u}\eta^1_t(x,u)f^1_t(y\mid x,u),
& \text{for all } y,t,\\[2mm] V_T(x)\leq c_0x\left\langle q, \xi^1_T\right\rangle,
& \text{for all } x,\\[2mm] V_t(x)\leq  c_0x\left\langle q, \xi^1_t\right\rangle +\sum_y f^2_t(y\mid x,u)V_{t+1}(y), & \text{for all }x,t,
\end{cases}
\end{aligned}
\end{equation}
where the objective function is
\begin{align}
    J_{\rm sm} &\Let \sum_{x} \mu^2_0(x) V_0(x)+ \sum_{t}\gamma_t \left\langle q, \xi^1_t\right\rangle -\sum_{t,x}c_1\xi^1_t(x)(x)^2\  \nn  \\& -\beta \sum_{t=0}^{T-1}\sum_{x,u}\eta^1_t(x,u) \log\left(\frac{\eta^1_t(x,u)}
{\sum_{u'}\eta^1_t(x,u')} \right) \nn
\end{align}
As before, we use \texttt{MOSEK}~\cite{mosek} to solve~\eqref{stock-trading-opt-duality-beta}.

\subsection{Linear Navigation}

We consider an MP-MFG with two populations traversing a three-state environment, where agents can move \texttt{LEFT}, \texttt{RIGHT}, or remain \texttt{STATIONARY} and $T=5$.
The dynamics of the game follow Figure 1(a) in~\cite{ref:YG:AA:LH:KHJ:Dist:Reach:MDP:TAC}.
Similar to the Grid Navigation problem, Population 1, at a given state $x^1_t\in\X$, receives $r_t^1(x_t^1,\mu_t^1,\mu_t^2)=\mu_t^2(x_t^1)$, which satisfies Assumption~\ref{assum:eps:0:1}. 
We assume that Population 1 is risk-averse to deviations in Population 2's trajectory.
As mentioned in the main text, for \(\eps\in(0,1)\), \(\ambset(\mu,\eps)\) depends explicitly on the nominal mean-field trajectory. 
Consequently, the inner minimization is coupled with the nominal trajectory and the dual formulation no longer provides a computational advantage.
We therefore solve the resulting max-min problem directly using the iterative \texttt{MINIMAX-APPA} algorithm proposed in~\cite{pmlr-v125-lin20a}, which is designed to handle such coupled minimax optimization problems.
In particular, we our aim to compute a solution to
\begin{align}\label{eq:linear-nav-opt}
\argmax_{\pi^1\in\Pi^1}\bar{\objective}_1(\pi^1;\mu)&=\argmax_{\pi^1\in\Pi^1}\inf_{\hat\mu^{2}\in\mathcal{W}_1(\mu,\eps)}\E^{\pi^1}\left[\sum_{t=0}^{T} \hat\mu_t^2(x_t^1)\right] \nn \\&-\beta\E^{\pi^1}\left[\sum_{t=0}^{T} \log\pi^1_t(u^1_t|x^1_t)\right] + \frac{\sigma}{2}\|\hat\mu_t^2\|^2_2 
\end{align}
Equivalently, we can reformulate~\eqref{eq:linear-nav-opt} using the policy-occupation-measure equivalence~\eqref{eq:inner:reg} and obtain,
\begin{align*}
\argmax_{\widetilde{o}^1 = (\tilde\xi^1, \tilde\eta^1)\in\occset^1}\Phi_{1,\entparam}^{\sigma}(\widetilde{o}^1;o) &= 
\argmax_{\widetilde{o}^1\in\occset^1}\min_{\hat\mu^{2}\in\mathcal{W}_1(\mu,\eps)}
\Bigg{[} \sum_{t=0}^T\sum_{x\in\X}\widetilde \xi_t^1(x) \cdot\hat\mu_t^{2}(x) +\frac{\sigma}{2}\left\|\hat\mu^{2}\right\|_2^2 \Bigg{]}  \\& -\entparam\sum_{t=0}^{T-1}\sum_{x\in\X}\sum_{u\in\U}
{\tilde\eta}_t^1(x,u) \log\frac{{\tilde\eta}_t^1(x,u)}{{\tilde\xi}_t^1(x)}. 
\end{align*}
We apply this algorithm for $\eps =0.1$, $\beta=1$ and $\sigma = 1$.
We perform 10 outer iterations and 5 inner iterations for the minimizing problem (i.e.,~\eqref{eq:inner:reg}).
Additionally, for the outer iteration, we scale the original learning rate for the minimization problem provided in~\cite{pmlr-v125-lin20a} by 0.01 for empirically stable updates.
For the maximizing variable (Population 1's occupation measure),  we perform 100 iteration of Nesterov's accelerated gradient descent and similarly scale the learning rate by 0.05 for stability.

\subsection{Generating Adversarial Mean-Field Flows}
In all our examples with \(\eps=1\), we generate deviated mean-field flows by uniformly sampling policies from \(\Pi\) and propagating each sampled policy $\pi^k$ where $ k\neq j$ to obtain the corresponding mean-field flow over the entire simplex. 
For \(\eps\in(0,1)\), we instead construct a random mean-field trajectory and project it onto \(\ambset(\mu,\eps)\). 
Recall from Proposition~\ref{prop:ambiguity:set} that the ambiguity set is a polytope. 
Consequently, the projection is well-defined and maps the sampled trajectory to a dynamically feasible flow in \(\ambset(\mu,\eps)\). 
Since the projection is onto the boundary of the ambiguity set, the resulting flow lies at the maximum prescribed deviation from the nominal flow, i.e., \(\dtv{\mu_t^k,\hat\mu_t^k}=\eps\) for all (t).
To also sample trajectories satisfying \(\dtv{\mu_t^k,\hat\mu_t^k}<\eps\), we randomly choose a smaller radius \(\eps'\in(0,\eps)\) and project the sampled trajectory onto the corresponding constraint set with \(\dtv{\mu_t^k,\hat\mu_t^k}\leq\eps'\).

\end{document}